\documentclass[12pt]{article}

\usepackage{amsmath}
\usepackage{amsthm}
\usepackage{cite}
\usepackage{caption}

\usepackage[margin=1in]{geometry}
\usepackage{amssymb}
\usepackage[hidelinks]{hyperref}
\usepackage[capitalize]{cleveref}
\usepackage{microtype}
\usepackage{tkz-graph}
\usepackage{longtable}

\usepackage{tikz}

\newtheorem{theorem}{Theorem}

\crefname{theorem}{Theorem}{Theorems}
\crefname{thm}{Theorem}{Theorems}
\crefname{figure}{Figure}{Figures}
\crefname{lem}{Lemma}{Lemmas}

\newtheorem{thm}[theorem]{Theorem}

\newtheorem{lem}[theorem]{Lemma}

\theoremstyle{definition}

\title{Small unit-distance graphs in the plane}

\author{Aidan Globus \and Hans Parshall}

\date{}

\tikzset{every node/.style={shape=circle,fill=black,inner sep=0pt, minimum size=0.5em,anchor=mid}}

\begin{document}

\maketitle

\begin{abstract}
We prove that a graph on up to 9 vertices is a unit-distance graph if and only if it does not contain one of 74 so-called minimal forbidden graphs.  This extends the work of Chilakamarri and Mahoney (1995), who provide a similar classification for unit-distance graphs on up to 7 vertices.
\end{abstract}

\section{Introduction}

A graph $G$ on vertices $V$ is said to be {\bf unit-distance} if there exists an embedding $\varphi\colon V \rightarrow \mathbf{R}^2$ such that for every pair of adjacent vertices $v,w \in V$, we have $|\varphi(v) - \varphi(w)| = 1$; here and throughout we use $| \cdot |$ to denote the usual Euclidean norm.  Several long-standing open problems concern unit-distance graphs.  For instance, Erd\H{o}s' unit-distance problem asks for the maximum number of edges $u(n)$ over all unit-distance graphs on $n$ vertices.  The lower bound of $u(n) \geq n^{1 + \frac{c}{\log\log n}}$ for some fixed $c > 0$ has not been improved since the 1946 paper of Erd\H{o}s~\cite{erdos46}, and it is suspected to be close to the truth.  Spencer, Szemer\'{e}di and Trotter~\cite{spencer84} established the upper bound of $u(n) \leq Cn^{4/3}$ for some fixed $C > 0$.  While this has not been improved, there are now several distinct proofs of the same upper bound; see~\cite{szemeredi16} for a particularly elegant argument and a more thorough overview.  Erd\H{o}s also popularized the Hadwiger--Nelson problem, which asks for the chromatic number $\chi(\mathbf{R}^2)$ of the infinite unit-distance graph with vertex set $\mathbf{R}^2$ and an edge between $v,w \in \mathbf{R}^2$ exactly when $|v - w| = 1$.  The history of the problem is carefully documented by Soifer~\cite{soifer09}; he credits Nelson with the lower bound of $\chi(\mathbf{R}^2) \geq 4$ and Isbell with the upper bound of $\chi(\mathbf{R}^2) \leq 7$.  These bounds remained best known for over half a century, but recently de~Grey~\cite{deGrey18} proved $\chi(\mathbf{R}^2) \geq 5$ by providing a 5-chromatic unit-distance graph on 1581 vertices.  A smaller 5-chromatic unit-distance graph with only 553 vertices has been produced by Heule~\cite{heule18}.

When a graph is not unit-distance, we call it {\bf forbidden}.  Unit-distance graphs can be frustrating to study in part because it is typically difficult to determine if a given graph is unit-distance or forbidden.  For instance, it was conjectured by Chv\'{a}tal~\cite{chvatal72} that the so-called Heawood graph on 14 vertices was forbidden.  Decades later, Gerbracht~\cite{gerbracht09} refuted this conjecture by providing several unit-distance embeddings of the Heawood graph.  More generally, Schaefer~\cite{schaefer13} has shown that deciding whether a given graph is unit-distance has the same complexity as deciding the truth of sentences in the existential theory of the real numbers; this is known to be NP-hard.  To get a feeling for the problem, consider the two graphs on 9 vertices and 15 edges depicted in \cref{mysteryFigure}.  We will show that one of these is unit-distance (see \cref{uglyFigure}), while the other is forbidden (see \cref{n9-last}).

\begin{figure}
\begin{center}
\begin{tikzpicture}[scale=.5]
\node[name=v0] at (2.5,5) {};
\node[name=v1] at (0.8682,4.3969) {};
\node[name=v2] at (0,2.8699) {};
\node[name=v3] at (0.3015,1.1334) {};
\node[name=v4] at (1.6318,0) {};
\node[name=v5] at (3.3682,0) {};
\node[name=v6] at (4.6985,1.1334) {};
\node[name=v7] at (5.0,2.8699) {};
\node[name=v8] at (4.1318,4.3969) {};
\draw (v0)--(v3) (v0)--(v5) (v0)--(v7)
(v1)--(v4)
(v1)--(v7)
(v1)--(v8)
(v2)--(v5)
(v2)--(v6)
(v2)--(v8)
(v3)--(v6)
(v3)--(v7)
(v4)--(v7)
(v4)--(v8)
(v5)--(v8)
(v6)--(v8);
\end{tikzpicture}
\hspace{1em}
\begin{tikzpicture}[scale=.5]
\node[name=v0] at (2.5,5) {};
\node[name=v1] at (0.8682,4.3969) {};
\node[name=v2] at (0,2.8699) {};
\node[name=v3] at (0.3015,1.1334) {};
\node[name=v4] at (1.6318,0) {};
\node[name=v5] at (3.3682,0) {};
\node[name=v6] at (4.6985,1.1334) {};
\node[name=v7] at (5.0,2.8699) {};
\node[name=v8] at (4.1318,4.3969) {};
\draw
(v0)--(v3)
(v0)--(v5)
(v0)--(v8)
(v1)--(v4)
(v1)--(v6)
(v1)--(v7)
(v2)--(v5)
(v2)--(v6)
(v2)--(v7)
(v3)--(v7)
(v3)--(v8)
(v4)--(v6)
(v4)--(v8)
(v5)--(v7)
(v6)--(v8);
\end{tikzpicture}
\end{center}
\caption{We will see that one of these two graphs is unit-distance (see \cref{uglyFigure}), while the other is forbidden (see \cref{n9-last}).}\label{mysteryFigure} 
\end{figure}
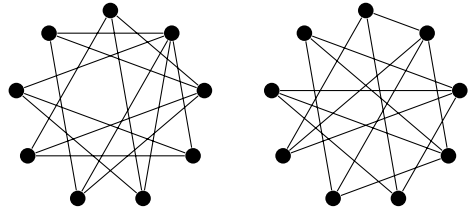

The goal of this article is to better understand unit-distance graphs by studying small obstructions.   We say that a forbidden graph is {\bf  minimal} when each of its proper subgraphs is unit-distance.  It is easy to see that both the complete graph $K_4$ and the complete bipartite graph $K_{2,3}$ are minimal forbidden graphs.  In fact, every graph on up to 5 vertices is unit-distance if and only if it does not contain either $K_4$ or $K_{2,3}$ as a subgraph.  Chilakamarri and Mahoney~\cite{chilakamarri95} extended this observation by proving that a set of six graphs, which we denote by $\mathcal{F}_{\leq 7}$ and depict in \cref{n7forbidden}, is the complete set of minimal forbidden graphs on up to 7 vertices.  In other words, every graph on up to 7 vertices is forbidden if and only if it contains a subgraph isomorphic to an element of $\mathcal{F}_{\leq 7}$.  Here we will extend their result to provide the complete list of minimal forbidden graphs on up to 9 vertices.  Let $\mathcal{F}_{\leq 9}$ denote the set of 74 graphs depicted in Appendix A, which consists of $\mathcal{F}_{\leq 7}$, 13 graphs on 8 vertices, and 55 graphs on 9 vertices.  We follow the notation of Chilakamarri and Mahoney and label these graphs $F(n,m,i)$, where $n$ indicates number of vertices, $m$ indicates number of edges, and the last index $i$ indicates only the order of appearance within this article.  Our main result is:

\begin{figure}[b!]
\begin{center}
\begin{tabular}{cccccc}
\begin{tikzpicture}
\node[name=1] at (0,0) {};
\node[name=2] at (1,0) {};
\node[name=3] at (1/2,{sqrt(3)/2}) {};
\node[name=4] at (3/2,{sqrt(3)/2}) {};
\draw (1) --(3) -- (2) -- (4) -- (1) -- (2) (3) -- (4);
\end{tikzpicture}
&
\begin{tikzpicture}
\node[name=1] at (0,0) {};
\node[name=2] at (1,0) {};
\node[name=3] at (1/2,{sqrt(3)/2}) {};
\node[name=4] at (3/2,{sqrt(3)/2}) {};
\node[name=5] at (3/4,{sqrt(3)/4}) {};
\draw (1) -- (2) -- (4) -- (3) -- (1) -- (5) -- (4);
\end{tikzpicture}
&
\begin{tikzpicture}
\node[name=1] at (0,0) {};
\node[name=2] at (1,0) {};
\node[name=3] at (2,0) {};
\node[name=4] at (1/2,{sqrt(3)/2}) {};
\node[name=5] at (3/2,{sqrt(3)/2}) {};
\node[name=6] at (1,-1/2) {};
\draw (4) -- (1) -- (2) -- (3) -- (5) -- (4) -- (2) -- (5) (1) -- (6) -- (3);
\end{tikzpicture}
&
\begin{tikzpicture}
\node[name=1] at (0,0) {};
\node[name=2] at (1,0) {};
\node[name=3] at (1/2,{sqrt(3)/2}) {};
\node[name=4] at (0,1) {};
\node[name=5] at (1,1) {};
\node[name=6] at (1/2,{1 + sqrt(3)/2}) {};
\node[name=7] at ({-sqrt(3)/2},1/2) {};
\draw (6) -- (4) -- (1) -- (2) -- (3) -- (6) -- (5) -- (2) (1) -- (3) (4) -- (7) -- (2);
\end{tikzpicture}
&
\begin{tikzpicture}
\node[name=1] at (0,0) {};
\node[name=2] at (1,0) {};
\node[name=3] at (2,0) {};
\node[name=4] at (1/2,{sqrt(3)/2}) {};
\node[name=5] at (3/2,{sqrt(3)/2}) {};
\node[name=6] at (5/2,{sqrt(3)/2}) {};
\node[name=7] at (5/2,-1/2) {};
\draw (4) -- (1) -- (2) -- (3) -- (5) -- (4) -- (2) -- (5) -- (6) -- (3);
\draw (1) -- (7) -- (6);
\end{tikzpicture}
&
\begin{tikzpicture}
\node[name=1] at (0,0) {};
\node[name=2] at (0,{sqrt(3)}) {};
\node[name=3] at (-1/2,{sqrt(3)/2}) {};
\node[name=4] at (1/2,{sqrt(3)/2}) {};
\node[name=5] at (3/2,{sqrt(3)/2}) {};
\node[name=6] at (1,0) {};
\node[name=7] at (1,{sqrt(3)}) {};
\draw (3) -- (1) -- (4) -- (2) -- (3) -- (4) -- (5) -- (7) -- (6) -- (1) (7) -- (2) (5) -- (6);
\end{tikzpicture}\\
$F(4,6,1)$ & $F(5,6,1)$ & $F(6,9,1)$ & $F(7,10,1)$ & $F(7,11,1)$ & $F(7,11,2)$
\end{tabular}
\end{center}

\caption{Chilakamarri and Mahoney~\cite{chilakamarri95} proved that these six graphs form the set $\mathcal{F}_{\leq 7}$ of all minimal forbidden graphs on up to 7 vertices.}\label{n7forbidden}
\end{figure}
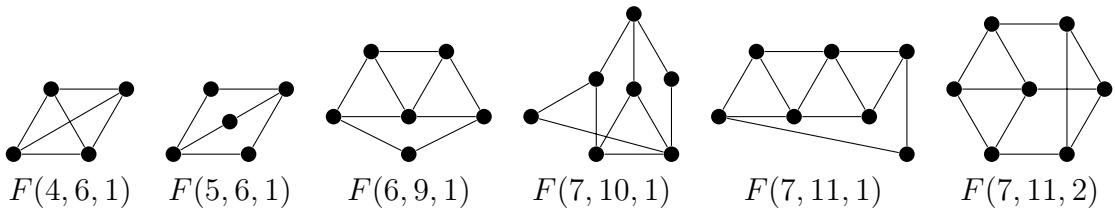

\begin{thm}\label{introTheorem}
A graph on at most 9 vertices is forbidden if and only if it contains a subgraph isomorphic to an element of $\mathcal{F}_{\leq 9}$.
\end{thm}

We will begin in the next section by introducing terminology and conventions that we will use throughout.  In Section 3, we will prove \cref{n8theorem}, which classifies the minimal forbidden graphs on 8 vertices.  In Section 4, we will prove \cref{n9theorem}, which classifies the minimal forbidden graphs on 9 vertices.  \cref{introTheorem} follows by combining the classification of $\mathcal{F}_{\leq 7}$ with \cref{n8theorem,n9theorem}.  In Appendix A, we collect the set of minimal forbidden graphs on up to 9 vertices, and in Appendix B, we report coordinates for embedded unit-distance graphs that we use in the proofs of \cref{n8theorem,n9theorem}.

Our approach relies heavily on the freely available computational tools \texttt{nauty}~\cite{mckay14} (to generate graphs) and \texttt{SageMath}~\cite{sagemath} (to work with them).  In the proof of \cref{n9theorem}, we found it necessary to use the implementation of cylindrical algebraic decomposition~\cite{collins75} within \texttt{Mathematica}~\cite{mathematica} to generate embeddings for the two unit-distance graphs depicted in \cref{uglyFigure}.  We intend for our computations to be readily reproducible, and we have made our code available~\cite{smallUDcode}.

\section{Preliminaries}

We define an {\bf embedding} of a unit-distance graph $G$ on vertices $V$ to be an injection $\varphi\colon V \rightarrow \mathbf{R}^2$ with the property that every pair of adjacent vertices $v,w \in V$ satisfies $|\varphi(v) - \varphi(w)| = 1$.  We will say that $G$ is {\bf rigid} if for every pair of its embeddings $\varphi, \psi$ and every pair of vertices $v,w \in V$, we have $|\varphi(v) - \varphi(w)| = |\psi(v) - \psi(w)|$.  We will frequently fix coordinates for rigid subgraphs of a unit-distance graph without impacting the generality of our arguments.

We will say that a unit-distance graph is {\bf embedded} when its vertices are distinct points in $\mathbf{R}^2$ and each of its edges are line segments of length 1.  We allow for the possibility that non-adjacent vertices are distance 1 apart.  Elementary geometry ensures that within any embedded unit-distance graph, every 3-cycle forms an equilateral triangle and every 4-cycle forms a rhombus.  We will use these observations extensively in the following form.

\begin{lem}\label{triangleRhombusLemma}
For any embedded unit-distance graph $G$, the following hold.
\begin{enumerate}
    \item[(i)] The angle between any two edges of a 3-cycle in $G$ is $\pi/3$.
    \item[(ii)] Opposite edges of a 4-cycle in $G$ are parallel.
\end{enumerate}
\end{lem}

We adopt the following conventions when drawing embedded unit-distance graphs.  Black dots and solid line segments represent the vertices and edges, respectively, of the graph under consideration.  White dots will represent relevant points of $\mathbf{R}^2$ that are determined by being distance 1 from other points pictured.  We use dashed edges to represent pairs of points that are required to be distance 1 apart despite not necessarily corresponding to adjacent vertices in the graph under consideration. For instance, we will frequently consider the following unit-distance graph:
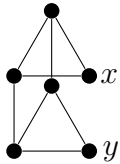
\begin{center}
\begin{tikzpicture}
\node[name=1] at (0,0) {};
\node[name=2,label=right:{$y$}] at (1,0) {};
\node[name=3] at (1/2,{sqrt(3)/2}) {};
\node[name=4] at (0,1) {};
\node[name=5,label=right:{$x$}] at (1,1) {};
\node[name=6] at (1/2,{1 + sqrt(3)/2}) {};
\draw (6) -- (4) -- (1) -- (2) -- (3) -- (6) -- (5) (1) -- (3);
\draw (4) -- (5);
\end{tikzpicture}
\end{center}
Due to \cref{triangleRhombusLemma}, it must either be the case that $|x - y| = 1$ or that $x$ lies exactly distance 1 away from the other common unit-distance neighbor of the two neighbors of $y$.  We can separate these exhaustive cases by considering the two following classes of embeddings:
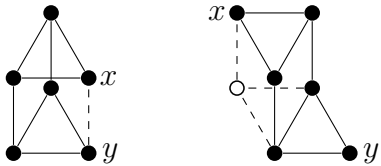
\begin{center}
\begin{tikzpicture}
\node[name=1] at (0,0) {};
\node[name=2,label=right:{$y$}] at (1,0) {};
\node[name=3] at (1/2,{sqrt(3)/2}) {};
\node[name=4] at (0,1) {};
\node[name=5,label=right:{$x$}] at (1,1) {};
\node[name=6] at (1/2,{1 + sqrt(3)/2}) {};
\draw (6) -- (4) -- (1) -- (2) -- (3) -- (6) -- (5) (1) -- (3);
\draw (4) -- (5);
\draw[dashed] (2) -- (5);
\end{tikzpicture}
\hspace{2em}
\begin{tikzpicture}
\node[name=1] at (0,0) {};
\node[name=2,label=right:{$y$}] at (1,0) {};
\node[name=3] at (1/2,{sqrt(3)/2}) {};
\node[name=4] at (0,1) {};
\node[name=5,label=left:{$x$}] at (-1/2,{1 + sqrt(3)/2}) {};
\node[name=6] at (1/2,{1 + sqrt(3)/2}) {};
\node[name=7] at (-1/2,{sqrt(3)/2}) {};
\node[fill=white,minimum size=0.4em] at (-1/2,{sqrt(3)/2}) {};
\draw (6) -- (4) -- (1) -- (2) -- (3) -- (6) -- (5) (1) -- (3);
\draw (4) -- (5);
\draw[dashed] (1) -- (7) -- (5) (3) -- (7);
\end{tikzpicture}
\end{center}
Of course, neither embedding pictured is rigid, but we intend for these illustrations to represent the equivalence class of embeddings with the same constraints imposed by the solid and dashed unit-distance edges.  We do not require white dots to be distinct from black dots.

\begin{figure}[t!]
\begin{center}
\begin{tikzpicture}
\node[name=1] at (0,0) {};
\node[name=2] at (1,0) {};
\node[name=3] at (1/2,{sqrt(3)/2}) {};
\node[name=4] at (0,1) {};
\node[name=5] at (1,1) {};
\node[name=6] at (1/2,{1 + sqrt(3)/2}) {};
\draw (6) -- (4) -- (1) -- (2) -- (3) -- (6) -- (5) -- (2) (1) -- (3);
\draw[dashed] (4) -- (5);
\end{tikzpicture}
\hspace{1em}
\begin{tikzpicture}
\node[name=1] at (0,0) {};
\node[name=2] at (1,0) {};
\node[name=3] at (1/2,{sqrt(3)/2}) {};
\node[name=4] at (3/2,{sqrt(3)/2}) {};
\node[name=5] at (-1/2,{sqrt(3)/2}) {};
\node[name=6] at (0,{sqrt(3)}) {};
\node[name=7] at (1,{sqrt(3)}) {};
\draw (5) -- (1) (3) -- (1) -- (2) -- (3) -- (5) -- (6) -- (7) -- (3) -- (6) (2) -- (4) -- (3);
\draw[dashed] (7) -- (4);
\end{tikzpicture}
\hspace{1em}
\begin{tikzpicture}
\node[name=1] at (0,0) {};
\node[name=2] at (1,0) {};
\node[name=3] at (1/2,{sqrt(3)/2}) {};
\node[name=4] at (3/2,{sqrt(3)/2}) {};
\node[name=5] at (-1/2,{sqrt(3)/2}) {};
\node[name=6] at (0,{sqrt(3)}) {};
\node[name=7] at (1,{sqrt(3)}) {};
\draw (5) -- (1) (3) -- (1) -- (2) -- (3) -- (5) -- (6) -- (7) -- (3) -- (6) (7) -- (4) -- (2);
\draw[dashed] (3) -- (4);
\end{tikzpicture}
\hspace{1em}
\begin{tikzpicture}
\node[name=1] at (0,0) {};
\node[name=2] at (1,0) {};
\node[name=3] at (1/2,{sqrt(3)/2}) {};
\node[name=4] at (3/2,{sqrt(3)/2}) {};
\node[name=5] at (0,1) {};
\node[name=6] at (1,1) {};
\node[name=7] at (1/2,{sqrt(3)/2 + 1}) {};
\node[name=8] at (3/2,{sqrt(3)/2 + 1}) {};
\draw (3) -- (1) -- (2) -- (4) -- (3) -- (2);
\draw (7) -- (5) -- (6) -- (8) -- (7);
\draw (1) -- (5);
\draw (3) -- (7);
\draw (4) -- (8);
\draw[dashed] (7) -- (6) -- (2);
\end{tikzpicture}
\end{center}
\caption{Each graph depicts a totally unfaithful unit-distance graph, drawn with solid edges, with non-adjacent vertices that are necessarily distance 1 apart in every embedding, represented with dashed edges.  To see that each graph is totally unfaithful in this fashion, it is enough to apply \cref{triangleRhombusLemma}.} \label{unfaithfulFigure}
\end{figure}
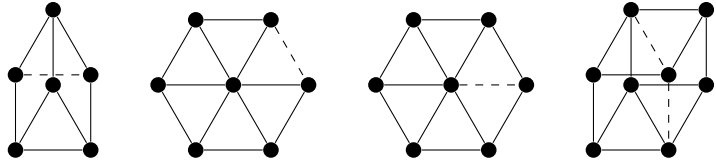

Borrowing terminology from Erd\H{o}s and Simonovits~\cite{erdos80}, we say that an embedding of a unit-distance graph $G$ on vertices $V$ is \textbf{faithful} when, for every pair of vertices $v,w \in V$, $|\varphi(v) - \varphi(w)| = 1$ if and only if $v$ and $w$ are adjacent.  We say that a unit-distance graph is \textbf{totally unfaithful} when it does not admit any faithful embedding.  In \cref{n8-2,n9-unfaithful}, we will prove that several graphs are forbidden by identifying a subgraph isomorphic to one of the totally unfaithful unit-distance graphs depicted in~\cref{unfaithfulFigure}.  Each of these graphs has at least one pair of non-adjacent vertices that are necessarily distance 1 apart in every embedding.  In particular, a graph $G$ containing one of the graphs in \cref{unfaithfulFigure} is unit-distance if and only if the graph obtained by adding the corresponding dashed edge to $G$ is also unit-distance.


\section{Forbidden graphs on eight vertices}

We first leverage rigid subgraphs to prove that a couple of graphs are forbidden.

\begin{lem}\label{n8-1}
The following two graphs are forbidden:
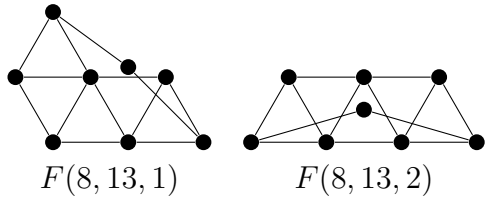
\begin{center}
\begin{tabular}{cc}
\begin{tikzpicture}
\node[name=1] at (0,0) {};
\node[name=2] at (1,0) {};
\node[name=3] at (1/2,{sqrt(3)/2}) {};
\node[name=4] at (3/2,{sqrt(3)/2}) {};
\node[name=5] at (-1/2,{sqrt(3)/2}) {};
\node[name=6] at (0,{sqrt(3)}) {};
\node[name=7] at (2,0) {};
\node[name=8] at (1,1) {};
\draw (5) -- (1) (3) -- (1) -- (2) -- (3) -- (5) -- (6) -- (3) (2) -- (4) -- (3);
\draw (4) -- (7) -- (2);
\draw (7) -- (8) -- (6);
\end{tikzpicture}
&
\begin{tikzpicture}
\node[name=1] at (0,0) {};
\node[name=2] at (1,0) {};
\node[name=3] at (2,0) {};
\node[name=4] at (1/2,{sqrt(3)/2}) {};
\node[name=5] at (3/2,{sqrt(3)/2}) {};
\node[name=6] at (5/2,{sqrt(3)/2}) {};
\node[name=7] at (3,0) {};
\node[name=8] at (3/2,{sqrt(3)/4}) {};
\draw (4) -- (1) -- (2) -- (3) -- (5) -- (4) -- (2) -- (5) -- (6) -- (3);
\draw (3) -- (7) -- (6);
\draw (1) -- (8) -- (7);
\end{tikzpicture}\\
$F(8,13,1)$ & $F(8,13,2)$
\end{tabular}
\end{center}
\end{lem}
\begin{proof}
Observe that $F(8,13,1)$ contains a subgraph isomorphic to the following rigid unit-distance subgraph:
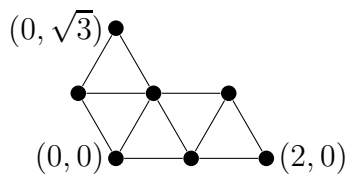
\begin{center}
\begin{tikzpicture}
\node[name=1,label=left:{$(0,0)$}] at (0,0) {};
\node[name=2] at (1,0) {};
\node[name=3] at (1/2,{sqrt(3)/2}) {};
\node[name=4] at (3/2,{sqrt(3)/2}) {};
\node[name=5] at (-1/2,{sqrt(3)/2}) {};
\node[name=6,label=left:{$(0,\sqrt{3})$}] at (0,{sqrt(3)}) {};
\node[name=7,label=right:{$(2,0)$}] at (2,0) {};
\draw (5) -- (1) (3) -- (1) -- (2) -- (3) -- (5) -- (6) -- (3) (2) -- (4) -- (3);
\draw (4) -- (7) -- (2);
\end{tikzpicture}
\end{center}
If $F(8,13,1)$ were unit-distance, then $(0,\sqrt{3})$ and $(2,0)$ would share a common unit-distance neighbor.  This is impossible since $|(0,\sqrt{3}) - (2,0)| = \sqrt{7} > 2$.

Similarly, $F(8,13,2)$ contains a subgraph isomorphic to the following rigid unit-distance subgraph:
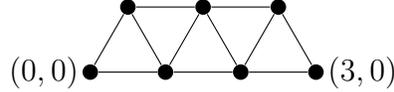
\begin{center}
\begin{tikzpicture}
\node[label=left:{$(0,0)$}, name=1] at (0,0) {};
\node[name=2] at (1,0) {};
\node[name=3] at (2,0) {};
\node[name=4] at (1/2,{sqrt(3)/2}) {};
\node[name=5] at (3/2,{sqrt(3)/2}) {};
\node[name=6] at (5/2,{sqrt(3)/2}) {};
\node[label=right:{$(3,0)$}, name=7] at (3,0) {};
\draw (4) -- (1) -- (2) -- (3) -- (5) -- (4) -- (2) -- (5) -- (6) -- (3);
\draw (3) -- (7) -- (6);
\end{tikzpicture}
\end{center}
If $F(8,13,2)$ were unit-distance, then $(0,0)$ and $(3,0)$ would share a common unit-distance neighbor.  This is impossible since $|(0,0) - (3,0)| = 3 > 2$.
\end{proof}

By locating totally unfaithful subgraphs, we prove that several more graphs are forbidden.

\begin{lem}\label{n8-2}
The following seven graphs are forbidden:

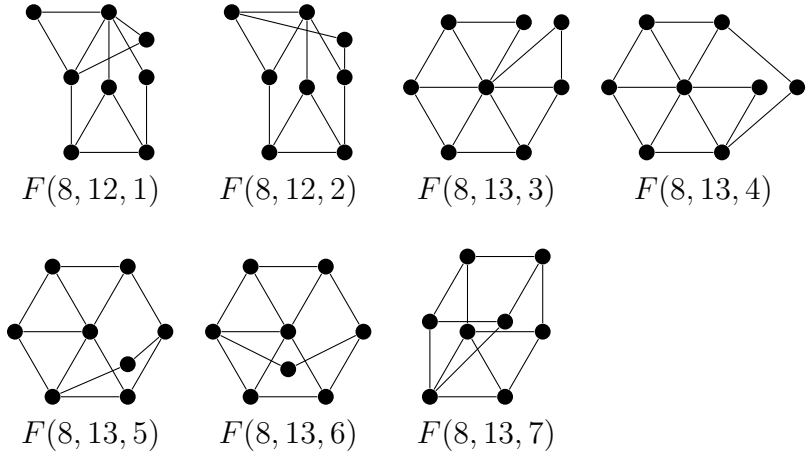
\begin{center}
\begin{tabular}{cccc}
\begin{tikzpicture}
\node[name=1] at (0,0) {};
\node[name=2] at (1,0) {};
\node[name=3] at (1/2,{sqrt(3)/2}) {};
\node[name=4] at (0,1) {};
\node[name=5] at (1,1) {};
\node[name=6] at (1/2,{1 + sqrt(3)/2}) {};
\node[name=7] at (1,3/2) {};
\node[name=8] at (-1/2,{1 + sqrt(3)/2}) {};
\draw (6) -- (4) -- (1) -- (2) -- (3) -- (6) -- (5) -- (2) (1) -- (3);
\draw (4) -- (7) -- (6) -- (8) -- (4);
\end{tikzpicture}
&
\begin{tikzpicture}
\node[name=1] at (0,0) {};
\node[name=2] at (1,0) {};
\node[name=3] at (1/2,{sqrt(3)/2}) {};
\node[name=4] at (0,1) {};
\node[name=5] at (1,1) {};
\node[name=6] at (1/2,{1 + sqrt(3)/2}) {};
\node[name=7] at (1,3/2) {};
\node[name=8] at (-1/2,{1 + sqrt(3)/2}) {};
\draw (6) -- (4) -- (1) -- (2) -- (3) -- (6) -- (5) -- (2) (1) -- (3);
\draw (6) -- (8) -- (4);
\draw (8) -- (7) -- (5);
\end{tikzpicture}
&
\begin{tikzpicture}
\node[name=1] at (0,0) {};
\node[name=2] at (1,0) {};
\node[name=3] at (1/2,{sqrt(3)/2}) {};
\node[name=4] at (3/2,{sqrt(3)/2}) {};
\node[name=5] at (-1/2,{sqrt(3)/2}) {};
\node[name=6] at (0,{sqrt(3)}) {};
\node[name=7] at (1,{sqrt(3)}) {};
\node[name=8] at (3/2,{sqrt(3)}) {};
\draw (5) -- (1) (3) -- (1) -- (2) -- (3) -- (5) -- (6) -- (7) -- (3) -- (6) (2) -- (4) -- (3) -- (8) -- (4);
\end{tikzpicture}
&
\begin{tikzpicture}
\node[name=1] at (0,0) {};
\node[name=2] at (1,0) {};
\node[name=3] at (1/2,{sqrt(3)/2}) {};
\node[name=4] at (3/2,{sqrt(3)/2}) {};
\node[name=5] at (-1/2,{sqrt(3)/2}) {};
\node[name=6] at (0,{sqrt(3)}) {};
\node[name=7] at (1,{sqrt(3)}) {};
\node[name=8] at (2,{sqrt(3)/2}) {};
\draw (5) -- (1) (3) -- (1) -- (2) -- (3) -- (5) -- (6) -- (7) -- (3) -- (6) (2) -- (4) -- (3);
\draw (2) -- (8) -- (7);
\end{tikzpicture}\\
$F(8,12,1)$ & $F(8,12,2)$ & $F(8,13,3)$ & $F(8,13,4)$\\
\\
\begin{tikzpicture}
\node[name=1] at (0,0) {};
\node[name=2] at (1,0) {};
\node[name=3] at (1/2,{sqrt(3)/2}) {};
\node[name=4] at (3/2,{sqrt(3)/2}) {};
\node[name=5] at (-1/2,{sqrt(3)/2}) {};
\node[name=6] at (0,{sqrt(3)}) {};
\node[name=7] at (1,{sqrt(3)}) {};
\node[name=8] at (1,{sqrt(3)/4}) {};
\draw (1) -- (8) -- (4);
\draw (5) -- (1) (3) -- (1) -- (2) -- (3) -- (5) -- (6) -- (7) -- (3) -- (6) (7) -- (4) -- (2);
\end{tikzpicture}
&
\begin{tikzpicture}
\node[name=1] at (0,0) {};
\node[name=2] at (1,0) {};
\node[name=3] at (1/2,{sqrt(3)/2}) {};
\node[name=4] at (3/2,{sqrt(3)/2}) {};
\node[name=5] at (-1/2,{sqrt(3)/2}) {};
\node[name=6] at (0,{sqrt(3)}) {};
\node[name=7] at (1,{sqrt(3)}) {};
\node[name=8] at (1/2,{sqrt(3)/2 - 1/2}) {};
\draw (5) -- (8) -- (4);
\draw (5) -- (1) (3) -- (1) -- (2) -- (3) -- (5) -- (6) -- (7) -- (3) -- (6) (7) -- (4) -- (2);
\end{tikzpicture}
&
\begin{tikzpicture}
\node[name=1] at (0,0) {};
\node[name=2] at (1,0) {};
\node[name=3] at (1/2,{sqrt(3)/2}) {};
\node[name=4] at (3/2,{sqrt(3)/2}) {};
\node[name=5] at (0,1) {};
\node[name=6] at (1,1) {};
\node[name=7] at (1/2,{sqrt(3)/2 + 1}) {};
\node[name=8] at (3/2,{sqrt(3)/2 + 1}) {};
\draw (3) -- (1) -- (2) -- (4) -- (3) -- (2);
\draw (7) -- (5) -- (6) -- (8) -- (7);
\draw (1) -- (5);
\draw (3) -- (7);
\draw (4) -- (8);
\draw (1) -- (6);
\end{tikzpicture}\\
$F(8,13,5)$ & $F(8,13,6)$ & $F(8,13,7)$
\end{tabular}
\end{center}
\end{lem}
\begin{proof}
Suppose, for a contradiction, that $F(8,12,1)$ were unit-distance.  By identifying a subgraph isomorphic to one of the totally unfaithful graphs from \cref{unfaithfulFigure}, we can add a new unit-distance edge to $F(8,12,1)$.  In particular, the following graph, consisting of $F(8,12,1)$ together with the dashed edge, must also be unit-distance.
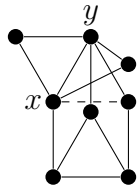
\begin{center}
\begin{tikzpicture}
\node[name=1] at (0,0) {};
\node[name=2] at (1,0) {};
\node[name=3] at (1/2,{sqrt(3)/2}) {};
\node[name=4,label=left:{$x$}] at (0,1) {};
\node[name=5] at (1,1) {};
\node[name=6,label=above:{$y$}] at (1/2,{1 + sqrt(3)/2}) {};
\node[name=7] at (1,3/2) {};
\node[name=8] at (-1/2,{1 + sqrt(3)/2}) {};
\draw (6) -- (4) -- (1) -- (2) -- (3) -- (6) -- (5) -- (2) (1) -- (3);
\draw (4) -- (7) -- (6) -- (8) -- (4);
\draw[dashed] (4) -- (5);
\end{tikzpicture}
\end{center}
Then $x$ and $y$ have three common unit-distance neighbors.  This would lead to a unit-distance embedding of the complete bipartite graph $K_{2,3}$, which we recall is the forbidden graph $F(5,6,1)$.  We have arrived at a contradiction, and so $F(8,12,1)$ is forbidden.

Similarly, assuming that any of the remaining graphs is unit-distance allows us to add a new unit-distance edge corresponding to a dashed edge from one of the totally unfaithful graphs from \cref{unfaithfulFigure}.  For $F(8,12,2)$ and $F(8,13,i)$ for $i \in \{3,4,5,7\}$, this results in two points with three common unit-distance neighbors, a contradiction.  For $F(8,13,6)$, this results in two points of distance 2 apart with two common neighbors, a contradiction.
\end{proof}

\newpage

We individually consider the remaining minimal forbidden graphs on 8 vertices.

\begin{lem}\label{n8-3}
	The following graph is forbidden:
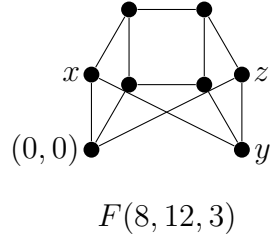
\begin{center}
\begin{tikzpicture}
\node[label=left:{$(0,0)$},name=1] at (0,0) {};
\node[label=left:{$x$},name=2] at (0,1) {};
\node[name=3] at (1/2,{sqrt(3)/2}) {};
\node[name=4] at (1/2,{sqrt(3)/2 + 1}) {};
\node[name=5] at (3/2,{sqrt(3)/2}) {};
\node[name=6] at (3/2,{sqrt(3)/2 + 1}) {};
\node[label=right:{$y$ \phantom{,0)}},name=7] at (2,0) {};
\node[label=right:{$z$},name=8] at (2,1) {};
\draw (1) -- (2) -- (4) -- (3) -- (1);
\draw (3) -- (5) -- (7) -- (8) -- (6) -- (4);
\draw (5) -- (6);
\draw (1) -- (8);
\draw (2) -- (7);
\end{tikzpicture}

$F(8,12,3)$
\end{center}
\end{lem}
\begin{proof}
Suppose, for a contradiction, that $F(8,12,3)$ were unit-distance.  Without loss of generality, we may fix one vertex at the origin and consider an embedding labeled as above.
Applying \cref{triangleRhombusLemma} to the rhombus through $\{(0,0), x, y, z\}$ demonstrates that $x = y - z$.  On the other hand, repeatedly applying \cref{triangleRhombusLemma}(ii) to the remaining three rhombi indicates that $x = z - y$.  This is only possible if $x = (0,0)$, so no such embedding is possible.
\end{proof}

\begin{lem}\label{n8-4}
The following graph is forbidden:

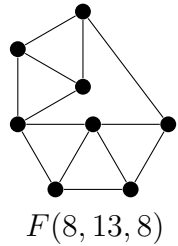
\begin{center}
\begin{tikzpicture}
\node[name=1] at (0,0) {};
\node[name=2] at (1,0) {};
\node[name=3] at (1/2,{sqrt(3)/2}) {};
\node[name=4] at (3/2,{sqrt(3)/2}) {};
\node[name=5] at (-1/2,{sqrt(3)/2}) {};
\node[name=6] at ({sqrt(3)/2 - 1/2}, {1/2 + sqrt(3)/2}) {};
\node[name=7] at (-1/2,{1 + sqrt(3)/2}) {};
\node[name=8] at ({-1/2 + sqrt(3)/2},{3/2 + sqrt(3)/2}) {};
\draw (1) -- (2) -- (4) -- (3) -- (5) -- (1) -- (3) -- (2);
\draw (5) -- (6) -- (8) -- (7) -- (5);
\draw (6) -- (7) (8) -- (4);
\end{tikzpicture}

$F(8,13,8)$
\end{center}
\end{lem}
\begin{proof}
Observe that $F(8,13,8)$ contains a subgraph isomorphic to the following unit-distance subgraph:
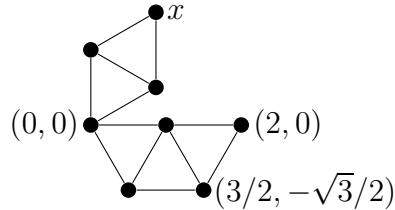
\begin{center}
\begin{tikzpicture}
\node[name=1] at (0,0) {};
\node[name=2,label=right:{$(3/2,-\sqrt{3}/2)$}] at (1,0) {};
\node[name=3] at (1/2,{sqrt(3)/2}) {};
\node[label=right:{$(2,0)$},name=4] at (3/2,{sqrt(3)/2}) {};
\node[label=left:{$(0,0)$},name=5] at (-1/2,{sqrt(3)/2}) {};
\node[name=6] at ({sqrt(3)/2 - 1/2}, {1/2 + sqrt(3)/2}) {};
\node[name=7] at (-1/2,{1 + sqrt(3)/2}) {};
\node[label=right:{$x$},name=8] at ({-1/2 + sqrt(3)/2},{3/2 + sqrt(3)/2}) {};
\draw (1) -- (2) -- (4) -- (3) -- (5) -- (1) -- (3) -- (2);
\draw (5) -- (6) -- (8) -- (7) -- (5);
\draw (6) -- (7);
\end{tikzpicture}
\end{center}
If $F(8,13,8)$ were unit-distance, then we would be able to position $x$ to satisfy both $|x| = \sqrt{3}$ and $|x - (2,0)| = 1$.  The only two such points are $x = (3/2,-\sqrt{3}/2)$, which is already occupied, and $x = (3/2,\sqrt{3}/2)$.  However, if $x = (3/2,\sqrt{3}/2)$, then $x$ would be distance 1 from $(1,0)$, leading to three common unit-distance neighbors between $(0,0)$ and $x$, a contradiction.
\end{proof}

\newpage

\begin{lem}\label{n8-5}
The following graph is forbidden:
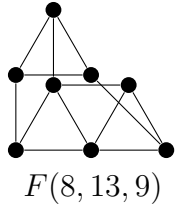
\begin{center}
\begin{tikzpicture}
\node[name=1] at (0,0) {};
\node[name=2] at (1,0) {};
\node[name=3] at (1/2,{sqrt(3)/2}) {};
\node[name=4] at (0,1) {};
\node[name=5] at (1,1) {};
\node[name=6] at (1/2,{1 + sqrt(3)/2}) {};
\node[name=7] at (2,0) {};
\node[name=8] at (3/2,{sqrt(3)/2}) {};
\draw (6) -- (4) -- (1) -- (2) -- (3) -- (6) -- (5) (1) -- (3);
\draw (4) -- (5);
\draw (8) -- (2) -- (7) -- (8) -- (3);
\draw (5) -- (7);
\end{tikzpicture}

$F(8,13,9)$
\end{center}
\end{lem}

\begin{proof}
Fixing some coordinates for a unit-distance subgraph of $F(8,13,9)$, we consider the following two classes of embeddings:
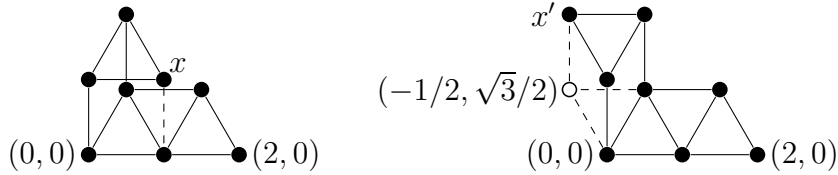
\begin{center}
\begin{tikzpicture}
\node[label=left:{$(0,0)$},name=1] at (0,0) {};
\node[name=2] at (1,0) {};
\node[name=3] at (1/2,{sqrt(3)/2}) {};
\node[name=4] at (0,1) {};
\node[name=5,label=above right:{$x$}] at (1,1) {};
\node[name=6] at (1/2,{1 + sqrt(3)/2}) {};
\node[name=7,label=right:{$(2,0)$}] at (2,0) {};
\node[name=8] at (3/2,{sqrt(3)/2}) {};
\draw (6) -- (4) -- (1) -- (2) -- (3) -- (6) -- (5) (1) -- (3);
\draw (4) -- (5);
\draw (8) -- (2) -- (7) -- (8) -- (3);
\draw[dashed] (5) -- (2);
\end{tikzpicture}
\hspace{1em}
\begin{tikzpicture}
\node[label=left:{$(0,0)$},name=1] at (0,0) {};
\node[name=2] at (1,0) {};
\node[name=3] at (1/2,{sqrt(3)/2}) {};
\node[name=4] at (0,1) {};
\node[name=5,label=left:{$x'$}] at (-1/2,{1 + sqrt(3)/2}) {};
\node[name=6] at (1/2,{1 + sqrt(3)/2}) {};
\node[name=7,label=right:{$(2,0)$}] at (2,0) {};
\node[name=8] at (3/2,{sqrt(3)/2}) {};
\node[name=9,label=left:{$(-1/2,\sqrt{3}/2)$}] at (-1/2,{sqrt(3)/2}) {};
\node[fill=white,minimum size=0.4em] at (-1/2,{sqrt(3)/2}) {};
\draw (6) -- (4) -- (1) -- (2) -- (3) -- (6) -- (5) (1) -- (3);
\draw (4) -- (5);
\draw (8) -- (2) -- (7) -- (8) -- (3);
\draw[dashed] (1) -- (9) -- (5) (3) -- (9);
\end{tikzpicture}
\end{center}
$F(8,13,9)$ is unit-distance only if we can arrange for either $x$ or $x'$ to be distance 1 from $(2,0)$.  Since we have $|x' - (-1/2,\sqrt{3}/2)| = 1$, we see that $|x' - (2,0)| > 1$.  Suppose, for a contradiction, that $|x - (2,0)| = 1$.  Together with $|x - (1,0)| = 1$, we see that the only two possibilities are $x = (3/2,\sqrt{3}/2)$, which is occupied, and $x = (3/2,-\sqrt{3}/2)$.  The latter would force $(1/2,\sqrt{3}/2)$ and $(3/2,-\sqrt{3}/2)$, two points of distance 2 apart, to have two common unit-distance neighbors; this leads to a contradiction.
\end{proof}

\begin{lem}\label{n8-6}
The following graph is forbidden:
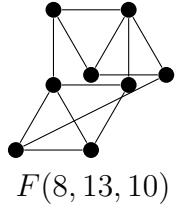
\begin{center}
\begin{tikzpicture}
\node[name=1] at (0,0) {};
\node[name=2] at (1,0) {};
\node[name=3] at (1/2,{sqrt(3)/2}) {};
\node[name=4] at (3/2,{sqrt(3)/2}) {};
\node[name=5] at (1/2,{sqrt(3)/2 + 1}) {};
\node[name=6] at (3/2,{sqrt(3)/2 + 1}) {};
\node[name=7] at (1,1) {};
\node[name=8] at (2,1) {};
\draw (1) -- (2) -- (4) -- (3) -- (2);
\draw (1) -- (3) -- (5) -- (6) -- (4);
\draw (5) -- (7) -- (8) -- (6) -- (7);
\draw (1) -- (8);
\end{tikzpicture}

$F(8,13,10)$
\end{center}
\end{lem}
\begin{proof}
Fixing some coordinates for a unit-distance subgraph of $F(8,13,10)$, we consider the following two classes of embeddings:

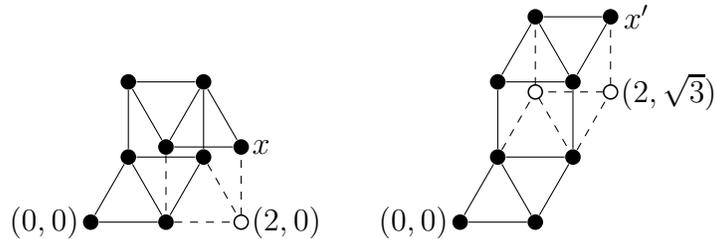
\begin{center}
\begin{tikzpicture}
\node[label=left:{$(0,0)$},name=1] at (0,0) {};
\node[name=2] at (1,0) {};
\node[name=3] at (1/2,{sqrt(3)/2}) {};
\node[name=4] at (3/2,{sqrt(3)/2}) {};
\node[name=5] at (1/2,{sqrt(3)/2 + 1}) {};
\node[name=6] at (3/2,{sqrt(3)/2 + 1}) {};
\node[name=7] at (1,1) {};
\node[name=8, label=right:{$x$}] at (2,1) {};
\draw (1) -- (2) -- (4) -- (3) -- (2);
\draw (1) -- (3) -- (5) -- (6) -- (4);
\draw (5) -- (7) -- (8) -- (6) -- (7);
\node at (2,0) {};
\node[fill=white,minimum size=0.4em,name=9,label=right:{$(2,0)$}] at (2,0) {};
\draw[dashed] (7) -- (2) -- (9) -- (8);
\draw[dashed] (4) -- (9);
\end{tikzpicture}
\hspace{1em}
\begin{tikzpicture}
\node[label=left:{$(0,0)$},name=1] at (0,0) {};
\node[name=2] at (1,0) {};
\node[name=3] at (1/2,{sqrt(3)/2}) {};
\node[name=4] at (3/2,{sqrt(3)/2}) {};
\node[name=5] at (1/2,{sqrt(3)/2 + 1}) {};
\node[name=6] at (3/2,{sqrt(3)/2 + 1}) {};
\node[name=7] at (1,{1 + sqrt(3)}) {};
\node[label=right:{$x'$},name=8] at (2,{1 + sqrt(3)}) {};
\node at (2,{sqrt(3)}) {};
\node at (1,{sqrt(3)}) {};
\node[fill=white,minimum size=0.4em,name=9,label=right:{$(2,\sqrt{3})$}] at (2,{sqrt(3)}) {};
\node[fill=white,minimum size=0.4em,name=10] at (1,{sqrt(3)}) {};
\draw (1) -- (2) -- (4) -- (3) -- (2);
\draw (1) -- (3) -- (5) -- (6) -- (4);
\draw (5) -- (7) -- (8) -- (6) -- (7);
\draw[dashed] (3) -- (10) -- (4) -- (9) -- (10);
\draw[dashed] (10) -- (7);
\draw[dashed] (9) -- (8);
\end{tikzpicture}
\end{center}
Observe that $F(8,13,10)$ is unit-distance only if we can arrange for either $|x| = 1$ or $|x'| = 1$.  Since we have $|x' - (2,\sqrt{3})| = 1$, we see that $|x'| > 1$.  Suppose, for a contradiction, that $|x| = 1$.  This would force $(0,0)$ and $(2,0)$, two points of distance 2 apart, to have two common unit-distance neighbors, a contradiction.
\end{proof}

We are now in a position to classify the forbidden graphs on 8 vertices.

\begin{thm}\label{n8theorem}
The set of minimal forbidden graphs on 8 vertices is given by
\[
\mathcal{F}_8 := \{F(8,12,i) : 1 \leq i \leq 3\} \cup \{F(8,13,i) : 1 \leq i \leq 10\}.
\]
\end{thm}

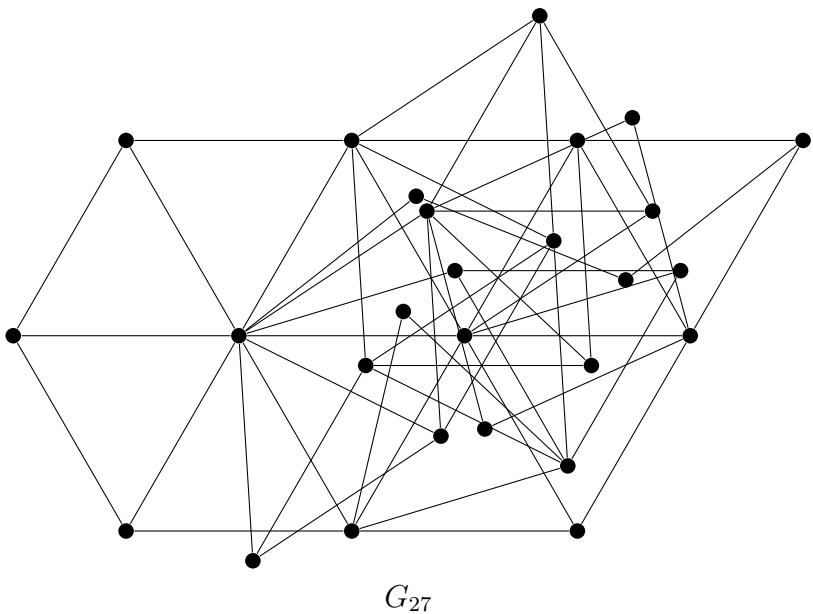
\begin{figure}[b!]
\begin{center}
\begin{tikzpicture}[scale=3]
\node[name=0] at (0, 0) {};
\node[name=1] at (-1/2, {1/2*sqrt(3)}) {};
\node[name=2] at (1/2, {1/2*sqrt(3)}) {};
\node[name=3] at (1, 0) {};
\node[name=4] at (1/2, {-1/2*sqrt(3)}) {};
\node[name=5] at  (-1/2, {-1/2*sqrt(3)}) {};
\node[name=6] at (-1, 0) {};
\node[name=7] at ({1/12*sqrt(33) + 1/12}, {-1/12*sqrt(11) + 1/12*sqrt(3)}) {};
\node[name=8] at ({1/12*sqrt(33) + 11/12}, {1/12*sqrt(11) + 1/12*sqrt(3)}) {};
\node[name=9] at ({1/6*sqrt(33) + 1/2}, {-1/3*sqrt(3)}) {};
\node[name=10] at ({1/6*sqrt(33)}, {1/6*sqrt(3)}) {};
\node[name=11] at (3/2, {1/2*sqrt(3)}) {};
\node[name=12] at ({1/12*sqrt(33) + 1/4}, {1/4*sqrt(11) - 5/12*sqrt(3)}) {};
\node[name=13] at ({1/6*sqrt(33) + 1}, {1/6*sqrt(3)}) {};
\node[name=14] at ({1/12*sqrt(33) - 5/12}, {-1/12*sqrt(11) - 5/12*sqrt(3)}) {};
\node[name=15] at (2, 0) {};
\node[name=16] at (5/6, {1/6*sqrt(11)}) {};
\node[name=17] at (11/6, {1/6*sqrt(11)}) {};
\node[name=18] at ({-1/60*sqrt(385) + 17/12}, {-7/60*sqrt(35) + 1/12*sqrt(11)}) {};
\node[name=19] at ({1/60*sqrt(385) + 17/12}, {7/60*sqrt(35) + 1/12*sqrt(11))}) {};
\node[name=20] at (5/2, {1/2*sqrt(3)}) {};
\node[name=21] at ({1/12*sqrt(11)*sqrt(3) + 13/12},{-1/12*sqrt(11) + 1/12*sqrt(3)}) {};
\node[name=22] at ({1/12*sqrt(11)*sqrt(3) + 5/12}, {1/12*sqrt(11) - 5/12*sqrt(3)}) {};
\node[name=23] at (3/2, {-1/2*sqrt(3)}) {};
\node[name=24] at (12/7, {1/7*sqrt(3)}) {};
\node[name=25] at (11/14, {5/14*sqrt(3)}) {};
\node[name=26] at (4/3, {1/6*sqrt(11) + 1/2*sqrt(3)}) {};
\foreach \x in {1,2,3,4,5,6,10,14,16,22,25}{
\draw (0) -- (\x);
}
\draw (1) -- (2);
\draw (1) -- (6);
\foreach \x in {3,7,8,11,26}{
\draw (2) -- (\x);
}
\foreach \x in {4,11,13,15,17,23}{
\draw (3) -- (\x);
}
\draw (4) -- (5);
\draw (4) -- (9);
\draw (4) -- (12);
\draw (4) -- (23);
\draw (5) -- (6);
\draw (8) -- (7) -- (9);
\draw (14) -- (7) -- (21);
\draw (8) -- (9);
\draw (22) -- (8) -- (26);
\draw (13) -- (10) -- (9) -- (12);
\draw (9) -- (13);
\draw (11) -- (15);
\draw (20) -- (11) -- (21);
\draw (14) -- (22);
\draw (18) -- (15) -- (19);
\draw (20) -- (15) -- (23);
\foreach \x in {17,18,19,21,22,26}{
\draw (16) -- (\x);
}
\draw (17) -- (26);
\draw (20) -- (24);
\draw (24) -- (25);
\end{tikzpicture}

$G_{27}$
\end{center}
\caption{The embedded unit-distance graph $G_{27}$ contains as a subgraph an isomorphic copy of every biconnected unit-distance graph on 8 vertices.  Exact coordinates for the vertices of $G_{27}$ are reported in \cref{g27table} in Appendix B.}\label{g27}
\end{figure}

\begin{proof}
Lemmas~\ref{n8-1} through \ref{n8-6} establish that every graph contained in $\mathcal{F}_8$ is forbidden.  Moreover, no graph in $\mathcal{F}_8$ contains a proper subgraph isomorphic to any graph in $\mathcal{F}_{\leq 7}$ or $\mathcal{F}_8$, so $\mathcal{F}_8$ contains only minimal forbidden graphs.  Set $
\mathcal{F}_{\leq 8} := \mathcal{F}_{\leq 7} \cup \mathcal{F}_8$.  Define a graph to be $\mathcal{F}_{\leq 8}$-free when it does not contain any element of $\mathcal{F}_{\leq 8}$ as a subgraph.  To finish the proof, we need only show that every $\mathcal{F}_{\leq 8}$-free graph on 8 vertices is unit-distance.  Of course, a disconnected graph is unit-distance if and only if each of its connected components is unit-distance.  Similarly, a connected graph is unit-distance if and only if each of its biconnected components is unit-distance; this was observed by Chilakamarri and Mahoney~\cite{chilakamarri95}.  Hence, we need only verify that every biconnected $\mathcal{F}_{\leq 8}$-free graph on 8 vertices is unit-distance.

We use \texttt{nauty} to generate the set of all 7123 biconnected graphs on 8 vertices.  \texttt{SageMath} computes that only 366 of these graphs are $\mathcal{F}_{\leq 8}$-free.  Moreover, each of these 366 are subgraphs of the embedded unit-distance graph $G_{27}$ portrayed in \cref{g27} with coordinates given in \cref{g27table} in Appendix B.  We constructed $G_{27}$ as follows.  First, we explicitly computed coordinates for several of the 366 graphs that we were attempting to embed.  Next, we focused on those graphs with a vertex of degree 2 that had not yet been embedded.  When we could successfully embed the subgraph without this degree-2 vertex into our set of already computed coordinates, we attempted to solve for new coordinates for this degree-2 vertex.  If we succeeded, we added this new vertex to our set of coordinates.  Iterating this procedure yielded a unit-distance graph which contained as a subgraph each of the 366 graphs being considered.  We finally produced $G_{27}$ by eliminating several unnecessary vertices.
\end{proof}

\section{Forbidden graphs on nine vertices}

As in the previous section, we begin by considering rigid and totally unfaithful subgraphs.

\begin{lem}\label{n9-first}
The following four graphs are forbidden:

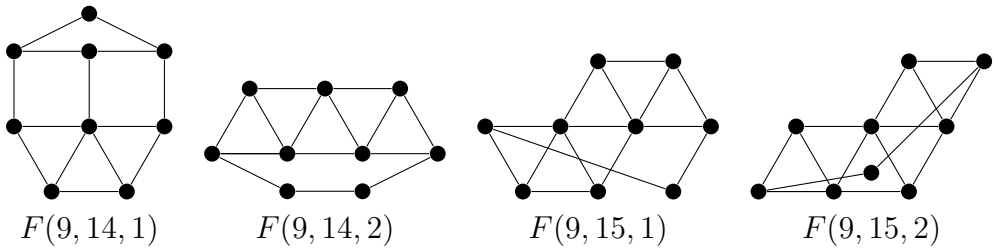
\begin{center}
\begin{tabular}{cccc}
\begin{tikzpicture}
\node[name=1] at (0,0) {};
\node[name=2] at (1/2,{sqrt(3)/2}) {};
\node[name=3] at (1,0) {};
\node[name=4] at (-1/2,{sqrt(3)/2}) {};
\node[name=5] at (3/2,{sqrt(3)/2}) {};
\node[name=6] at (-1/2,{1 + sqrt(3)/2}) {};
\node[name=7] at (1/2,{1 + sqrt(3)/2}) {};
\node[name=8] at (3/2,{1 + sqrt(3)/2}) {};
\node[name=9] at (1/2,{3/2 + sqrt(3)/2}) {};
\draw (1) -- (2) -- (3) -- (1);
\draw (2) -- (5) -- (3);
\draw (1) -- (4) -- (2);
\draw (4) -- (6) -- (7) -- (2);
\draw (5) -- (8) -- (7);
\draw (6) -- (9) -- (8);
\end{tikzpicture}
&
\begin{tikzpicture}
\node[name=1] at (0,0) {};
\node[name=2] at (1/2,{sqrt(3)/2}) {};
\node[name=3] at (1,0) {};
\node[name=4] at (2,0) {};
\node[name=5] at (3/2,{sqrt(3)/2}) {};
\node[name=6] at (3,0) {};
\node[name=7] at (5/2,{sqrt(3)/2}) {};
\draw (1) -- (2) -- (3) -- (1);
\draw (2) -- (5) -- (3);
\draw (1) -- (4);
\draw (4) -- (5) -- (7) -- (4) -- (6) -- (7);
\node[name=8] at (1,-1/2) {};
\node[name=9] at (2,-1/2) {};
\draw (1) -- (8) -- (9) -- (6);
\end{tikzpicture}
&
\begin{tikzpicture}
\node[name=1] at (0,0) {};
\node[name=2] at (1/2,{sqrt(3)/2}) {};
\node[name=3] at (1,0) {};
\node[name=4] at (2,{sqrt(3)}) {};
\node[name=5] at (3/2,{sqrt(3)/2}) {};
\node[name=6] at (-1/2,{sqrt(3)/2}) {};
\node[name=7] at (5/2,{sqrt(3)/2}) {};
\node[name=8] at (1,{sqrt(3)}) {};
\node[name=9] at (2,0) {};
\draw (3) -- (1);
\draw (5) -- (3);
\draw (4) -- (5) -- (7) -- (4);
\draw (1) -- (6) -- (2) -- (1);
\draw (2) -- (8) -- (5) -- (2) -- (3);
\draw (8) -- (4);
\draw (7) -- (9) -- (6);
\end{tikzpicture}
&
\begin{tikzpicture}
\node[name=1] at (0,0) {};
\node[name=2] at (1/2,{sqrt(3)/2}) {};
\node[name=3] at (1,0) {};
\node[name=4] at (2,0) {};
\node[name=5] at (3/2,{sqrt(3)/2}) {};
\node[name=6] at (2,{sqrt(3)}) {};
\node[name=7] at (5/2,{sqrt(3)/2}) {};
\draw (1) -- (2) -- (3) -- (1);
\draw (2) -- (5) -- (3);
\draw (1) -- (4);
\draw (4) -- (5) -- (7) -- (4);
\node[name=8] at (3,{sqrt(3)}) {};
\node[name=9] at (3/2,1/4) {};
\draw (1) -- (9) -- (8);
\draw (6) -- (8) -- (7);
\draw (5) -- (6) -- (7);
\end{tikzpicture}\\
$F(9,14,1)$ & $F(9,14,2)$ & $F(9,15,1)$ & $F(9,15,2)$
\end{tabular}
\end{center}
\end{lem}

\begin{proof}
As in \cref{n8-1}, each of these is forbidden due to containing a subgraph isomorphic to a rigid unit-distance.  By applying \cref{triangleRhombusLemma}, we see that $F(9,14,1)$ is unit-distance if and only if two points of distance 2 apart share 2 common neighbors, which is impossible.  To see that $F(9,14,2)$ is forbidden, observe that there is a unique path along three unit-distance line segments between any two points of distance 3 apart.  To see that each of $F(9,15,1)$ and $F(9,15,2)$ is forbidden, recall that two points of distance greater than 2 apart share no common unit-distance neighbor.
\end{proof}

\newpage
\begin{lem}\label{n9-unfaithful}
The following 29 graphs are forbidden:

\begin{center}
\begin{tabular}{ccccc}
\begin{tikzpicture}
\node[name=1] at (0,0) {};
\node[name=2] at (1,0) {};
\node[name=3] at (1/2,{sqrt(3)/2}) {};
\node[name=4] at (0,1) {};
\node[name=5] at (1,1) {};
\node[name=6] at (1/2,{1 + sqrt(3)/2}) {};
\draw (6) -- (4) -- (1) -- (2) -- (3) -- (6) -- (5) -- (2) (1) -- (3);
\node[name=7] at (0,2.2) {};
\node[name=8] at (1,2.2) {};
\node[name=9] at (3/2,1.75) {};
\draw (4) -- (7) -- (8) -- (5);
\draw (5) -- (9) -- (7);
\end{tikzpicture}
&
\begin{tikzpicture}
\node[name=1] at (0,0) {};
\node[name=2] at (1,0) {};
\node[name=3] at (1/2,{sqrt(3)/2}) {};
\node[name=4] at (0,1) {};
\node[name=5] at (1,1) {};
\node[name=6] at (1/2,{1 + sqrt(3)/2}) {};
\draw (6) -- (4) -- (1) -- (2) -- (3) -- (6) -- (5) -- (2) (1) -- (3);
\node[name=7] at (3/2,{sqrt(3)/2}) {};
\node[name=8] at (3/2,{1 + sqrt(3)/2}) {};
\draw (3) -- (7) -- (2);
\draw (6) -- (8) -- (5);
\node[name=9] at (2,3/2) {};
\draw (6) -- (9) -- (7);
\end{tikzpicture}
&
\begin{tikzpicture}
\node[name=1] at (0,0) {};
\node[name=2] at (1,0) {};
\node[name=3] at (1/2,{sqrt(3)/2}) {};
\node[name=4] at (0,1) {};
\node[name=5] at (1,1) {};
\node[name=6] at (1/2,{1 + sqrt(3)/2}) {};
\draw (6) -- (4) -- (1) -- (2) -- (3) -- (6) -- (5) -- (2) (1) -- (3);
\node[name=7] at (1/2,1/4) {};
\node[name=8] at (3/2,{1 + sqrt(3)/2}) {};
\node[name=9] at (2,1) {};
\draw (6) -- (8) -- (9) -- (5);
\draw (4) -- (7) -- (5);
\draw (7) -- (9);
\end{tikzpicture}
&
\begin{tikzpicture}
\node[name=1] at (0,0) {};
\node[name=2] at (1,0) {};
\node[name=3] at (1/2,{sqrt(3)/2}) {};
\node[name=4] at (0,1) {};
\node[name=5] at (1,1) {};
\node[name=6] at (1/2,{1 + sqrt(3)/2}) {};
\draw (6) -- (4) -- (1) -- (2) -- (3) -- (6) -- (5) -- (2) (1) -- (3);
\node[name=7] at ({1 + sqrt(3)/2},1/2) {};
\node[name=8] at (1/2,1/4) {};
\node[name=9] at (1/2,-1/2) {};
\draw (2) -- (7) -- (5);
\draw (7) -- (8) -- (4);
\draw (2) -- (9) -- (8);
\end{tikzpicture}
&
\begin{tikzpicture}
\node[name=1] at (0,0) {};
\node[name=2] at (1,0) {};
\node[name=3] at (1/2,{sqrt(3)/2}) {};
\node[name=4] at (0,1) {};
\node[name=5] at (1,1) {};
\node[name=6] at (1/2,{1 + sqrt(3)/2}) {};
\draw (6) -- (4) -- (1) -- (2) -- (3) -- (6) -- (5) -- (2) (1) -- (3);
\node[name=7] at (3/2,{sqrt(3)/2}) {};
\node[name=8] at (3/2,{1 + sqrt(3)/2}) {};
\draw (2) -- (7) -- (3);
\draw (7) -- (8) -- (5);
\node[name=9] at (1/2,9/4) {};
\draw (4) -- (9) -- (8);
\end{tikzpicture}\\
$F(9,13,1)$ & $F(9,14,3)$ & $F(9,14,4)$ & $F(9,14,5)$ & $F(9,14,6)$\\
\\
\begin{tikzpicture}
\node[name=1] at (0,0) {};
\node[name=2] at (1,0) {};
\node[name=3] at (1/2,{sqrt(3)/2}) {};
\node[name=4] at (0,1) {};
\node[name=5] at (1,1) {};
\node[name=6] at (1/2,{1 + sqrt(3)/2}) {};
\draw (6) -- (4) -- (1) -- (2) -- (3) -- (6) -- (5) -- (2) (1) -- (3);
\node[name=7] at (3/2,{sqrt(3)/2}) {};
\node[name=8] at (3/2,{1 + sqrt(3)/2}) {};
\draw (2) -- (7) -- (3);
\draw (7) -- (8) -- (6);
\node[name=9] at (1/2,9/4) {};
\draw (4) -- (9) -- (8);
\end{tikzpicture}
&
\begin{tikzpicture}
\node[name=1] at (0,0) {};
\node[name=2] at (1,0) {};
\node[name=3] at (1/2,{sqrt(3)/2}) {};
\node[name=4] at (0,1) {};
\node[name=5] at (1,1) {};
\node[name=6] at (1/2,{1 + sqrt(3)/2}) {};
\draw (6) -- (4) -- (1) -- (2) -- (3) -- (6) -- (5) -- (2) (1) -- (3);
\node[name=7] at ({1 + sqrt(3)/2},1/2) {};
\node[name=8] at (5/4,-3/8) {};
\node[name=9] at (1/2,1/4) {};
\draw (5) -- (7) -- (2);
\draw (7) -- (8) -- (1);
\draw (7) -- (9) -- (4);
\end{tikzpicture}
&
\begin{tikzpicture}
\node[name=1] at (0,0) {};
\node[name=2] at (1,0) {};
\node[name=3] at (1/2,{sqrt(3)/2}) {};
\node[name=4] at (0,1) {};
\node[name=5] at (1,1) {};
\node[name=6] at (1/2,{1 + sqrt(3)/2}) {};
\draw (6) -- (4) -- (1) -- (2) -- (3) -- (6) -- (5) -- (2) (1) -- (3);
\node[name=7] at (3/2,{1 + sqrt(3)/2}) {};
\node[name=8] at (1,{1 + sqrt(3)}) {};
\node[name=9] at (1/8,9/4) {};
\draw (5) -- (7) -- (6) -- (8) -- (7);
\draw (8) -- (9) -- (4);
\end{tikzpicture}
&
\begin{tikzpicture}
\node[name=1] at (0,0) {};
\node[name=2] at (1,0) {};
\node[name=3] at (1/2,{sqrt(3)/2}) {};
\node[name=4] at (0,1) {};
\node[name=5] at (1,1) {};
\node[name=6] at (1/2,{1 + sqrt(3)/2}) {};
\draw (6) -- (4) -- (1) -- (2) -- (3) -- (6) -- (5) -- (2) (1) -- (3);
\node[name=7] at (3/2,{1 + sqrt(3)/2}) {};
\draw (5) -- (7) -- (6);
\node[name=8] at (-1/2,{1 + sqrt(3)/2}) {};
\node[name=9] at (1/2,5/2) {};
\draw (4) -- (8) -- (6);
\draw (8) -- (9) -- (7);
\end{tikzpicture}
&
\begin{tikzpicture}
\node[name=1] at (0,0) {};
\node[name=2] at (1,0) {};
\node[name=3] at (1/2,{sqrt(3)/2}) {};
\node[name=4] at (0,1) {};
\node[name=5] at (1,1) {};
\node[name=6] at (1/2,{1 + sqrt(3)/2}) {};
\draw (6) -- (4) -- (1) -- (2) -- (3) -- (6) -- (5) -- (2) (1) -- (3);
\node[name=7] at (3/2,{1 + sqrt(3)/2}) {};
\draw (5) -- (7) -- (6);
\node[name=8] at (3/2,{sqrt(3)/2}) {};
\draw (2) -- (8) -- (3);
\node[name=9] at (2,3/2) {};
\draw (5) -- (9) -- (8);
\end{tikzpicture}\\
$F(9,14,7)$ & $F(9,14,8)$ & $F(9,14,9)$ & $F(9,14,10)$ & $F(9,14,11)$\\
\\
\begin{tikzpicture}
\node[name=1] at (0,0) {};
\node[name=2] at (1,0) {};
\node[name=3] at (1/2,{sqrt(3)/2}) {};
\node[name=4] at (0,1) {};
\node[name=5] at (1,1) {};
\node[name=6] at (1/2,{1 + sqrt(3)/2}) {};
\draw (6) -- (4) -- (1) -- (2) -- (3) -- (6) -- (5) -- (2) (1) -- (3);
\node[name=7] at (3/2,{1 + sqrt(3)/2}) {};
\node[name=8] at (3/2,3/4) {};
\node[name=9] at (1/2,1/4) {};
\draw (5) -- (7) -- (6);
\draw (4) -- (9) -- (5);
\draw (9) -- (8) -- (7);
\end{tikzpicture}
&
\begin{tikzpicture}
\node[name=1] at (0,0) {};
\node[name=2] at (1,0) {};
\node[name=3] at (1/2,{sqrt(3)/2}) {};
\node[name=4] at (0,1) {};
\node[name=5] at (1,1) {};
\node[name=6] at (1/2,{1 + sqrt(3)/2}) {};
\draw (6) -- (4) -- (1) -- (2) -- (3) -- (6) -- (5) -- (2) (1) -- (3);
\node[name=7] at (2,-1/4) {};
\node[name=8] at (2,3/4) {};
\node[name=9] at ({5/4},1/4) {};
\draw (2) -- (7) -- (8) -- (5);
\draw (7) -- (9) -- (8) (9) -- (4);
\end{tikzpicture}
&
\begin{tikzpicture}
\node[name=1] at (0,0) {};
\node[name=2] at (1,0) {};
\node[name=3] at (1/2,{sqrt(3)/2}) {};
\node[name=4] at (0,1) {};
\node[name=5] at (1,1) {};
\node[name=6] at (1/2,{1 + sqrt(3)/2}) {};
\draw (6) -- (4) -- (1) -- (2) -- (3) -- (6) -- (5) -- (2) (1) -- (3);
\node[name=7] at ({1 + sqrt(3)/2},1/2) {};
\node[name=8] at (-1/2,1/2) {};
\node[name=9] at (0,3/2) {};
\draw (2) -- (7) -- (5);
\draw (7) -- (8) -- (1);
\draw (5) -- (9) -- (8);
\end{tikzpicture}
&
\begin{tikzpicture}
\node[name=1] at (0,0) {};
\node[name=2] at (1,0) {};
\node[name=3] at (1/2,{sqrt(3)/2}) {};
\node[name=4] at (0,1) {};
\node[name=5] at (1,1) {};
\node[name=6] at (1/2,{1 + sqrt(3)/2}) {};
\draw (6) -- (4) -- (1) -- (2) -- (3) -- (6) -- (5) -- (2) (1) -- (3);
\node[name=7] at (3/2,{1 + sqrt(3)/2}) {};
\node[name=8] at (2,1) {};
\node[name=9] at (1/2,1/4) {};
\draw (6) -- (7) -- (8) -- (5) -- (7);
\draw (4) -- (9) -- (8);
\end{tikzpicture}
&
\begin{tikzpicture}
\node[name=1] at (0,0) {};
\node[name=2] at (1,0) {};
\node[name=3] at (1/2,{sqrt(3)/2}) {};
\node[name=4] at (3/2,{sqrt(3)/2}) {};
\node[name=5] at (0,1) {};
\node[name=6] at (1,1) {};
\node[name=7] at (1/2,{sqrt(3)/2 + 1}) {};
\node[name=8] at (3/2,{sqrt(3)/2 + 1}) {};
\draw (3) -- (1) -- (2) -- (4) -- (3) -- (2);
\draw (7) -- (5) -- (6) -- (8) -- (7);
\draw (1) -- (5);
\draw (3) -- (7);
\draw (4) -- (8);
\node[name=9] at (3/2,0) {};
\draw (3) -- (9) -- (6);
\end{tikzpicture}\\
$F(9,14,12)$ & $F(9,14,13)$ & $F(9,14,14)$ & $F(9,14,15)$ & $F(9,14,16)$\\
\\ 

\begin{tikzpicture}
\node[name=1] at (0,0) {};
\node[name=2] at (1,0) {};
\node[name=3] at (1/2,{sqrt(3)/2}) {};
\node[name=4] at (0,1) {};
\node[name=5] at (1,1) {};
\node[name=6] at (1/2,{1 + sqrt(3)/2}) {};
\draw (6) -- (4) -- (1) -- (2) -- (3) -- (6) -- (5) -- (2) (1) -- (3);
\node[name=7] at (3/2,{sqrt(3)/2}) {};
\draw (2) -- (7) -- (3);
\node[name=8] at (2,{3/2}) {};
\node[name=9] at (1,2) {};
\draw (7) -- (8) -- (4);
\draw (8) -- (9) -- (4);
\draw (9) -- (5);
\end{tikzpicture}
&
\begin{tikzpicture}
\node[name=1] at (0,0) {};
\node[name=2] at (1,0) {};
\node[name=3] at (1/2,{sqrt(3)/2}) {};
\node[name=4] at (0,1) {};
\node[name=5] at (1,1) {};
\node[name=6] at (1/2,{1 + sqrt(3)/2}) {};
\draw (6) -- (4) -- (1) -- (2) -- (3) -- (6) -- (5) -- (2) (1) -- (3);
\node[name=7] at ({1 + sqrt(3)/2},1/2) {};
\node[name=8] at ({-sqrt(3)/2},1/2) {};
\node[name=9] at ({1 + sqrt(3)/2},-1/2) {};
\draw (9) -- (8);
\draw (1) -- (8) -- (4);
\draw (2) -- (7) -- (5);
\draw (2) -- (9) -- (7);
\end{tikzpicture}
&
\begin{tikzpicture}
\node[name=1] at (0,0) {};
\node[name=2] at (1,0) {};
\node[name=3] at (1/2,{sqrt(3)/2}) {};
\node[name=4] at (0,1) {};
\node[name=5] at (1,1) {};
\node[name=6] at (1/2,{1 + sqrt(3)/2}) {};
\draw (6) -- (4) -- (1) -- (2) -- (3) -- (6) -- (5) -- (2) (1) -- (3);
\node[name=7] at ({1 + sqrt(3)/2},1/2) {};
\node[name=8] at (3/2,{1 + sqrt(3)/2}) {};
\node[name=9] at (9/4,5/4) {};
\draw (5) -- (7) -- (2);
\draw (6) -- (8) -- (5) -- (9) -- (7);
\draw (9) -- (8);
\end{tikzpicture}
&
\begin{tikzpicture}
\node[name=1] at (0,0) {};
\node[name=2] at (1,0) {};
\node[name=3] at (1/2,{sqrt(3)/2}) {};
\node[name=4] at (0,1) {};
\node[name=5] at (1,1) {};
\node[name=6] at (1/2,{1 + sqrt(3)/2}) {};
\draw (6) -- (4) -- (1) -- (2) -- (3) -- (6) -- (5) -- (2) (1) -- (3);
\node[name=7] at (3/2,{sqrt(3)/2}) {};
\node[name=8] at (1,{sqrt(3)}) {};
\draw (2) -- (7) -- (3) -- (8) -- (7);
\node[name=9] at (2,11/8) {};
\draw (5) -- (9) -- (8);
\draw (9) -- (4);
\end{tikzpicture}
&
\begin{tikzpicture}
\node[name=1] at (0,0) {};
\node[name=2] at (1,0) {};
\node[name=3] at (1/2,{sqrt(3)/2}) {};
\node[name=4] at (0,1) {};
\node[name=5] at (1,1) {};
\node[name=6] at (1/2,{1 + sqrt(3)/2}) {};
\draw (6) -- (4) -- (1) -- (2) -- (3) -- (6) -- (5) -- (2) (1) -- (3);
\node[name=7] at (1/2,1/4) {};
\node[name=8] at (3/2,1/4) {};
\node[name=9] at (1,{1/4 - sqrt(3)/2}) {};
\draw (5) -- (8) -- (7) -- (5);
\draw (4) -- (7) -- (9) -- (8);
\draw (9) -- (3);
\end{tikzpicture}\\
$F(9,15,3)$ & $F(9,15,4)$ & $F(9,15,5)$ & $F(9,15,6)$ & $F(9,15,7)$\\
\\

\begin{tikzpicture}
\node[name=1] at (0,0) {};
\node[name=2] at (1,0) {};
\node[name=3] at (1/2,{sqrt(3)/2}) {};
\node[name=4] at (0,1) {};
\node[name=5] at (1,1) {};
\node[name=6] at (1/2,{1 + sqrt(3)/2}) {};
\draw (6) -- (4) -- (1) -- (2) -- (3) -- (6) -- (5) -- (2) (1) -- (3);
\node[name=7] at ({1 + sqrt(3)/2},1/2) {};
\node[name=8] at ({-sqrt(3)/2},1/2) {};
\node[name=9] at ({1 + sqrt(3)/2},3/2) {};
\draw (9) -- (8);
\draw (1) -- (8) -- (4);
\draw (2) -- (7) -- (5);
\draw (5) -- (9) -- (7);
\end{tikzpicture}
&
\begin{tikzpicture}
\node[name=1] at (0,0) {};
\node[name=2] at (1,0) {};
\node[name=3] at (1/2,{sqrt(3)/2}) {};
\node[name=4] at (3/2,{sqrt(3)/2}) {};
\node[name=5] at (-1/2,{sqrt(3)/2}) {};
\node[name=6] at (0,{sqrt(3)}) {};
\node[name=7] at (1,{sqrt(3)}) {};
\draw (5) -- (1) (3) -- (1) -- (2) -- (3) -- (5) -- (6) -- (7) -- (3) -- (6) (2) -- (4) -- (3);
\node[name=8] at (2,{sqrt(3)}) {};
\node[name=9] at (1,9/8) {};
\draw (6) -- (8) -- (4);
\draw (8) -- (9) -- (1);
\end{tikzpicture}
&
\begin{tikzpicture}
\node[name=1] at (0,0) {};
\node[name=2] at (1,0) {};
\node[name=3] at (1/2,{sqrt(3)/2}) {};
\node[name=4] at (3/2,{sqrt(3)/2}) {};
\node[name=5] at (-1/2,{sqrt(3)/2}) {};
\node[name=6] at (0,{sqrt(3)}) {};
\node[name=7] at (1,{sqrt(3)}) {};
\draw (5) -- (1) (3) -- (1) -- (2) -- (3) -- (5) -- (6) -- (7) -- (3) -- (6) (2) -- (4) -- (3);
\node[name=8] at (2,0) {};
\node[name=9] at (2,1) {};
\draw (2) -- (8) -- (4);
\draw (8) -- (9) -- (7);
\end{tikzpicture}
&
\begin{tikzpicture}
\node[name=1] at (0,0) {};
\node[name=2] at (1,0) {};
\node[name=3] at (1/2,{sqrt(3)/2}) {};
\node[name=4] at (3/2,{sqrt(3)/2}) {};
\node[name=5] at (-1/2,{sqrt(3)/2}) {};
\node[name=6] at (0,{sqrt(3)}) {};
\node[name=7] at (1,{sqrt(3)}) {};
\draw (5) -- (1) (3) -- (1) -- (2) -- (3) -- (5) -- (6) -- (7) -- (3) -- (6) (7) -- (4) -- (2);
\node[name=8] at (2,0) {};
\node[name=9] at (1,-1/2) {};
\draw (1) -- (9) -- (8);
\draw (2) -- (8) -- (4);
\end{tikzpicture}
&
\begin{tikzpicture}
\node[name=1] at (0,0) {};
\node[name=2] at (1,0) {};
\node[name=3] at (1/2,{sqrt(3)/2}) {};
\node[name=4] at (3/2,{sqrt(3)/2}) {};
\node[name=5] at (-1/2,{sqrt(3)/2}) {};
\node[name=6] at (0,{sqrt(3)}) {};
\node[name=7] at (1,{sqrt(3)}) {};
\draw (5) -- (1) (3) -- (1) -- (2) -- (3) -- (5) -- (6) -- (7) -- (3) -- (6) (7) -- (4) -- (2);
\node[name=8] at (2,{sqrt(3)}) {};
\node[name=9] at (1,1/2) {};
\draw (7) -- (8) -- (4) -- (9) -- (7);
\end{tikzpicture}\\
$F(9,15,8)$ & $F(9,15,9)$ & $F(9,15,10)$ & $F(9,15,11)$ & $F(9,15,12)$\\
\\

\begin{tikzpicture}
\node[name=1] at (0,0) {};
\node[name=2] at (1,0) {};
\node[name=3] at (1/2,{sqrt(3)/2}) {};
\node[name=4] at (3/2,{sqrt(3)/2}) {};
\node[name=5] at (-1/2,{sqrt(3)/2}) {};
\node[name=6] at (0,{sqrt(3)}) {};
\node[name=7] at (1,{sqrt(3)}) {};
\draw (5) -- (1) (3) -- (1) -- (2) -- (3) -- (5) -- (6) -- (7) -- (3) -- (6) (7) -- (4) -- (2);
\node[name=8] at (2,{sqrt(3)}) {};
\node[name=9] at (1,1/2) {};
\draw (7) -- (8) -- (4);
\draw (5) -- (9) -- (8);
\end{tikzpicture}
&
\begin{tikzpicture}
\node[name=1] at (0,0) {};
\node[name=2] at (1,0) {};
\node[name=3] at (1/2,{sqrt(3)/2}) {};
\node[name=4] at (3/2,{sqrt(3)/2}) {};
\node[name=5] at (-1/2,{sqrt(3)/2}) {};
\node[name=6] at (0,{sqrt(3)}) {};
\node[name=7] at (1,{sqrt(3)}) {};
\draw (5) -- (1) (3) -- (1) -- (2) -- (3) -- (5) -- (6) -- (7) -- (3) -- (6) (7) -- (4) -- (2);
\node[name=8] at (2,{sqrt(3)}) {};
\node[name=9] at (1,1/2) {};
\draw (7) -- (8) -- (4);
\draw (1) -- (9) -- (8);
\end{tikzpicture}
&
\begin{tikzpicture}
\node[name=1] at (0,0) {};
\node[name=2] at (1,0) {};
\node[name=3] at (1/2,{sqrt(3)/2}) {};
\node[name=4] at (3/2,{sqrt(3)/2}) {};
\node[name=5] at (-1/2,{sqrt(3)/2}) {};
\node[name=6] at (0,{sqrt(3)}) {};
\node[name=7] at (1,{sqrt(3)}) {};
\draw (5) -- (1) (3) -- (1) -- (2) -- (3) -- (5) -- (6) -- (7) -- (3) -- (6) (7) -- (4) -- (2);
\node[name=8] at (-1,0) {};
\draw (1) -- (8) -- (5);
\node[name=9] at (1/2,1/4) {};
\draw (8) -- (9) -- (4);
\end{tikzpicture}
&
\begin{tikzpicture}
\node[name=1] at (0,0) {};
\node[name=2] at (1,0) {};
\node[name=3] at (1/2,{sqrt(3)/2}) {};
\node[name=4] at (3/2,{sqrt(3)/2}) {};
\node[name=5] at (-1/2,{sqrt(3)/2}) {};
\node[name=6] at (0,{sqrt(3)}) {};
\node[name=7] at (1,{sqrt(3)}) {};
\draw (5) -- (1) (3) -- (1) -- (2) -- (3) -- (5) -- (6) -- (7) -- (3) -- (6) (7) -- (4) -- (2);
\node[name=8] at (2,{sqrt(3)}) {};
\node[name=9] at (2,1/2) {};
\draw (7) -- (8) -- (4);
\draw (3) -- (9) -- (8);
\end{tikzpicture}\\
$F(9,15,13)$ & $F(9,15,14)$ & $F(9,15,15)$ & $F(9,15,16)$
\end{tabular}
\end{center}
\end{lem}

\newpage
\begin{proof}
As in \cref{n8-2}, each of these is forbidden due to containing a subgraph isomorphic to one of the totally unfaithful graphs depicted in \cref{unfaithfulFigure}. For each graph, we include the additional edge that must appear in every embedding of its totally unfaithful subgraph.  We then use \texttt{SageMath} to verify that the resulting graph contains one of the graphs from $\mathcal{F}_{\leq 8}$ already known to be forbidden.
\end{proof}

\begin{lem}
The following two graphs are forbidden:
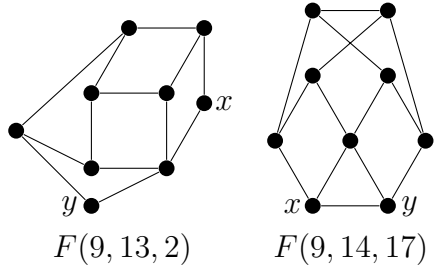
\begin{center}
\begin{tabular}{cc}
\begin{tikzpicture}
\node[name=1] at (0,0) {};
\node[name=2] at (1,0) {};
\node[name=3] at (1,1) {};
\node[name=4] at (0,1) {};
\node[name=5] at (1/2,{1 + sqrt(3)/2}) {};
\node[name=6] at (3/2,{1 + sqrt(3)/2}) {};
\node[name=7,label=right:{$x$}] at (3/2,{sqrt(3)/2}) {};
\node[name=8] at (-1,1/2) {};
\node[name=9,label=left:{$y$}] at (0,-1/2) {};
\draw (1) -- (2) -- (3) -- (4) -- (1) -- (8) -- (5) -- (6) -- (7) -- (2);
\draw (3) -- (6) (4) -- (5) (2) -- (9) -- (8);
\end{tikzpicture}
&
\begin{tikzpicture}
\node[name=1,label=left:{$x$}] at (0,0) {};
\node[name=2,label=right:{$y$}] at (1,0) {};
\node[name=3] at (1/2,{sqrt(3)/2}) {};
\node[name=4] at (3/2,{sqrt(3)/2}) {};
\node[name=5] at (1,{sqrt(3)}) {};
\node[name=6] at (0,{sqrt(3)}) {};
\node[name=7] at (-1/2,{sqrt(3)/2}) {};
\node[name=8] at (1,{3*sqrt(3)/2}) {};
\node[name=9] at (0,{3*sqrt(3)/2}) {};
\draw (1) -- (2) -- (3) -- (1) (2) -- (4) -- (5) -- (3) (1) -- (7) -- (6) -- (3) (4) -- (8) -- (9) -- (5) (7) -- (9) (6) -- (8);
\end{tikzpicture}\\
$F(9,13,2)$ & $F(9,14,17)$
\end{tabular}
\end{center}
\end{lem}
\begin{proof}
Suppose, for a contradiction, that either graph were unit-distance and consider an embedding labeled as above. Repeatedly applying \cref{triangleRhombusLemma}(ii) shows that the directed edges from the common neighbor of $x$ and $y$ to each of $x$ and $y$ are equal as unit vectors.  In either graph, this implies $x = y$, so no such embedding is possible.
\end{proof}

\begin{lem}
The following graph is forbidden:
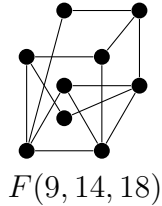
\begin{center}
\begin{tikzpicture}
\node[name=1] at (0,0) {};
\node[name=2] at (1,0) {};
\node[name=3] at (1/2,{sqrt(3)/2}) {};
\node[name=4] at (3/2,{sqrt(3)/2}) {};
\draw (1) -- (2) -- (4) -- (3) -- (1) (2) -- (3);
\node[name=5] at (0,5/4) {};
\node[name=6] at (1,5/4) {};
\node[name=7] at (3/2,{1 + sqrt(3)/2}) {};
\node[name=8] at (1/2,{sqrt(3)/4}) {};
\node[name=9] at (1/2,{1 + sqrt(3)/2}) {};
\draw (1) -- (5) -- (6) -- (2) (4) -- (7) -- (6);
\draw (7) -- (9) -- (1);
\draw (5) -- (8) -- (4);
\end{tikzpicture}

$F(9,14,18)$
\end{center}
\end{lem}
\begin{proof}
$F(9,14,18)$ contains a subgraph isomorphic to the following unit-distance graph:
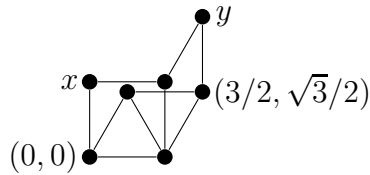
\begin{center}
\begin{tikzpicture}
\node[name=1,label=left:{$(0,0)$}] at (0,0) {};
\node[name=2] at (1,0) {};
\node[name=3] at (1/2,{sqrt(3)/2}) {};
\node[name=4,label=right:{$(3/2,\sqrt{3}/2)$}] at (3/2,{sqrt(3)/2}) {};
\draw (1) -- (2) -- (4) -- (3) -- (1) (2) -- (3);
\node[name=5,label=left:{$x$}] at (0,1) {};
\node[name=6] at (1,1) {};
\node[name=7,label=right:{$y$}] at (3/2,{1 + sqrt(3)/2}) {};
\draw (1) -- (5) -- (6) -- (2) (4) -- (7) -- (6);
\end{tikzpicture}
\end{center}
By \cref{triangleRhombusLemma}(ii), $|x - y| = \sqrt{3}$.  For $F(9,14,18)$ to be unit-distance, there would need to be a common unit-distance neighbor $x'$ of $x$ and $(3/2,\sqrt{3}/2)$ and a common unit-distance neighbor $y'$ of $y$ and $(0,0)$ with both $x'$ and $y'$ distinct from the points already depicted.  Since the parallelogram between the four labeled points has side lengths $1$ and $\sqrt{3}$, the parallelogram law tells us that $|y|^2 + |x - (3/2,\sqrt{3})|^2 = 8$.  Together with the restrictions of $|y| \leq 2$ and $|x - (3/2,\sqrt{3})| \leq 2$ required for $x'$ and $y'$ to exist, we must have $|y| = 2$, forcing either $y = (2,0)$ or $y = (1,\sqrt{3})$.  In either case, $y$ has a unique common unit-distance neighbor with $(0,0)$ which is already depicted, so it is impossible to include the point $y'$.
\end{proof}

\newpage

\begin{lem}\label{collinearLemma}
The following graph is forbidden:
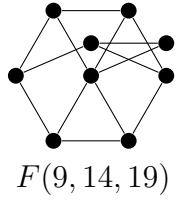
\begin{center}
\begin{tikzpicture}
\node[name=1] at (0,0) {};
\node[name=2] at (1,0) {};
\node[name=3] at (2,0) {};
\node[name=4] at (1/2, {sqrt(3)/2}) {};
\node[name=5] at (1/2, {-sqrt(3)/2}) {};
\node[name=6] at (3/2, {sqrt(3)/2}) {};
\node[name=7] at (3/2, {-sqrt(3)/2}) {};
\draw (1) -- (4) -- (2) -- (5) -- (1);
\draw (3) -- (6) -- (2) -- (7) -- (3);
\draw (4) -- (6) (5) -- (7);
\node[name=8] at (1,{sqrt(3)/4}) {};
\node[name=9] at (2,{sqrt(3)/4}) {};
\draw (1) -- (8) -- (3);
\draw (8) -- (9) -- (2);
\end{tikzpicture}

$F(9,14,19)$
\end{center}
\end{lem}

\begin{proof}

Observe that $F(9,14,19)$ contains the following unit-distance graph:

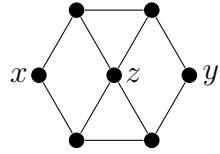
\begin{center}
\begin{tikzpicture}
\node[label=left:{$x$},name=1] at (0,0) {};
\node[name=2,label=right:{$z$}] at (1,0) {};
\node[label=right:{$y$},name=3] at (2,0) {};
\node[name=4] at (1/2, {sqrt(3)/2}) {};
\node[name=5] at (1/2, {-sqrt(3)/2}) {};
\node[name=6] at (3/2, {sqrt(3)/2}) {};
\node[name=7] at (3/2, {-sqrt(3)/2}) {};
\draw (1) -- (4) -- (2) -- (5) -- (1);
\draw (3) -- (6) -- (2) -- (7) -- (3);
\draw (4) -- (6) (5) -- (7);
\end{tikzpicture}
\end{center}
It is not hard to see that in any embedding of this unit-distance graph, $x, y$ and $z$ are collinear; we fix an embedding with $x = (a,0)$ with $a < 0$, $z = (0,0)$, and $y = (b,0)$ with $b > 0$.  With basic trigonometry we can solve for
\[
a = b/2 - \sqrt{3 - 3b^2/4}.
\]
For $F(9,14,19)$ to be unit-distance, there must exist $x' = (x_1,x_2)$ and $y' = (y_1, y_2)$ with $|x' - x| = |x' - y| = 1$ and $|y' - x'| = |y'| = 1$ distinct from the points already depicted.   In particular, we require $y_1 \not \in \{a/2,b/2\}$, since the four corresponding unit-distance neighbors of the origin of the form $(a/2, \cdot)$ or $(b/2, \cdot)$ are already occupied.  With $a$ as above, we set up the system of equations
\begin{align*}
(x_1 - a)^2 + x_2^2 &= 1\\
(x_1 - b)^2 + x_2^2 &= 1\\
(y_1 - x_1)^2 + (y_2 - x_2)^2 &= 1\\
y_1^2 + y_2^2 &= 1
\end{align*}
The only solution to the system above with $y_1 \not \in \{a/2,b/2\}$ has both $a = -1$ and $b = 1$.  Hence, if $F(9,14,19)$ were unit-distance, then the points $(-1,0)$ and $(1,0)$ would have two unit-distance neighbors; a contradiction.
\end{proof}

\begin{lem}
The following graph is forbidden:
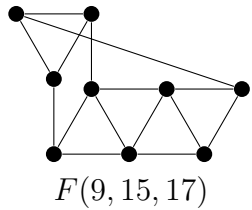
\begin{center}
\begin{tikzpicture}
\node[name=1] at (0,0) {};
\node[name=2] at (1,0) {};
\node[name=3] at (1/2,{sqrt(3)/2}) {};
\node[name=4] at (0,1) {};
\node[name=5] at (-1/2,{1 + sqrt(3)/2}) {};
\node[name=6] at (1/2,{1 + sqrt(3)/2}) {};
\node[name=7] at (2,0) {};
\node[name=8] at (3/2,{sqrt(3)/2}) {};
\node[name=9] at (5/2,{sqrt(3)/2}) {};
\draw (6) -- (4) -- (1) -- (2) -- (3) -- (6) -- (5) (1) -- (3);
\draw (4) -- (5);
\draw (8) -- (2) -- (7) -- (8) -- (3);
\draw (7) -- (9) -- (8);
\draw (5) -- (9);
\end{tikzpicture}

$F(9,15,17)$
\end{center}
\end{lem}

\newpage

\begin{proof}
Fixing some coordinates for a unit-distance subgraph of $F(9,15,17)$, we consider the following two classes of embeddings:

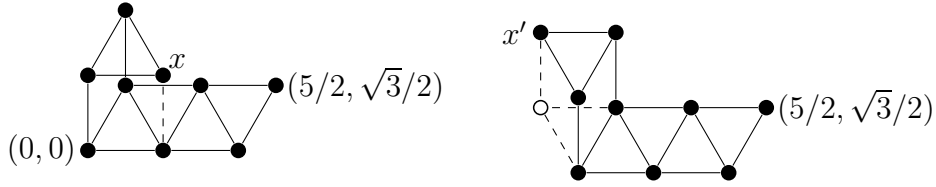
\begin{center}
\begin{tikzpicture}
\node[name=1,label=left:{$(0,0)$}] at (0,0) {};
\node[name=2] at (1,0) {};
\node[name=3] at (1/2,{sqrt(3)/2}) {};
\node[name=4] at (0,1) {};
\node[name=5,label=above right:{$x$}] at (1,1) {};
\node[name=6] at (1/2,{1 + sqrt(3)/2}) {};
\node[name=7] at (2,0) {};
\node[name=8] at (3/2,{sqrt(3)/2}) {};
\node[name=9,label=right:{$(5/2,\sqrt{3}/2)$}] at (5/2,{sqrt(3)/2}) {};
\draw (6) -- (4) -- (1) -- (2) -- (3) -- (6) -- (5) (1) -- (3);
\draw (4) -- (5);
\draw (8) -- (2) -- (7) -- (8) -- (3);
\draw (7) -- (9) -- (8);
\draw[dashed] (2) -- (5);
\end{tikzpicture}
\hspace{1em}
\begin{tikzpicture}
\node[name=1] at (0,0) {};
\node[name=2] at (1,0) {};
\node[name=3] at (1/2,{sqrt(3)/2}) {};
\node[name=4] at (0,1) {};
\node[name=5,label=left:{$x'$}] at (-1/2,{1 + sqrt(3)/2}) {};
\node[name=6] at (1/2,{1 + sqrt(3)/2}) {};
\node[name=7] at (2,0) {};
\node[name=8] at (3/2,{sqrt(3)/2}) {};
\node[name=9,label=right:{$(5/2,\sqrt{3}/2)$}] at (5/2,{sqrt(3)/2}) {};
\node[name=10] at (-1/2,{sqrt(3)/2}) {};
\draw (6) -- (4) -- (1) -- (2) -- (3) -- (6) -- (5) (1) -- (3);
\draw (4) -- (5);
\draw (8) -- (2) -- (7) -- (8) -- (3);
\draw (7) -- (9) -- (8);
\draw[dashed] (5) -- (10) -- (1) (3) -- (10);
\node[fill=white,minimum size=0.4em] at (-1/2,{sqrt(3)/2}) {};
\end{tikzpicture}
\end{center}
Observe that $F(9,15,17)$ is unit-distance only if we can arrange for either $x$ or $x'$ to be distance 1 from the point $(5/2,\sqrt{3}/2)$.  Notice that $|x - (5/2,\sqrt{3}/2)| = 1$ would lead to two points with three common unit-distance neighbors, which is impossible.  Moreover, since $|x' - (-1/2,\sqrt{3}/2)| = 1$, we see that $|x' - (5/2,\sqrt{3}/2)| > 1$.  In any case, $F(9,15,17)$ is forbidden.
\end{proof}

\begin{lem}
The following graph is forbidden:
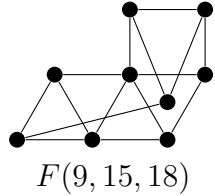
\begin{center}
\begin{tikzpicture}
\node[name=1] at (0,0) {};
\node[name=2] at (1,0) {};
\node[name=3] at (1/2,{sqrt(3)/2}) {};
\node[name=4] at (2,1/2) {};
\node[name=5] at (3/2,{sqrt(3)}) {};
\node[name=6] at (5/2,{sqrt(3)}) {};
\node[name=7] at (2,0) {};
\node[name=8] at (3/2,{sqrt(3)/2}) {};
\node[name=9] at (5/2,{sqrt(3)/2}) {};
\draw (6) -- (4) (1) -- (2) -- (3)  (6) -- (5) (1) -- (3);
\draw (4) -- (5);
\draw (8) -- (2) -- (7) -- (8) -- (3);
\draw (7) -- (9) -- (8);
\draw (5) -- (8);
\draw (6) -- (9);
\draw (1) -- (4);
\end{tikzpicture}

$F(9,15,18)$
\end{center}
\end{lem}
\begin{proof}
Fixing some coordinates for a unit-distance subgraph of $F(9,15,18)$, we consider the following two classes of embeddings:

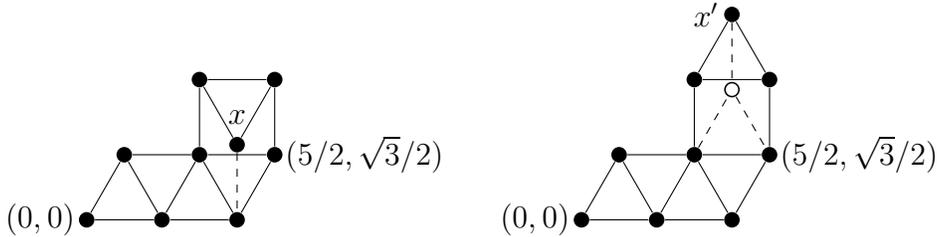
\begin{center}
\begin{tikzpicture}
\node[name=1,label=left:{$(0,0)$}] at (0,0) {};
\node[name=2] at (1,0) {};
\node[name=3] at (1/2,{sqrt(3)/2}) {};
\node[fill=white,label=above:{$x$}] at (2,1.1) {};
\node[name=4] at (2,1) {};
\node[name=5] at (3/2,{1 + sqrt(3)/2}) {};
\node[name=6] at (5/2,{1 + sqrt(3)/2}) {};
\node[name=7] at (2,0) {};
\node[name=8] at (3/2,{sqrt(3)/2}) {};
\node[name=9,label=right:{$(5/2,\sqrt{3}/2)$}] at (5/2,{sqrt(3)/2}) {};
\draw (6) -- (4) (1) -- (2) -- (3)  (6) -- (5) (1) -- (3);
\draw (4) -- (5);
\draw (8) -- (2) -- (7) -- (8) -- (3);
\draw (7) -- (9) -- (8);
\draw (5) -- (8);
\draw (6) -- (9);
\draw[dashed] (4) -- (7);
\end{tikzpicture}
\hspace{1em}
\begin{tikzpicture}
\node[name=1,label=left:{$(0,0)$}] at (0,0) {};
\node[name=2] at (1,0) {};
\node[name=3] at (1/2,{sqrt(3)/2}) {};
\node[name=4,label=left:{$x'$}] at (2,{1 + sqrt(3)}) {};
\node[name=5] at (3/2,{1 + sqrt(3)/2}) {};
\node[name=6] at (5/2,{1 + sqrt(3)/2}) {};
\node[name=7] at (2,0) {};
\node[name=8] at (3/2,{sqrt(3)/2}) {};
\node[name=9,label=right:{$(5/2,\sqrt{3}/2)$}] at (5/2,{sqrt(3)/2}) {};
\node[name=10] at (2,{sqrt(3)}) {};
\node[fill=white,minimum size=0.4em] at (2,{sqrt(3)}) {};
\draw (6) -- (4) (1) -- (2) -- (3)  (6) -- (5) (1) -- (3);
\draw (4) -- (5);
\draw (8) -- (2) -- (7) -- (8) -- (3);
\draw (7) -- (9) -- (8);
\draw (5) -- (8);
\draw (6) -- (9);
\draw[dashed] (4) -- (10);
\draw[dashed] (8) -- (10) -- (9);
\end{tikzpicture}
\end{center}
Observe that $F(9,15,18)$ is unit-distance only if we can arrange for either $|x| = 1$ or $|x'| = 1$.  If $|x| = 1$, then we would have two distinct common unit-distance neighbors between the origin and $(2,0)$, which is impossible.  Moreover, since $|x' - (2,\sqrt{3})| = 1$, it must be that $|x'| > 1$.  Hence, $F(9,15,18)$ is forbidden.
\end{proof}

\begin{lem}
The following graph is forbidden:

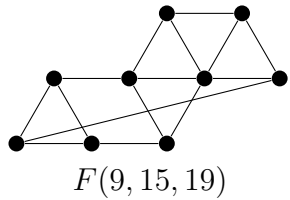
\begin{center}
\begin{tikzpicture}
\node[name=1] at (0,0) {};
\node[name=2] at (1/2,{sqrt(3)/2}) {};
\node[name=3] at (1,0) {};
\node[name=4] at (2,{sqrt(3)}) {};
\node[name=5] at (3/2,{sqrt(3)/2}) {};
\node[name=6] at (-1/2,{sqrt(3)/2}) {};
\node[name=7] at (5/2,{sqrt(3)/2}) {};
\node[name=8] at (1,{sqrt(3)}) {};
\node[name=9] at (-1,0) {};
\draw (3) -- (1);
\draw (5) -- (3);
\draw (4) -- (5) -- (7) -- (4);
\draw (1) -- (6) -- (2);
\draw (2) -- (8) -- (5) -- (2) -- (3);
\draw (8) -- (4);
\draw (1) -- (9) -- (6);
\draw (9) -- (7);
\end{tikzpicture}

$F(9,15,19)$
\end{center}
\end{lem}

\newpage

\begin{proof}
Fixing some coordinates for a unit-distance subgraph of $F(9,15,19)$, we consider the following embedding:
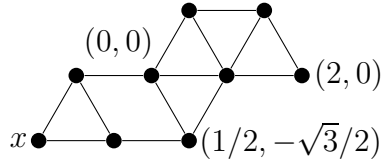
\begin{center}
\begin{tikzpicture}
\node[name=1] at (0,0) {};
\node[label=above left:{$(0,0)$},name=2] at (1/2,{sqrt(3)/2}) {};
\node[label=right:{$(1/2,-\sqrt{3}/2)$},name=3] at (1,0) {};
\node[name=4] at (2,{sqrt(3)}) {};
\node[name=5] at (3/2,{sqrt(3)/2}) {};
\node[name=6] at (-1/2,{sqrt(3)/2}) {};
\node[label=right:{$(2,0)$},name=7] at (5/2,{sqrt(3)/2}) {};
\node[name=8] at (1,{sqrt(3)}) {};
\node[label=left:{$x$},name=9] at (-1,0) {};
\draw (3) -- (1);
\draw (5) -- (3);
\draw (4) -- (5) -- (7) -- (4);
\draw (1) -- (6) -- (2);
\draw (2) -- (8) -- (5) -- (2) -- (3);
\draw (8) -- (4);
\draw (1) -- (9) -- (6);
\end{tikzpicture}
\end{center}
Observe that $F(9,15,19)$ is unit-distance if and only if we can arrange for $|x - (2,0)| = 1$.  However, in any embedding of the unit-distance graph depicted above, we either have $|x - (-1/2,-\sqrt{3}/2)| = 1$ or $|x - (1,0)| = 1$.  In the first case, it is impossible for $|x - (2,0)| = 1$ since $|(-1/2,-\sqrt{3}/2) - (2,0)| = \sqrt{7} > 2$.  For the second case, suppose that $|x - (1,0)| = 1$ and $|x - (2,0)| = 1$.  Then since $(3/2,\sqrt{3}/2)$ is occupied, we must have $x = (3/2,-\sqrt{3}/2)$.  But this leads to $x$ and $(0,0)$ having three common unit-distance neighbors, a contradiction.
\end{proof}

\begin{lem}
The following graph is forbidden:
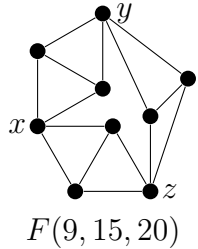
\begin{center}
\begin{tikzpicture}
\node[name=1] at (0,0) {};
\node[name=2,label=right:{$z$}] at (1,0) {};
\node[name=3] at (1/2,{sqrt(3)/2}) {};
\node[label=left:{$x$},name=5] at (-1/2,{sqrt(3)/2}) {};
\node[name=6] at ({sqrt(3)/2 - 1/2}, {1/2 + sqrt(3)/2}) {};
\node[name=7] at (-1/2,{1 + sqrt(3)/2}) {};
\node[label=right:{$y$},name=8] at ({-1/2 + sqrt(3)/2},{3/2 + sqrt(3)/2}) {};
\node[name=4] at (1,1) {};
\node[name=9] at (3/2,3/2) {};
\draw (1) -- (2) -- (3) -- (5) -- (1) -- (3);
\draw (5) -- (6) -- (8) -- (7) -- (5);
\draw (6) -- (7);
\draw (2) -- (4) -- (9) -- (8) -- (4) (2) -- (9);
\end{tikzpicture}

$F(9,15,20)$
\end{center}
\end{lem}
\begin{proof}
Suppose, for a contradiction, that $F(9,15,20)$ were unit-distance.  Then $\{x,y,z\}$ form the vertices of an equilateral triangle of side length $\sqrt{3}$, and we may as well assume $x = (0,0)$, $y = (\sqrt{3}/2,3/2)$, and $z = (\sqrt{3},0)$.  These three points have a common unit-distance neighbor, namely, $(\sqrt{3}/2,1/2)$.  As the two common unit-distance neighbors of $x$ and $y$ are already pictured, one of these must also lie distance 1 from $z$.  However, adding either edge results in two points with three common unit-distance neighbors, a contradiction.
\end{proof}

\begin{lem}
The following graph is forbidden:
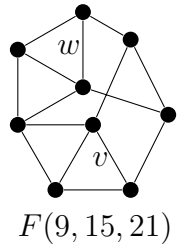
\begin{center}
\begin{tikzpicture}
\node[name=1] at (0,0) {};
\node[name=2] at (1,0) {};
\node[name=3] at (1/2,{sqrt(3)/2}) {};
\node[name=5] at (-1/2,{sqrt(3)/2}) {};
\node[name=6] at ({sqrt(3)/2 - 1/2}, {1/2 + sqrt(3)/2}) {};
\node[name=7] at (-1/2,{1 + sqrt(3)/2}) {};
\node[name=8] at ({sqrt(3)/2 - 1/2},{3/2 + sqrt(3)/2}) {};
\node[name=4] at (1,2) {};
\node[name=9] at (3/2,1) {};
\draw (1) -- (2) -- (3) -- (5) -- (1) -- (3);
\draw (5) -- (6) -- (8) -- (7) -- (5);
\draw (6) -- (7);
\draw (3) -- (4) (9) -- (6);
\draw (4) -- (9);
\draw (8) -- (4) (9) -- (2);
\node[fill=none,label={$v$}] at  (.6,.15) {};
\node[fill=none,label={$w$}] at  (.18,1.6) {};
\end{tikzpicture}

$F(9,15,21)$
\end{center}
\end{lem}
\begin{proof}
Suppose, for a contradiction, that $F(9,15,21)$ were unit-distance.  By \cref{triangleRhombusLemma}, the edges $v$ and $w$ must be parallel.  As the two diamond subgraphs share a vertex, we can assume without loss of generality that a unit-distance embedding of $F(9,15,21)$ contains one of the following subgraphs:
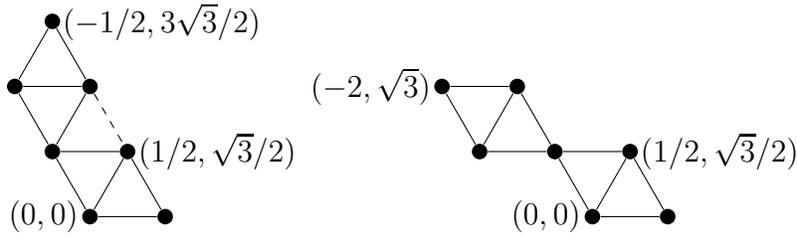
\begin{center}
\begin{tikzpicture}
\node[name=1,label=left:{$(0,0)$}] at (0,0) {};
\node[name=2] at (1,0) {};
\node[name=3,label=right:{$(1/2,\sqrt{3}/2)$}] at (1/2,{sqrt(3)/2}) {};
\node[name=5] at (-1/2,{sqrt(3)/2}) {};
\node[name=6] at (0, {sqrt(3)}) {};
\node[name=7] at (-1,{sqrt(3)}) {};
\node[name=8,label=right:{$(-1/2,3\sqrt{3}/2)$}] at (-1/2,{3*sqrt(3)/2}) {};
\draw (1) -- (2) -- (3) -- (5) -- (1) -- (3);
\draw (5) -- (6) -- (8) -- (7) -- (5);
\draw (6) -- (7);
\draw[dashed] (6) -- (3);
\end{tikzpicture}
\begin{tikzpicture}
\node[name=1,label=left:{$(0,0)$}] at (0,0) {};
\node[name=2] at (1,0) {};
\node[name=3,label=right:{$(1/2,\sqrt{3}/2)$}] at (1/2,{sqrt(3)/2}) {};
\node[name=5] at (-1/2,{sqrt(3)/2}) {};
\node[name=6] at (-3/2, {sqrt(3)/2}) {};
\node[name=7] at (-1,{sqrt(3)}) {};
\node[name=8,label=left:{$(-2,\sqrt{3})$}] at (-2,{sqrt(3)}) {};
\draw (1) -- (2) -- (3) -- (5) -- (1) -- (3);
\draw (5) -- (6) -- (8) -- (7) -- (5);
\draw (6) -- (7);
\end{tikzpicture}
\end{center}
Any such embedding of $F(9,15,21)$ must contain a new vertex not already pictured above that is one of the following: (i) a common unit-distance neighbor of $(-1/2,3\sqrt{3}/2)$ and $(1/2,\sqrt{3}/2)$, (ii) a common unit-distance neighbor of  $(-1/2,3\sqrt{3}/2)$ and $(0,0)$, (iii) a common unit-distance neighbor of $(-2,\sqrt{3})$ and $(0,0)$, or (iv) a common unit-distance neighbor of $(-2,\sqrt{3})$ and $(1/2,\sqrt{3})$.  Case (i) leads to two points of distance 2 apart sharing two common unit-distance neighbors, while the cases (ii) -- (iv) each require two points of distance greater than 2 to share a common unit-distance neighbor.  Each case results in a contradiction, so $F(9,15,21)$ is forbidden as desired.
\end{proof}

\begin{lem}
The following graph is forbidden:
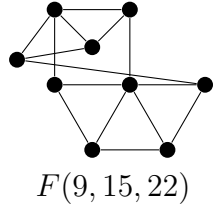
\begin{center}
\begin{tikzpicture}
\node[name=1] at (0,0) {};
\node[name=2] at (1,0) {};
\node[name=3] at (1/2,{-sqrt(3)/2}) {};
\node[name=4] at (0,1) {};
\node[name=5] at (1,1) {};
\node[name=6] at (1/2,1/2) {};
\node[name=7] at (2,0) {};
\node[name=8] at (3/2,{-sqrt(3)/2}) {};
\node[name=9] at (-1/2, 1/3) {};
\draw (6) -- (4) -- (1) -- (2) -- (3) (6) -- (5) (1) -- (3);
\draw (4) -- (5) -- (2);
\draw (8) -- (2) -- (7) -- (8) -- (3);
\draw (4) -- (9) -- (6);
\draw (9) -- (7);
\end{tikzpicture}

$F(9,15,22)$
\end{center}
\end{lem}

\begin{proof}
Fixing some coordinates for a unit-distance subgraph of $F(9,15,22)$, we consider the following two classes of embeddings:
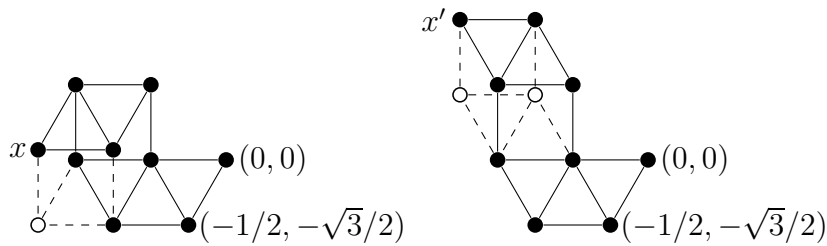
\begin{center}
\begin{tikzpicture}
\node[name=1] at (0,0) {};
\node[name=2] at (1,0) {};
\node[name=3] at (1/2,{-sqrt(3)/2}) {};
\node[name=4] at (0,1) {};
\node[name=5] at (1,1) {};
\node[name=6] at (1/2,{1 - sqrt(3)/2}) {};
\node[name=7,label=right:{$(0,0)$}] at (2,0) {};
\node[name=8,label=right:{$(-1/2,-\sqrt{3}/2)$}] at (3/2,{-sqrt(3)/2}) {};
\node[name=9,label=left:{$x$}] at (-1/2, {1 - sqrt(3)/2}) {};
\draw (6) -- (4) -- (1) -- (2) -- (3) (6) -- (5) (1) -- (3);
\draw (4) -- (5) -- (2);
\draw (8) -- (2) -- (7) -- (8) -- (3);
\draw (4) -- (9) -- (6);
\node[name=10] at (-1/2,{-sqrt(3)/2}) {};
\node[fill=white,minimum size=0.4em] at (-1/2,{-sqrt(3)/2}) {};
\draw[dashed] (9) -- (10);
\draw[dashed] (3) -- (10) -- (1);
\draw[dashed] (6) -- (3);
\end{tikzpicture}
\begin{tikzpicture}
\node[name=1] at (0,0) {};
\node[name=2] at (1,0) {};
\node[name=3] at (1/2,{-sqrt(3)/2}) {};
\node[name=4] at (0,1) {};
\node[name=5] at (1,1) {};
\node[name=6] at (1/2,{1 + sqrt(3)/2}) {};
\node[name=7,label=right:{$(0,0)$}] at (2,0) {};
\node[name=8,label=right:{$(-1/2,-\sqrt{3}/2)$}] at (3/2,{-sqrt(3)/2}) {};
\node[name=9,label=left:{$x'$}] at (-1/2, {1 + sqrt(3)/2}) {};
\draw (6) -- (4) -- (1) -- (2) -- (3) (6) -- (5) (1) -- (3);
\draw (4) -- (5) -- (2);
\draw (8) -- (2) -- (7) -- (8) -- (3);
\draw (4) -- (9) -- (6);
\node[name=10] at (1/2,{sqrt(3)/2}) {};
\node[name=11] at (-1/2,{sqrt(3)/2}) {};
\node[fill=white, minimum size=0.4em] at (1/2,{sqrt(3)/2}) {};
\node[fill=white, minimum size=0.4em] at (-1/2,{sqrt(3)/2}) {};
\draw[dashed] (2) -- (10) -- (1) -- (11) -- (10) -- (6);
\draw[dashed] (11) -- (9);
\end{tikzpicture}
\end{center}
Observe that $F(9,15,22)$ is unit-distance only if we can arrange for either $|x| = 1$ or $|x'| = 1$.  But since $|x - (-5/2,-\sqrt{3}/2)| = 1$ and $|x' - (-5/2,\sqrt{3}/2)| = 1$, we see both $|x| > 1$ and $|x'| > 1$.  Hence, $F(9,15,22)$ is not unit-distance.
\end{proof}

\newpage

\begin{lem}
The following graph is forbidden:
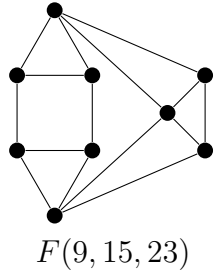
\begin{center}
\begin{tikzpicture}
\node[name=1] at (0,0) {};
\node[name=2] at (1,0) {};
\node[name=3] at (0,1) {};
\node[name=4] at (1,1) {};
\node[name=5] at (1/2,{1 + sqrt(3)/2}) {};
\node[name=6] at (1/2,{-sqrt(3)/2}) {};
\node[name=7] at (2,1/2) {};
\node[name=8] at (5/2,1) {};
\node[name=9] at (5/2,0) {};
\draw (1) -- (2) -- (4) -- (3) -- (1) (8) -- (9);
\draw (1) -- (6) -- (2) (7) -- (6) -- (9) -- (7);
\draw (3) -- (5) -- (4) (7) -- (5) -- (8) -- (7);
\end{tikzpicture}

$F(9,15,23)$
\end{center}
\end{lem}
\begin{proof}
Fixing some coordinates for a unit-distance subgraph of $F(9,15,23)$, we consider the following two classes of embeddings:
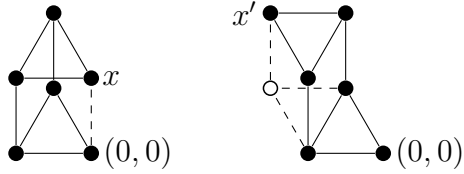
\begin{center}
\begin{tikzpicture}
\node[name=1] at (0,0) {};
\node[name=2,label=right:{$(0,0)$}] at (1,0) {};
\node[name=3] at (1/2,{sqrt(3)/2}) {};
\node[name=4] at (0,1) {};
\node[name=5,label=right:{$x$}] at (1,1) {};
\node[name=6] at (1/2,{1 + sqrt(3)/2}) {};
\draw (6) -- (4) -- (1) -- (2) -- (3) -- (6) -- (5) (1) -- (3);
\draw (4) -- (5);
\draw[dashed] (2) -- (5);
\end{tikzpicture}
\hspace{1em}
\begin{tikzpicture}
\node[name=1] at (0,0) {};
\node[name=2,label=right:{$(0,0)$}] at (1,0) {};
\node[name=3] at (1/2,{sqrt(3)/2}) {};
\node[name=4] at (0,1) {};
\node[name=5,label=left:{$x'$}] at (-1/2,{1 + sqrt(3)/2}) {};
\node[name=6] at (1/2,{1 + sqrt(3)/2}) {};
\node[name=7] at (-1/2,{sqrt(3)/2}) {};
\node[fill=white,minimum size=0.4em] at (-1/2,{sqrt(3)/2}) {};
\draw (6) -- (4) -- (1) -- (2) -- (3) -- (6) -- (5) (1) -- (3);
\draw (4) -- (5);
\draw[dashed] (1) -- (7) -- (5) (3) -- (7);
\end{tikzpicture}
\end{center}
Observe that $F(9,15,23)$ is unit-distance only if we can arrange for either $|x| = 2$ or $|x'| = 2$; of course, we already know $|x| = 1$.  Suppose that $|x'| = 2$.  Together with $|x' - (-3/2,-\sqrt{3}/2)| = 1$, we must have either $x' = (-1,\sqrt{3})$ or $x' = (-2,0)$.  The first case results in two distinct vertices overlapping at $(-1/2,\sqrt{3}/2)$, while the second case results in two distinct vertices overlapping at $(-1,0)$.  Neither allows for an embedding of $F(9,15,23)$, so it cannot be unit-distance.
\end{proof}

\begin{lem}
The following graph is forbidden:
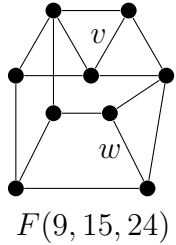
\begin{center}
\begin{tikzpicture}
\node[name=1] at (0,0) {};
\node[name=2] at (1,0) {};
\node[name=3] at (1/2,{sqrt(3)/2}) {};
\node[name=4] at (0,-3/2) {};
\node[name=6] at (1/2,-1/2) {};
\node[name=7] at (2,0) {};
\node[name=8] at (3/2,{sqrt(3)/2}) {};
\node[name=5] at (5/4,{-1/2}) {};
\node[name=9] at (1.75,-3/2) {};
\draw (6) -- (4) -- (1) -- (2) -- (3) -- (6) (1) -- (3);
\draw (8) -- (2) -- (7) -- (8) -- (3);
\draw (6) -- (5) -- (9) -- (4);
\draw (5) -- (7) -- (9);
\node[fill=none] at (1.1,{sqrt(3)/4}) {$v$};
\node[fill=none] at (5/4,{-1.1}) {$w$};
\end{tikzpicture}

$F(9,15,24)$
\end{center}
\end{lem}
\begin{proof}
Suppose, for a contradiction, that $F(9,15,24)$ were unit-distance.  By \cref{triangleRhombusLemma}(ii), the edges $v$ and $w$ must be parallel.  Without loss of generality, then, we can assume that a unit-distance embedding of $F(9,15,24)$ contains the following unit-distance subgraph:
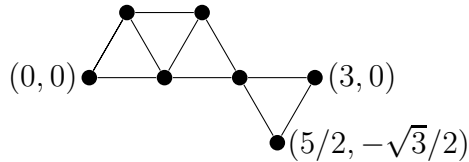
\begin{center}
\begin{tikzpicture}
\node[label=left:{$(0,0)$},name=1] at (0,0) {};
\node[name=2] at (1,0) {};
\node[name=3] at (1/2,{sqrt(3)/2}) {};
\node[name=7] at (2,0) {};
\node[name=8] at (3/2,{sqrt(3)/2}) {};
\node[name=5,label=right:{$(5/2,-\sqrt{3}/2)$}] at (5/2,{-sqrt(3)/2}) {};
\node[label=right:{$(3,0)$},name=9] at (3,0) {};
\draw (1) -- (2) -- (3) -- (1) -- (3);
\draw (8) -- (2) -- (7) -- (8) -- (3);
\draw (5) -- (9);
\draw (5) -- (7) -- (9);
\end{tikzpicture}
\end{center}
Observe, then, that $F(9,15,24)$ is unit-distance only if either $(3,0)$ or $(5/2,-\sqrt{3}/2)$ share a common unit-distance neighbor with the origin.  Of course, both lie at distance greater than 2 from the origin, so neither is possible and $F(9,15,24)$ must be forbidden.
\end{proof}

\begin{lem}
The following graph is forbidden:
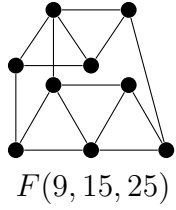
\begin{center}
\begin{tikzpicture}
\node[name=1] at (0,0) {};
\node[name=2] at (1,0) {};
\node[name=3] at (1/2,{sqrt(3)/2}) {};
\node[name=4] at (0,9/8) {};
\node[name=5] at (1,9/8) {};
\node[name=6] at (1/2,{1 + sqrt(3)/2}) {};
\node[name=7] at (2,0) {};
\node[name=8] at (3/2,{sqrt(3)/2}) {};
\node[name=9] at (3/2,{1 + sqrt(3)/2}) {};
\draw (6) -- (4) -- (1) -- (2) -- (3) -- (6) -- (5) (1) -- (3);
\draw (4) -- (5);
\draw (8) -- (2) -- (7) -- (8) -- (3);
\draw (6) -- (9) -- (5);
\draw (9) -- (7);
\end{tikzpicture}

$F(9,15,25)$
\end{center}
\end{lem}
\begin{proof}
Fixing some coordinates for a unit-distance subgraph of $F(9,15,25)$, we consider the following two classes of embeddings:
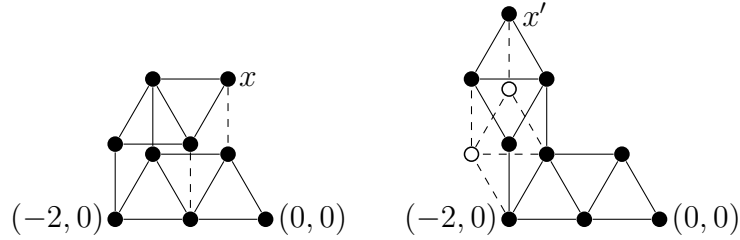
\begin{center}
\begin{tikzpicture}
\node[name=1,label=left:{$(-2,0)$}] at (0,0) {};
\node[name=2] at (1,0) {};
\node[name=3] at (1/2,{sqrt(3)/2}) {};
\node[name=4] at (0,1) {};
\node[name=5] at (1,1) {};
\node[name=6] at (1/2,{1 + sqrt(3)/2}) {};
\node[name=7,label=right:{$(0,0)$}] at (2,0) {};
\node[name=8] at (3/2,{sqrt(3)/2}) {};
\node[name=9,label=right:{$x$}] at (3/2,{1 + sqrt(3)/2}) {};
\draw (6) -- (4) -- (1) -- (2) -- (3) -- (6) -- (5) (1) -- (3);
\draw (4) -- (5);
\draw (8) -- (2) -- (7) -- (8) -- (3);
\draw (6) -- (9) -- (5);
\draw[dashed] (9) -- (8) (5) -- (2);
\end{tikzpicture}
\hspace{1em}
\begin{tikzpicture}
\node[name=1,label=left:{$(-2,0)$}] at (0,0) {};
\node[name=2] at (1,0) {};
\node[name=3] at (1/2,{sqrt(3)/2}) {};
\node[name=4] at (0,1) {};
\node[name=5] at (-1/2,{1 + sqrt(3)/2}) {};
\node[name=6] at (1/2,{1 + sqrt(3)/2}) {};
\node[name=7,label=right:{$(0,0)$}] at (2,0) {};
\node[name=8] at (3/2,{sqrt(3)/2}) {};
\node[name=9,label=right:{$x'$}] at (0,{1 + sqrt(3)}) {};
\node[name=10] at (-1/2,{sqrt(3)/2}) {};
\node[fill=white,minimum size=0.4em] at (-1/2,{sqrt(3)/2}) {};
\node[name=11] at (0,{sqrt(3)}) {};
\node[fill=white,minimum size=0.4em] at (0,{sqrt(3)}) {};
\draw (6) -- (4) -- (1) -- (2) -- (3) -- (6) -- (5) (1) -- (3);
\draw (4) -- (5);
\draw (8) -- (2) -- (7) -- (8) -- (3);
\draw (6) -- (9) -- (5);
\draw[dashed] (1) -- (10) -- (11) -- (3) -- (10) -- (5) (9) -- (11);
\end{tikzpicture}
\end{center}
Observe that $F(9,15,25)$ is unit-distance only if we can arrange for $|x| = 1$ or $|x'| = 1$.  If $|x| = 1$, then $x$ and $(-1,0)$ would have three common unit-distance neighbors, a contradiction.  Moreover, since $|x' - (-2,\sqrt{3})| = 1$, we see $|x'| > 1$.  In any case, $F(9,15,25)$ is forbidden.
\end{proof}

\begin{lem}
The following graph is forbidden:
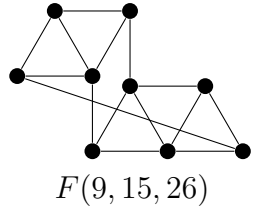
\begin{center}
\begin{tikzpicture}
\node[name=1] at (0,0) {};
\node[name=2] at (1,0) {};
\node[name=3] at (1/2,{sqrt(3)/2}) {};
\node[name=4] at (0,1) {};
\node[name=5] at (-1/2,{1 + sqrt(3)/2}) {};
\node[name=6] at (1/2,{1 + sqrt(3)/2}) {};
\node[name=7] at (2,0) {};
\node[name=8] at (3/2,{sqrt(3)/2}) {};
\node[name=9] at (-1,1) {};
\draw (6) -- (4) -- (1) -- (2) -- (3) -- (6) -- (5) (1) -- (3);
\draw (4) -- (5);
\draw (8) -- (2) -- (7) -- (8) -- (3);
\draw (4) -- (9) -- (5);
\draw (9) -- (7);
\end{tikzpicture}

$F(9,15,26)$
\end{center}
\end{lem}
\begin{proof}
Fixing some coordinates for a unit-distance subgraph of $F(9,15,26)$, we consider the following two classes of embeddings:
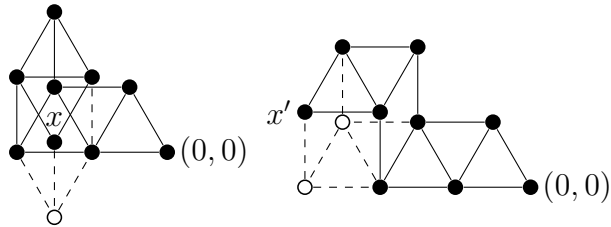
\begin{center}
\begin{tikzpicture}
\node[name=1] at (0,0) {};
\node[name=2] at (1,0) {};
\node[name=3] at (1/2,{sqrt(3)/2}) {};
\node[name=4] at (0,1) {};
\node[name=5] at (1,1) {};
\node[name=6] at (1/2,{1 + sqrt(3)/2}) {};
\node[name=7,label=right:{$(0,0)$}] at (2,0) {};
\node[name=8] at (3/2,{sqrt(3)/2}) {};
\node[name=9] at (1/2,{1 - sqrt(3)/2}) {};
\node[fill=none] at (0.5,0.35) {$x$};
\node[name=10] at (1/2,{-sqrt(3)/2}) {};
\node[fill=white,minimum size=0.4em] at (1/2,{-sqrt(3)/2}) {};
\draw (6) -- (4) -- (1) -- (2) -- (3) -- (6) -- (5) (1) -- (3);
\draw (4) -- (5);
\draw (8) -- (2) -- (7) -- (8) -- (3);
\draw (4) -- (9) -- (5);
\draw[dashed] (1) -- (10) -- (2) -- (5) (9) -- (10);
\end{tikzpicture}
\begin{tikzpicture}
\node[name=1] at (0,0) {};
\node[name=2] at (1,0) {};
\node[name=3] at (1/2,{sqrt(3)/2}) {};
\node[name=4] at (0,1) {};
\node[name=5] at (-1/2,{1 + sqrt(3)/2}) {};
\node[name=6] at (1/2,{1 + sqrt(3)/2}) {};
\node[name=7,label=right:{$(0,0)$}] at (2,0) {};
\node[name=8] at (3/2,{sqrt(3)/2}) {};
\node[name=9,label=left:{$x'$}] at (-1,1) {};
\node[name=10] at (-1/2,{sqrt(3)/2}) {};
\node[fill=white,minimum size=0.4em] at (-1/2,{sqrt(3)/2}) {};
\node[name=11] at (-1,0) {};
\node[fill=white,minimum size=0.4em] at (-1,0) {};
\draw (6) -- (4) -- (1) -- (2) -- (3) -- (6) -- (5) (1) -- (3);
\draw (4) -- (5);
\draw (8) -- (2) -- (7) -- (8) -- (3);
\draw (4) -- (9) -- (5);
\draw[dashed] (1) -- (10) -- (11) -- (1) (3) -- (10) -- (5) (11) -- (9);
\end{tikzpicture}
\end{center}
Observe that $F(9,15,26)$ is unit-distance only if we can arrange for $|x| = 1$ or $|x'| = 1$.  As $|x' - (-3,0)| = 1$, we see $|x'| > 1$.  Suppose, for a contradiction, that $|x| = 1$.  Then $x$ is a common unit-distance neighbor of $(-3/2,-\sqrt{3}/2)$ distinct from $(-1,0)$, in which case $x = (-1/2,-\sqrt{3}/2)$.  This leads to several vertices occupying the same location, violating our assumption that this drawing is an embedding.  Hence, $F(9,15,26)$ is forbidden.
\end{proof}

\newpage

\begin{lem}
The following graph is forbidden:

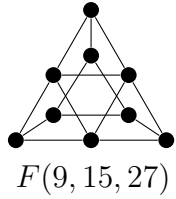
\begin{center}
\begin{tikzpicture}
\node[name=1] at (0,0) {};
\node[name=2] at (1,0) {};
\node[name=3] at (2,0) {};
\node[name=4] at (1/2,{sqrt(3)/2}) {};
\node[name=5] at (3/2,{sqrt(3)/2}) {};
\node[name=6] at (1,{sqrt(3)}) {};
\draw (1) -- (2) -- (3) -- (5) -- (2) -- (4) -- (1);
\draw (4) -- (5) -- (6) -- (4);
\node[name=7] at (1/2,1/3) {};
\node[name=8] at (3/2,1/3) {};
\node[name=9] at (1,9/8) {};
\draw (1) -- (7) -- (8) -- (9) -- (7);
\draw (8) -- (3);
\draw (9) -- (6);
\end{tikzpicture}

$F(9,15,27)$
\end{center}

\end{lem}

\begin{proof}
Fix some coordinates for the following rigid unit-distance subgraph of $F(9,15,27)$:
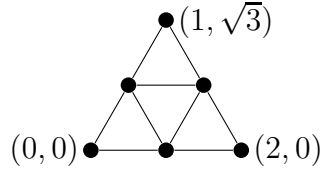
\begin{center}
\begin{tikzpicture}
\node[name=1,label=left:{$(0,0)$}] at (0,0) {};
\node[name=2] at (1,0) {};
\node[name=3,label=right:{$(2,0)$}] at (2,0) {};
\node[name=4] at (1/2,{sqrt(3)/2}) {};
\node[name=5] at (3/2,{sqrt(3)/2}) {};
\node[name=6,label=right:{$(1,\sqrt{3})$}] at (1,{sqrt(3)}) {};
\draw (1) -- (2) -- (3) -- (5) -- (2) -- (4) -- (1);
\draw (4) -- (5) -- (6) -- (4);
\end{tikzpicture}
\end{center}
Then $F(9,15,27)$ is unit-distance only if we can find three points $(x_1,x_2),(y_1,y_2),(z_1,z_2) \in \mathbf{R}^2$, distinct from the vertices depicted above, satisfying the following system of equations:
\begin{align*}
x_1^2 + x_2^2 &= 1\\
y_1^2 + y_2^2 &= 1\\
z_1^2 + z_2^2 &= 1\\
(x_1 + y_1 - 2)^2 + (x_2 + y_2)^2 &= 1\\
(x_1 + z_1 - 1)^2 + (x_2 + z_2 - \sqrt{3})^2 &= 1\\
(y_1 - z_1)^2 + (y_2 - z_2)^2 &= 1
\end{align*}
There are no solutions to this system except for those involving the occupied points $(1,0)$, $(1/2,\sqrt{3}/2)$, and $(3/2,\sqrt{3}/2)$.  Hence, $F(9,15,27)$ is forbidden.
\end{proof}

\begin{lem}
The following graph is forbidden:
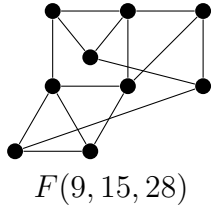
\begin{center}
\begin{tikzpicture}
\node[name=1] at (0,0) {};
\node[name=2] at (1,0) {};
\node[name=3] at (1/2,{sqrt(3)/2}) {};
\node[name=4] at (3/2,{sqrt(3)/2}) {};
\node[name=5] at (1/2,{1 + sqrt(3)/2}) {};
\node[name=6] at (3/2,{1 + sqrt(3)/2}) {};
\draw (2) -- (3) -- (1) -- (2) -- (4) -- (3);
\draw (3) -- (5) -- (6) -- (4);
\node[name=7] at (1, 5/4) {};
\node[name=8] at (5/2,{sqrt(3)/2}) {};
\draw (5) -- (7) -- (6) (8) -- (7);
\node[name=9] at (5/2, {1 + sqrt(3)/2}) {};
\draw (6) -- (9) -- (8);
\draw (9) -- (4);
\draw (8) -- (1);
\end{tikzpicture}

$F(9,15,28)$
\end{center}
\end{lem}
\begin{proof}

Consider the following unit-distance subgraph of $F(9,15,28)$:

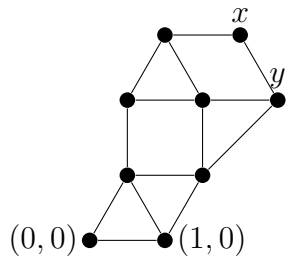
\begin{center}
\begin{tikzpicture}
\node[label=left:{$(0,0)$}, name=1] at (0,0) {};
\node[name=2,label=right:{$(1,0)$}] at (1,0) {};
\node[name=3] at (1/2,{sqrt(3)/2}) {};
\node[name=4] at (3/2,{sqrt(3)/2}) {};
\node[name=5] at (1/2,{1 + sqrt(3)/2}) {};
\node[name=6] at (3/2,{1 + sqrt(3)/2}) {};
\draw (2) -- (3) -- (1) -- (2) -- (4) -- (3);
\draw (3) -- (5) -- (6) -- (4);
\node[name=7] at (1, {1 + sqrt(3)}) {};
\node[label=above:{$x$},name=8] at (2, {1 + sqrt(3)}) {};
\draw (5) -- (7) -- (6) (8) -- (7);
\node[label=above:{$y$},name=9] at (5/2, {1 + sqrt(3)/2}) {};
\draw (6) -- (9) -- (8);
\draw (9) -- (4);
\end{tikzpicture}
\end{center}
By \cref{triangleRhombusLemma}, the edge from $x$ to $y$ must be parallel to either the edge from $(0,0)$ to $(1/2,\sqrt{3}/2)$ or the edge from $(1,0)$ to $(1/2,\sqrt{3}/2)$.  It follows that $F(9,15,28)$ is unit-distance only if we can arrange for $|x| = 1$ and either $y = x + (1/2,-\sqrt{3}/2)$ or $y = x + (1/2,\sqrt{3}/2)$.  Suppose that $|x| = 1$.  Since $|y - (3/2,\sqrt{3}/2)| = 1$, we see that $F(9,15,28)$ is unit-distance only if we can further arrange for either $|x - (1,\sqrt{3})| = 1$ or $|x - (1,0)| = 1$.  As the origin and $(1,\sqrt{3})$ have a unique unit-distance neighbor that already appears, the first case is impossible.  On the other hand, suppose for a contradiction that $|x - (1,0)| = 1$.  This can only be an embedding if $x = (1/2,-\sqrt{3}/2)$, which leads to two common unit-distance neighbors between $x$ and $(3/2,\sqrt{3}/2)$, namely $(1,0)$ and $y$.  But $|x - (3/2,\sqrt{3}/2)| = 2$, so this results in the desired contradiction.
\end{proof}

\begin{lem}
The following graph is forbidden:
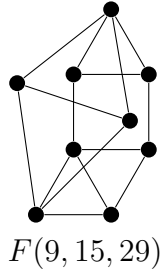
\begin{center}
\begin{tikzpicture}
\node[name=2] at (1,0) {};
\node[name=4] at (2,{1 + sqrt(3)}) {};
\node[name=5] at (3/2,{1 + sqrt(3)/2}) {};
\node[name=6] at (5/2,{1 + sqrt(3)/2}) {};
\node[name=7] at (2,0) {};
\node[name=8] at (3/2,{sqrt(3)/2}) {};
\node[name=9] at (5/2,{sqrt(3)/2}) {};
\draw (6) -- (4) (6) -- (5);
\draw (4) -- (5);
\draw (8) -- (2) -- (7) -- (8);
\draw (7) -- (9) -- (8);
\draw (5) -- (8);
\draw (6) -- (9);
\node[name=1] at (3/4,7/4) {};
\node[name=3] at (9/4,5/4) {};
\draw (4) -- (1) -- (2) -- (3) -- (4) (1) -- (3);
\end{tikzpicture}

$F(9,15,29)$
\end{center}
\end{lem}
\begin{proof}
Fixing some coordinates for a unit-distance subgraph of $F(9,15,29)$, we consider the following two classes of embeddings:
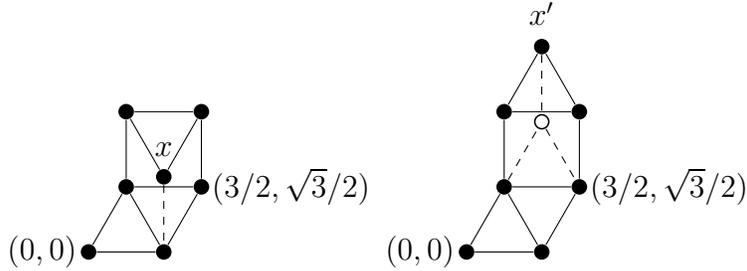
\begin{center}
\begin{tikzpicture}
\node[label=left:{$(0,0)$},name=2] at (1,0) {};
\node[fill=white,label=above:{$x$}] at (2,1.1) {};
\node[name=4] at (2,1) {};
\node[name=5] at (3/2,{1 + sqrt(3)/2}) {};
\node[name=6] at (5/2,{1 + sqrt(3)/2}) {};
\node[name=7] at (2,0) {};
\node[name=8] at (3/2,{sqrt(3)/2}) {};
\node[label=right:{$(3/2,\sqrt{3}/2)$},name=9] at (5/2,{sqrt(3)/2}) {};
\draw (6) -- (4) (6) -- (5);
\draw (4) -- (5);
\draw (8) -- (2) -- (7) -- (8);
\draw (7) -- (9) -- (8);
\draw (5) -- (8);
\draw (6) -- (9);
\draw[dashed] (4) -- (7);
\end{tikzpicture}
\begin{tikzpicture}
\node[label=left:{$(0,0)$},name=2] at (1,0) {};
\node[fill=white,label=above:{$x'$}] at (2,{1.1 + sqrt(3)}) {};
\node[name=4] at (2,{1 + sqrt(3)}) {};
\node[name=5] at (3/2,{1 + sqrt(3)/2}) {};
\node[name=6] at (5/2,{1 + sqrt(3)/2}) {};
\node[name=7] at (2,0) {};
\node[name=8] at (3/2,{sqrt(3)/2}) {};
\node[name=9,label=right:{$(3/2,\sqrt{3}/2)$}] at (5/2,{sqrt(3)/2}) {};
\draw (6) -- (4) (6) -- (5);
\draw (4) -- (5);
\draw (8) -- (2) -- (7) -- (8);
\draw (7) -- (9) -- (8);
\draw (5) -- (8);
\draw (6) -- (9);
\node[name=10] at (2,{sqrt(3)}) {};
\draw[dashed] (4) -- (10);
\node[fill=white,minimum size=.4em] at (2,{sqrt(3)}) {};
\draw[dashed] (8) -- (10) -- (9);
\end{tikzpicture}
\end{center}
Then $F(9,15,29)$ is unit-distance only if we can arrange for either $|x| = \sqrt{3}$ or $|x'| = \sqrt{3}$.  Suppose, for a contradiction, that $|x| = \sqrt{3}$.  Together with $|x - (1,0)| = 1$, this can only be an embedding if $x = (3/2,-\sqrt{3}/2)$.  This would lead to two points of distance 2 apart, namely $x$ and $(1/2,\sqrt{3}/2)$, with two common unit-distance neighbors, a contradiction.  Suppose instead that $|x'| = \sqrt{3}$.  Arguing similarly, this forces $x' = (0,\sqrt{3})$, leading to three common neighbors between $x'$ and $(3/2,\sqrt{3}/2)$, again a contradiction.  In any case, we see $F(9,15,29)$ is forbidden.
\end{proof}

\begin{lem}
The following graph is forbidden:
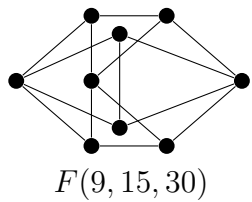
\begin{center}
\begin{tikzpicture}
\node[name=1] at (-1/2,0) {};
\node[name=2] at (1/2,0) {};
\node[name=3] at (5/2,0) {};
\node[name=4] at (1/2, {sqrt(3)/2}) {};
\node[name=5] at (1/2, {-sqrt(3)/2}) {};
\node[name=6] at (3/2, {sqrt(3)/2}) {};
\node[name=7] at (3/2, {-sqrt(3)/2}) {};
\draw (1) -- (4) -- (2) -- (5) -- (1);
\draw (3) -- (6) -- (2) -- (7) -- (3);
\draw (4) -- (6) (5) -- (7);
\node[name=8] at (7/8,5/8) {};
\node[name=9] at (7/8,-5/8) {};
\draw (8) -- (9) -- (3) -- (8) -- (1) -- (9);
\end{tikzpicture}

$F(9,15,30)$
\end{center}
\end{lem}

\begin{proof}

Observe that $F(9,15,30)$ contains the following unit-distance graph:

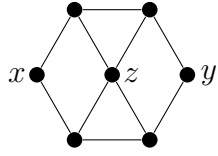
\begin{center}
\begin{tikzpicture}
\node[label=left:{$x$},name=1] at (0,0) {};
\node[name=2,label=right:{$z$}] at (1,0) {};
\node[label=right:{$y$},name=3] at (2,0) {};
\node[name=4] at (1/2, {sqrt(3)/2}) {};
\node[name=5] at (1/2, {-sqrt(3)/2}) {};
\node[name=6] at (3/2, {sqrt(3)/2}) {};
\node[name=7] at (3/2, {-sqrt(3)/2}) {};
\draw (1) -- (4) -- (2) -- (5) -- (1);
\draw (3) -- (6) -- (2) -- (7) -- (3);
\draw (4) -- (6) (5) -- (7);
\end{tikzpicture}
\end{center}
As in \cref{collinearLemma}, we can fix an embedding with $x = (a,0)$ with $a < 0$, $z = (0,0)$, and $y = (b,0)$ with $b > 0$, where $a = b/2 - \sqrt{3 - 3b^2/4}$.  Then $F(9,15,30)$ is unit-distance only if we can arrange for $b - a = b/2 + \sqrt{3 - 3b^2/4} = \sqrt{3}$.  The two solutions of this equation occur only when $b = 0$ or when $a = 0$, but neither of these are consistent with our setup.  It follows that $F(9,15,30)$ is forbidden.
\end{proof}

\begin{lem}
The following graph is forbidden:
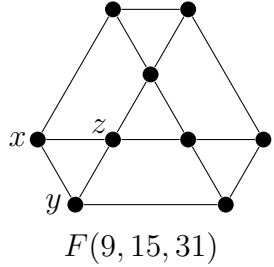
\begin{center}
\begin{tikzpicture}
\node[label=left:{$x$},name=1] at (0,0) {};
\node[label=above left:{$z$},name=2] at (1,0) {};
\node[name=3] at (2,0) {};
\node[name=4] at (3,0) {};
\node[label=left:{$y$},name=5] at (1/2,-{sqrt(3)/2}) {};
\node[name=6] at (5/2,-{sqrt(3)/2}) {};
\node[name=7] at (3/2,{sqrt(3)/2}) {};
\node[name=8] at (1,{sqrt(3)}) {};
\node[name=9] at (2,{sqrt(3)}) {};
\node[fill=white] at (1/2,0) {};
\node[fill=white] at (1,-.75) {};
\draw (1) -- (2) -- (3) -- (4) -- (6) -- (3) -- (7) -- (2) -- (5) -- (1);
\draw (7) -- (8) -- (9) -- (7);
\draw (1) -- (8) (9) -- (4) (5) -- (6);
\end{tikzpicture}

$F(9,15,31)$
\end{center}
\end{lem}

\begin{proof}
Suppose, for a contradiction, that $F(9,15,31)$ were unit-distance.  By considering the equilateral triangle on vertices $\{x,y,z\}$, we can assume that the directed edge from $z \rightarrow y$ can be obtained from the directed edge $z \rightarrow x$ by a rotation of $\pi/3$.  On the other hand, by considering the angles between the edges of the three rhombi, we see that $z \rightarrow y$ can be obtained from $z \rightarrow x$ by a rotation of $\pm \pi/3$ followed by another rotation of $\pm \pi/3$.  As rotation by $\pi/3$ is not equivalent to rotation by any of $-2\pi/3, 0$ or $2\pi/3$, we have arrived at the desired contradiction.
\end{proof}

\begin{lem}
The following graph is forbidden:
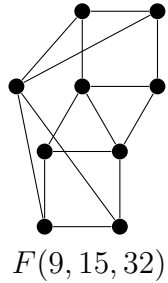
\begin{center}
\begin{tikzpicture}
\node[name=1] at (0,0) {};
\node[name=2] at (1,0) {};
\node[name=3] at (1/2,{sqrt(3)/2}) {};
\node[name=4] at (3/2,{sqrt(3)/2}) {};
\node[name=5] at (0,-1) {};
\node[name=6] at (1,-1) {};
\node[name=7] at (1/2,{1 + sqrt(3)/2}) {};
\node[name=8] at (3/2,{1 + sqrt(3)/2}) {};
\draw (1) -- (2) -- (3) -- (4) -- (2) (3) -- (1);
\draw (1) -- (5) -- (6) -- (2);
\draw (3) -- (7) -- (8) -- (4);
\node[name=9] at (-3/8,{sqrt(3)/2}) {};
\draw (7) -- (9) -- (8);
\draw (5) -- (9) --  (6);
\end{tikzpicture}

$F(9,15,32)$
\end{center}
\end{lem}

\newpage

\begin{proof}
Fixing some coordinates for a unit-distance subgraph of $F(9,15,32)$, we consider the following two classes of embeddings:
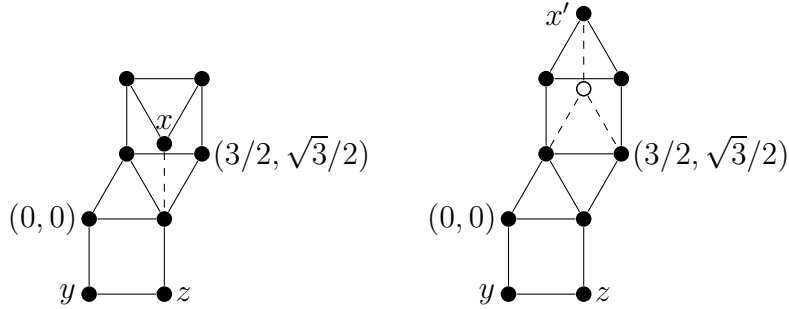
\begin{center}
\begin{tikzpicture}
\node[label=left:{$(0,0)$},name=1] at (0,0) {};
\node[name=2] at (1,0) {};
\node[name=3] at (1/2,{sqrt(3)/2}) {};
\node[label=right:{$(3/2,\sqrt{3}/2)$},name=4] at (3/2,{sqrt(3)/2}) {};
\node[label=left:{$y$},name=5] at (0,-1) {};
\node[label=right:{$z$},name=6] at (1,-1) {};
\node[name=7] at (1/2,{1 + sqrt(3)/2}) {};
\node[name=8] at (3/2,{1 + sqrt(3)/2}) {};
\draw (1) -- (2) -- (3) -- (4) -- (2) (3) -- (1);
\draw (1) -- (5) -- (6) -- (2);
\draw (3) -- (7) -- (8) -- (4);
\node[name=9] at (1,1) {};
\node[fill=none] at (1,1.2) {$x$};
\draw (8) -- (9) -- (7);
\draw[dashed] (9) -- (2);
\end{tikzpicture}
\hspace{1em}
\begin{tikzpicture}
\node[label=left:{$(0,0)$},name=1] at (0,0) {};
\node[name=2] at (1,0) {};
\node[name=3] at (1/2,{sqrt(3)/2}) {};
\node[label=right:{$(3/2,\sqrt{3}/2)$},name=4] at (3/2,{sqrt(3)/2}) {};
\node[label=left:{$y$},name=5] at (0,-1) {};
\node[label=right:{$z$},name=6] at (1,-1) {};
\node[name=7] at (1/2,{1 + sqrt(3)/2}) {};
\node[name=8] at (3/2,{1 + sqrt(3)/2}) {};
\draw (1) -- (2) -- (3) -- (4) -- (2) (3) -- (1);
\draw (1) -- (5) -- (6) -- (2);
\draw (3) -- (7) -- (8) -- (4);
\node[label=left:{$x'$},name=9] at (1,{1 + sqrt(3)}) {};
\draw (8) -- (9) -- (7);
\node[name=10] at (1,{sqrt(3)}) {};
\node[fill=white,minimum size=.4em] at (1,{sqrt(3)}) {};
\draw[dashed] (3) -- (10) -- (4) (10) -- (9);
\end{tikzpicture}
\end{center}
Then $F(9,15,32)$ is unit-distance only if we can arrange for $|x - y| = |x - z| = 1$ or $|x' - y| = |x' - z| = 1$.  Observe that any common unit-distance neighbor of $y$ and $z$ must also lie distance 1 from either $(1/2,\sqrt{3}/2)$ or $(1/2,-\sqrt{3}/2)$.

Notice that $|x - (1/2,\sqrt{3}/2)| \neq 1$.  Otherwise, we would have two points with three common unit-distance neighbors.  Suppose, for a contradiction, that $|x - (1/2,-\sqrt{3}/2)| = 1$.  Together with $|x - (1,0)| = 1$, this can only be an embedding provided $x = (3/2,-\sqrt{3}/2)$.  But this leads to two common unit-distance neighbors between $(1/2,\sqrt{3}/2)$ and $x$ despite $|x - (1/2,\sqrt{3}/2)| = 2$; a contradiction.  Hence, there is no embedding with $|x - y| = |x - z| = 1$.

Now, $|x' - (1,\sqrt{3})| = 1$, so we see that $|x' - (1/2,-\sqrt{3}/2)| > 1$.  For the last remaining case, we suppose $|x' - (1/2,\sqrt{3}/2)| = 1$.  As $|x' - (1,\sqrt{3})| = 1$, this can only be an embedding if $x' = (0,\sqrt{3})$.  However, this is impossible since it would lead to two points, namely $x'$ and $(3/2,\sqrt{3}/2)$, with three common unit-distance neighbors.  Hence, there is no embedding with $|x' - y| = |x' - z| = 1$, and we conclude that $F(9,15,32)$ is forbidden.
\end{proof}

\begin{lem}
The following graph is forbidden:
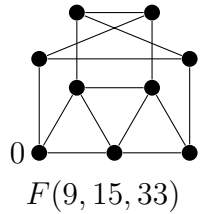
\begin{center}
\begin{tikzpicture}
\node[label=left:{$0$}, name=1] at (0,0) {};
\node[name=2] at (1,0) {};
\node[name=3] at (1/2,{sqrt(3)/2}) {};
\node[name=4] at (3/2,{sqrt(3)/2}) {};
\node[name=5] at (2,0) {};
\node[name=6] at (1/2,{1 + sqrt(3)/2}) {};
\node[name=7] at (3/2,{1 + sqrt(3)/2}) {};
\node[name=8] at (2,5/4) {};
\node[name=9] at (0,5/4) {};
\draw (1) -- (2) -- (3) -- (4) -- (5) -- (2) -- (4);
\draw (1) -- (3) -- (6) -- (7) -- (4);
\draw (6) -- (8) -- (9) -- (7);
\draw (8) -- (5);
\draw (9) -- (1);
\end{tikzpicture}

$F(9,15,33)$
\end{center}

\end{lem}

\begin{proof}

Fixing some coordinates for a unit-distance subgraph of $F(9,15,33)$, we consider the following class of embeddings:

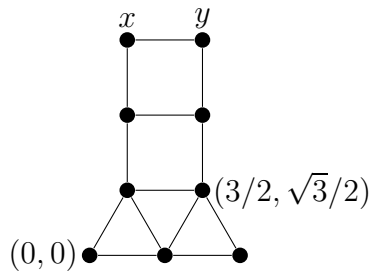
\begin{center}
\begin{tikzpicture}
\node[label=left:{$(0,0)$}, name=1] at (0,0) {};
\node[name=2] at (1,0) {};
\node[name=3] at (1/2,{sqrt(3)/2}) {};
\node[name=4,label=right:{$(3/2,\sqrt{3}/2)$}] at (3/2,{sqrt(3)/2}) {};
\node[name=5] at (2,0) {};
\node[name=6] at (1/2,{1 + sqrt(3)/2}) {};
\node[name=7] at (3/2,{1 + sqrt(3)/2}) {};
\node[name=8,label=above:{$x$}] at (1/2,{2 + sqrt(3)/2}) {};
\node[name=9,label=above:{$y$}] at (3/2,{2 + sqrt(3)/2}) {};
\draw (1) -- (2) -- (3) -- (4) -- (5) -- (2) -- (4);
\draw (1) -- (3) -- (6) -- (7) -- (4);
\draw (6) -- (8) -- (9) -- (7);
\end{tikzpicture}
\end{center}
Then $F(9,15,33)$ is unit-distance only if we can arrange for $|x - (2,0)| = 1$ and $|y| = 1$.  From \cref{triangleRhombusLemma}, we see that $y = x + (1,0)$, so we are considering whether it is possible for both $|y| = 1$ and $|y - (3,0)| = 1$.  Since the origin and $(3,0)$ share no common unit-distance neighbors, it follows that $F(9,15,33)$ is forbidden.
\end{proof}

\begin{lem}\label{n9-last}
The following graph is forbidden:
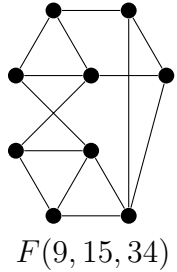
\begin{center}
\begin{tikzpicture}
\node[name=1] at (2,0) {};
\node[name=2] at (1,0) {};
\node[name=3] at (3/2,{sqrt(3)/2}) {};
\node[name=4] at (1/2,{sqrt(3)/2}) {};
\node[name=5] at (1/2,{1 + sqrt(3)/2}) {};
\node[name=6] at (3/2,{1 + sqrt(3)/2}) {};
\draw (2) -- (3) -- (1) -- (2) -- (4) -- (3);
\draw (3) -- (5) -- (6) -- (4);
\node[name=7] at (1, {1 + sqrt(3)}) {};
\node[name=8] at (2, {1 + sqrt(3)}) {};
\draw (5) -- (7) -- (6) (8) -- (7);
\node[name=9] at (5/2, {1 + sqrt(3)/2}) {};
\draw (6) -- (9) -- (8);
\draw (8) -- (1) -- (9);
\end{tikzpicture}

$F(9,15,34)$
\end{center}
\end{lem}

\begin{proof}
Fixing some coordinates for a unit-distance subgraph of $F(9,15,34)$, we consider the following two classes of embeddings:

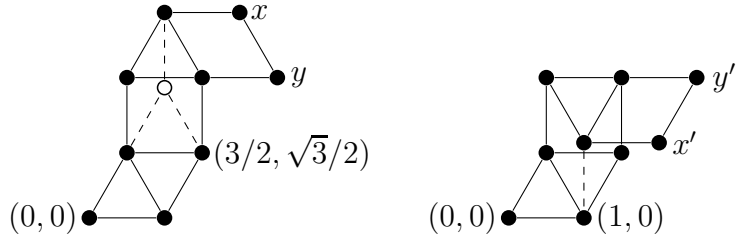
\begin{center}
\begin{tikzpicture}
\node[label=left:{$(0,0)$}, name=1] at (0,0) {};
\node[name=2] at (1,0) {};
\node[name=3] at (1/2,{sqrt(3)/2}) {};
\node[name=4,label=right:{$(3/2,\sqrt{3}/2)$}] at (3/2,{sqrt(3)/2}) {};
\node[name=5] at (1/2,{1 + sqrt(3)/2}) {};
\node[name=6] at (3/2,{1 + sqrt(3)/2}) {};
\draw (2) -- (3) -- (1) -- (2) -- (4) -- (3);
\draw (3) -- (5) -- (6) -- (4);
\node[name=7] at (1, {1 + sqrt(3)}) {};
\node[label=right:{$x$},name=8] at (2, {1 + sqrt(3)}) {};
\draw (5) -- (7) -- (6) (8) -- (7);
\node[label=right:{$y$},name=9] at (5/2, {1 + sqrt(3)/2}) {};
\draw (6) -- (9) -- (8);
\node[name=10] at (1,{sqrt(3)}) {};
\draw[dashed] (3) -- (10) -- (4);
\draw[dashed] (10) -- (7);
\node[fill=white,minimum size=.4em] at (1,{sqrt(3)}) {};
\end{tikzpicture}
\hspace{1em}
\begin{tikzpicture}
\node[label=left:{$(0,0)$}, name=1] at (0,0) {};
\node[name=2,label=right:{$(1,0)$}] at (1,0) {};
\node[name=3] at (1/2,{sqrt(3)/2}) {};
\node[name=4] at (3/2,{sqrt(3)/2}) {};
\node[name=5] at (1/2,{1 + sqrt(3)/2}) {};
\node[name=6] at (3/2,{1 + sqrt(3)/2}) {};
\draw (2) -- (3) -- (1) -- (2) -- (4) -- (3);
\draw (3) -- (5) -- (6) -- (4);
\node[name=7] at (1, {1}) {};
\node[label=right:{$x'$},name=8] at (2, 1) {};
\draw (5) -- (7) -- (6) (8) -- (7);
\node[label=right:{$y'$},name=9] at (5/2, {1 + sqrt(3)/2}) {};
\draw[dashed] (7) -- (2);
\draw (6) -- (9) -- (8);
\end{tikzpicture}
\end{center}
Then $F(9,15,34)$ is unit-distance only if we can arrange for $|x| = |y| = 1$ or $|x'| = |y'| = 1$.  From \cref{triangleRhombusLemma}, observe that $y = x + (1/2,-\sqrt{3}/2)$ and $y' = x' + (1/2,\sqrt{3}/2)$, respectively.

If we suppose $|x| = 1$ and $|x + (1/2,-\sqrt{3}/2)| = 1$, the only way for this to be an embedding is with $x = (-1,0)$.  However, in this case, it would be impossible for $x$ to share a common unit-distance neighbor with $(1,\sqrt{3})$.

If we suppose instead $|x'| = 1$ and $|x' + (1/2,\sqrt{3}/2)|$, then either $x' = (-1,0)$ or $x' = (1/2,-\sqrt{3}/2)$.  Observe that $x' \neq (-1,0)$, since otherwise $(-1,0)$ and $(1,0)$ would be distance 2 apart with two common unit-distance neighbors.  Similarly, $x' \neq (1/2,-\sqrt{3}/2)$ since $y'$ and $(1,0)$ must be distinct.  In any case, we see that $F(9,15,34)$ is forbidden.  Note that $F(9,15,34)$ also appears on the right in \cref{mysteryFigure}.
\end{proof}

\begin{thm}\label{n9theorem}
The set of minimal forbidden graphs on 9 vertices is given by
\[
\mathcal{F}_9 := \{F(9,13,i) : 1 \leq i \leq 2\} \cup \{F(9,14,i) : 1 \leq i \leq 19\} \cup \{F(9,15,i) : 1 \leq i \leq 34\}.
\]
\end{thm}

\begin{proof}
Lemmas \ref{n9-first} through \ref{n9-last}  establish that every graph contained in $\mathcal{F}_9$ is forbidden.  Moreover, no graph in $\mathcal{F}_9$ contains a proper subgraph isomorphic to any graph in $\mathcal{F}_{\leq 8}$ or $\mathcal{F}_9$.  Setting $
\mathcal{F}_{\leq 9} := \mathcal{F}_{\leq 8} \cup \mathcal{F}_9$, it suffices to show that every $\mathcal{F}_{\leq 9}$-free biconnected graph on 9 vertices is unit-distance.

\begin{figure}[t!]
\begin{center}
\begin{tabular}{ccc}
\begin{tikzpicture}[scale=2]
\node[name=1] at (0,0) {};
\node[name=2] at (1,0) {};
\node[name=3] at (2,0) {};
\node[name=4] at (1/2,{sqrt(3)/2}) {};
\node[name=5] at (3/2,{sqrt(3)/2}) {};
\node[name=6] at (.62836,.375878) {};
\node[name=7] at (1.72803,-.962305) {};
\node[name=8] at (.791893,-0.61066) {};
\node[name=9] at (1.56449,0.0242322) {};
\draw (1) -- (2) -- (3) -- (5) -- (4) -- (1);
\draw (4) -- (2) -- (5);
\draw (5) -- (6) -- (8) -- (7) -- (9) -- (6) (8) -- (9);
\draw (1) -- (8) (3) -- (7);
\end{tikzpicture}
&
\hspace{3em}
&
\begin{tikzpicture}[scale=2]
\node[name=1] at (0,0) {};
\node[name=2] at ({sqrt(3)/2},1/2) {};
\node[name=3] at ({sqrt(3)/2},-1/2) {};
\node[name=4] at ({sqrt(3)},0) {};
\node[name=5] at (0.332674,-0.943042) {};
\node[name=6] at (0.983035,-0.183417) {};
\node[name=7] at ({sqrt(3) - 0.399726},-0.916635) {};
\node[name=8] at ({sqrt(3) - 0.593966}, 0.80449) {};
\node[name=9] at (0.738359,-0.112145) {};
\draw (1) -- (2) -- (4) -- (3) -- (1) (2) -- (3);
\draw (4) -- (8) -- (9) -- (7) -- (4);
\draw (1) -- (5) -- (6) -- (1);
\draw (5) -- (7) (6) -- (8) (4) -- (9);
\end{tikzpicture}\\
$H_1$ & & $H_2$
\end{tabular}
\end{center}
\caption{$H_1$ and $H_2$ are the last two graphs that we verify to be unit-distance; $H_2$ also appears on the left in \cref{mysteryFigure}.  Exact coordinates for the vertices of the pictured embeddings are reported in \cref{h1table,h2table} in Appendix B.}\label{uglyFigure}
\end{figure}
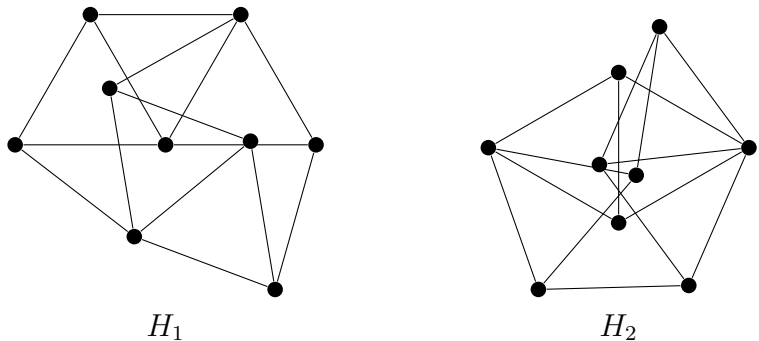

We use \texttt{nauty} to generate the set of all 194,066 biconnected graphs on 9 vertices.  \texttt{SageMath} computes that only 2984 of these graphs are $\mathcal{F}_{\leq 9}$-free, and all but 275 of these are subgraphs of the unit-distance graph $G_{27}$ that we constructed for the proof of \cref{n8theorem}.  For some of these remaining 275 graphs, we explicitly compute several coordinates.  We proceed to add new vertices to $G_{27}$ using the same procedure described in the proof of \cref{n8theorem} that we used to construct $G_{27}$.   In total, we add 91 vertices to $G_{27}$ to construct an embedded unit-distance graph $G_{118}$, and we report these additional vertices in \cref{g118table}.  \texttt{SageMath} verifies that there are exactly two $\mathcal{F}_{\leq 9}$-free biconnected graphs on 9 vertices, which we label $H_1$ and $H_2$ and depict in \cref{uglyFigure}, that are not subgraphs of $G_{118}$.  We were unable to produce exact coordinates for $H_1$ and $H_2$ using \texttt{SageMath}.  However, using the implementation of cylindrical algebraic decomposition~\cite{collins75} within \texttt{Mathematica}, we found exact coordinates for these graphs and report them in \cref{h1table,h2table}.
\end{proof}

\section*{Acknowledgements}

The authors are grateful to Dustin G. Mixon for several helpful discussions and Eric Weisstein for correcting inconsistencies between Appendix A and our \texttt{SageMath} worksheet.

\bibliographystyle{abbrv}
\bibliography{smallUD}

\newpage

\section*{Appendix A: Minimal forbidden graphs}

We collect here all 74 minimal forbidden graphs on up to 9 vertices.  The set of minimal forbidden graphs on up to 7 vertices was proved to be complete by Chilakamarri and Mahoney~\cite{chilakamarri95}.  Their work together with \cref{n8theorem,n9theorem} proves that this set of 74 minimal forbidden graphs on up to 9 vertices is complete.  These graphs can also be found in the \texttt{SageMath} worksheet that we have made available~\cite{smallUDcode}.

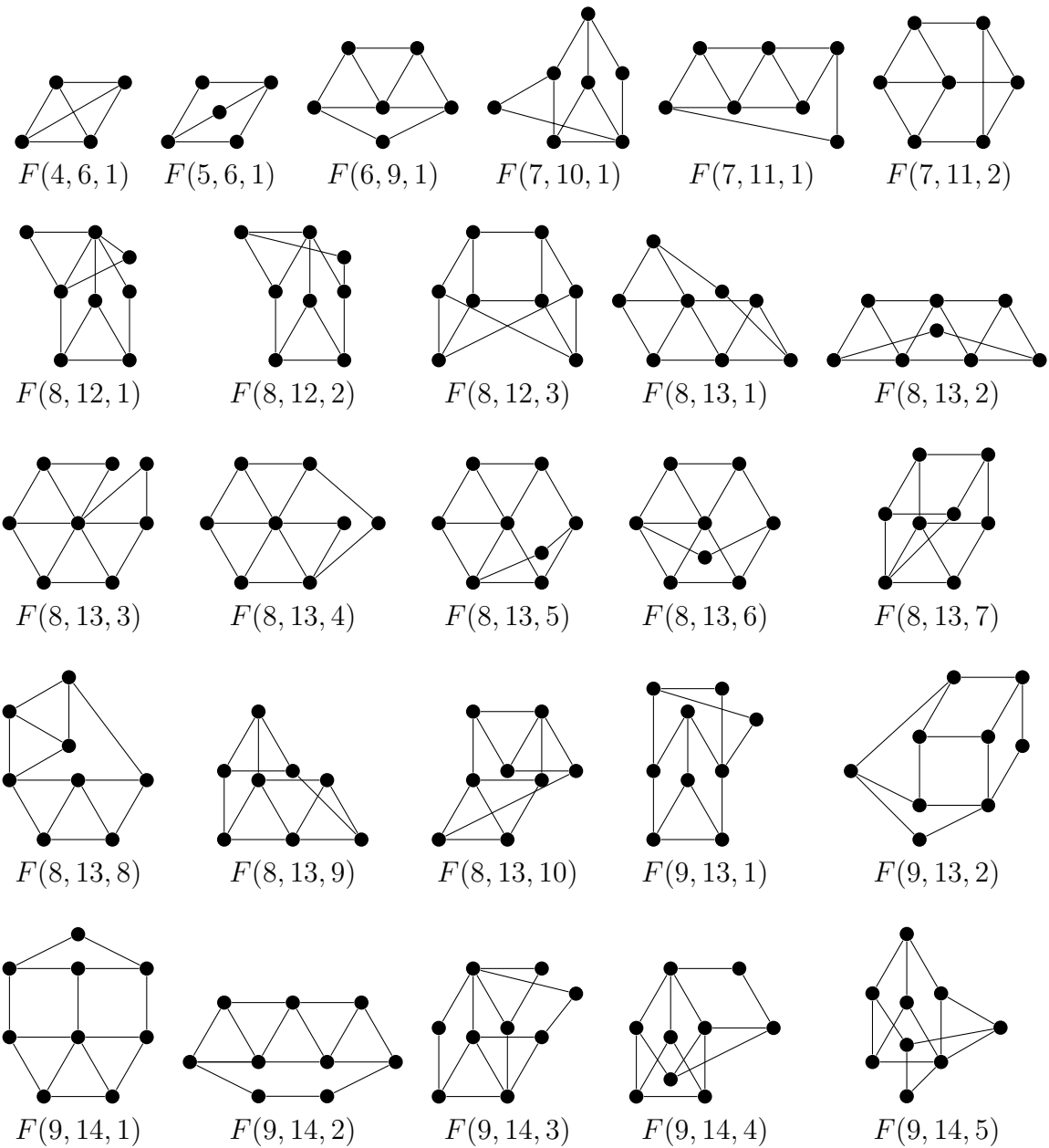
\begin{center}
\begin{tabular}{cccccc}
\begin{tikzpicture}
\node[name=1] at (0,0) {};
\node[name=2] at (1,0) {};
\node[name=3] at (1/2,{sqrt(3)/2}) {};
\node[name=4] at (3/2,{sqrt(3)/2}) {};
\draw (1) --(3) -- (2) -- (4) -- (1) -- (2) (3) -- (4);
\end{tikzpicture}
&
\begin{tikzpicture}
\node[name=1] at (0,0) {};
\node[name=2] at (1,0) {};
\node[name=3] at (1/2,{sqrt(3)/2}) {};
\node[name=4] at (3/2,{sqrt(3)/2}) {};
\node[name=5] at (3/4,{sqrt(3)/4}) {};
\draw (1) -- (2) -- (4) -- (3) -- (1) -- (5) -- (4);
\end{tikzpicture}
&
\begin{tikzpicture}
\node[name=1] at (0,0) {};
\node[name=2] at (1,0) {};
\node[name=3] at (2,0) {};
\node[name=4] at (1/2,{sqrt(3)/2}) {};
\node[name=5] at (3/2,{sqrt(3)/2}) {};
\node[name=6] at (1,-1/2) {};
\draw (4) -- (1) -- (2) -- (3) -- (5) -- (4) -- (2) -- (5) (1) -- (6) -- (3);
\end{tikzpicture}
&
\begin{tikzpicture}
\node[name=1] at (0,0) {};
\node[name=2] at (1,0) {};
\node[name=3] at (1/2,{sqrt(3)/2}) {};
\node[name=4] at (0,1) {};
\node[name=5] at (1,1) {};
\node[name=6] at (1/2,{1 + sqrt(3)/2}) {};
\node[name=7] at ({-sqrt(3)/2},1/2) {};
\draw (6) -- (4) -- (1) -- (2) -- (3) -- (6) -- (5) -- (2) (1) -- (3) (4) -- (7) -- (2);
\end{tikzpicture}
&
\begin{tikzpicture}
\node[name=1] at (0,0) {};
\node[name=2] at (1,0) {};
\node[name=3] at (2,0) {};
\node[name=4] at (1/2,{sqrt(3)/2}) {};
\node[name=5] at (3/2,{sqrt(3)/2}) {};
\node[name=6] at (5/2,{sqrt(3)/2}) {};
\node[name=7] at (5/2,-1/2) {};
\draw (4) -- (1) -- (2) -- (3) -- (5) -- (4) -- (2) -- (5) -- (6) -- (3);
\draw (1) -- (7) -- (6);
\end{tikzpicture}
&
\begin{tikzpicture}
\node[name=1] at (0,0) {};
\node[name=2] at (0,{sqrt(3)}) {};
\node[name=3] at (-1/2,{sqrt(3)/2}) {};
\node[name=4] at (1/2,{sqrt(3)/2}) {};
\node[name=5] at (3/2,{sqrt(3)/2}) {};
\node[name=6] at (1,0) {};
\node[name=7] at (1,{sqrt(3)}) {};
\draw (3) -- (1) -- (4) -- (2) -- (3) -- (4) -- (5) -- (7) -- (6) -- (1) (7) -- (2) (5) -- (6);
\end{tikzpicture}\\
$F(4,6,1)$ & $F(5,6,1)$ & $F(6,9,1)$ & $F(7,10,1)$ & $F(7,11,1)$ & $F(7,11,2)$
\end{tabular}
\vspace{1em}

\begin{tabular}{ccccc}
\begin{tikzpicture}
\node[name=1] at (0,0) {};
\node[name=2] at (1,0) {};
\node[name=3] at (1/2,{sqrt(3)/2}) {};
\node[name=4] at (0,1) {};
\node[name=5] at (1,1) {};
\node[name=6] at (1/2,{1 + sqrt(3)/2}) {};
\node[name=7] at (1,3/2) {};
\node[name=8] at (-1/2,{1 + sqrt(3)/2}) {};
\draw (6) -- (4) -- (1) -- (2) -- (3) -- (6) -- (5) -- (2) (1) -- (3);
\draw (4) -- (7) -- (6) -- (8) -- (4);
\end{tikzpicture}
&
\begin{tikzpicture}
\node[name=1] at (0,0) {};
\node[name=2] at (1,0) {};
\node[name=3] at (1/2,{sqrt(3)/2}) {};
\node[name=4] at (0,1) {};
\node[name=5] at (1,1) {};
\node[name=6] at (1/2,{1 + sqrt(3)/2}) {};
\node[name=7] at (1,3/2) {};
\node[name=8] at (-1/2,{1 + sqrt(3)/2}) {};
\draw (6) -- (4) -- (1) -- (2) -- (3) -- (6) -- (5) -- (2) (1) -- (3);
\draw (6) -- (8) -- (4);
\draw (8) -- (7) -- (5);
\end{tikzpicture}
&
\begin{tikzpicture}
\node[name=1] at (0,0) {};
\node[name=2] at (0,1) {};
\node[name=3] at (1/2,{sqrt(3)/2}) {};
\node[name=4] at (1/2,{sqrt(3)/2 + 1}) {};
\node[name=5] at (3/2,{sqrt(3)/2}) {};
\node[name=6] at (3/2,{sqrt(3)/2 + 1}) {};
\node[name=7] at (2,0) {};
\node[name=8] at (2,1) {};
\draw (1) -- (2) -- (4) -- (3) -- (1);
\draw (3) -- (5) -- (7) -- (8) -- (6) -- (4);
\draw (5) -- (6);
\draw (1) -- (8);
\draw (2) -- (7);
\end{tikzpicture}
&
\begin{tikzpicture}
\node[name=1] at (0,0) {};
\node[name=2] at (1,0) {};
\node[name=3] at (1/2,{sqrt(3)/2}) {};
\node[name=4] at (3/2,{sqrt(3)/2}) {};
\node[name=5] at (-1/2,{sqrt(3)/2}) {};
\node[name=6] at (0,{sqrt(3)}) {};
\node[name=7] at (2,0) {};
\node[name=8] at (1,1) {};
\draw (5) -- (1) (3) -- (1) -- (2) -- (3) -- (5) -- (6) -- (3) (2) -- (4) -- (3);
\draw (4) -- (7) -- (2);
\draw (7) -- (8) -- (6);
\end{tikzpicture}
&
\begin{tikzpicture}
\node[name=1] at (0,0) {};
\node[name=2] at (1,0) {};
\node[name=3] at (2,0) {};
\node[name=4] at (1/2,{sqrt(3)/2}) {};
\node[name=5] at (3/2,{sqrt(3)/2}) {};
\node[name=6] at (5/2,{sqrt(3)/2}) {};
\node[name=7] at (3,0) {};
\node[name=8] at (3/2,{sqrt(3)/4}) {};
\draw (4) -- (1) -- (2) -- (3) -- (5) -- (4) -- (2) -- (5) -- (6) -- (3);
\draw (3) -- (7) -- (6);
\draw (1) -- (8) -- (7);
\end{tikzpicture}\\
$F(8,12,1)$ & $F(8,12,2)$ & $F(8,12,3)$ & $F(8,13,1)$ & $F(8,13,2)$\\
\\

\begin{tikzpicture}
\node[name=1] at (0,0) {};
\node[name=2] at (1,0) {};
\node[name=3] at (1/2,{sqrt(3)/2}) {};
\node[name=4] at (3/2,{sqrt(3)/2}) {};
\node[name=5] at (-1/2,{sqrt(3)/2}) {};
\node[name=6] at (0,{sqrt(3)}) {};
\node[name=7] at (1,{sqrt(3)}) {};
\node[name=8] at (3/2,{sqrt(3)}) {};
\draw (5) -- (1) (3) -- (1) -- (2) -- (3) -- (5) -- (6) -- (7) -- (3) -- (6) (2) -- (4) -- (3) -- (8) -- (4);
\end{tikzpicture}
&
\begin{tikzpicture}
\node[name=1] at (0,0) {};
\node[name=2] at (1,0) {};
\node[name=3] at (1/2,{sqrt(3)/2}) {};
\node[name=4] at (3/2,{sqrt(3)/2}) {};
\node[name=5] at (-1/2,{sqrt(3)/2}) {};
\node[name=6] at (0,{sqrt(3)}) {};
\node[name=7] at (1,{sqrt(3)}) {};
\node[name=8] at (2,{sqrt(3)/2}) {};
\draw (5) -- (1) (3) -- (1) -- (2) -- (3) -- (5) -- (6) -- (7) -- (3) -- (6) (2) -- (4) -- (3);
\draw (2) -- (8) -- (7);
\end{tikzpicture}
&
\begin{tikzpicture}
\node[name=1] at (0,0) {};
\node[name=2] at (1,0) {};
\node[name=3] at (1/2,{sqrt(3)/2}) {};
\node[name=4] at (3/2,{sqrt(3)/2}) {};
\node[name=5] at (-1/2,{sqrt(3)/2}) {};
\node[name=6] at (0,{sqrt(3)}) {};
\node[name=7] at (1,{sqrt(3)}) {};
\node[name=8] at (1,{sqrt(3)/4}) {};
\draw (1) -- (8) -- (4);
\draw (5) -- (1) (3) -- (1) -- (2) -- (3) -- (5) -- (6) -- (7) -- (3) -- (6) (7) -- (4) -- (2);
\end{tikzpicture}
&
\begin{tikzpicture}
\node[name=1] at (0,0) {};
\node[name=2] at (1,0) {};
\node[name=3] at (1/2,{sqrt(3)/2}) {};
\node[name=4] at (3/2,{sqrt(3)/2}) {};
\node[name=5] at (-1/2,{sqrt(3)/2}) {};
\node[name=6] at (0,{sqrt(3)}) {};
\node[name=7] at (1,{sqrt(3)}) {};
\node[name=8] at (1/2,{sqrt(3)/2 - 1/2}) {};
\draw (5) -- (8) -- (4);
\draw (5) -- (1) (3) -- (1) -- (2) -- (3) -- (5) -- (6) -- (7) -- (3) -- (6) (7) -- (4) -- (2);
\end{tikzpicture}
&
\begin{tikzpicture}
\node[name=1] at (0,0) {};
\node[name=2] at (1,0) {};
\node[name=3] at (1/2,{sqrt(3)/2}) {};
\node[name=4] at (3/2,{sqrt(3)/2}) {};
\node[name=5] at (0,1) {};
\node[name=6] at (1,1) {};
\node[name=7] at (1/2,{sqrt(3)/2 + 1}) {};
\node[name=8] at (3/2,{sqrt(3)/2 + 1}) {};
\draw (3) -- (1) -- (2) -- (4) -- (3) -- (2);
\draw (7) -- (5) -- (6) -- (8) -- (7);
\draw (1) -- (5);
\draw (3) -- (7);
\draw (4) -- (8);
\draw (1) -- (6);
\end{tikzpicture}\\
$F(8,13,3)$ & $F(8,13,4)$ & $F(8,13,5)$ & $F(8,13,6)$ & $F(8,13,7)$\\
\\

\begin{tikzpicture}
\node[name=1] at (0,0) {};
\node[name=2] at (1,0) {};
\node[name=3] at (1/2,{sqrt(3)/2}) {};
\node[name=4] at (3/2,{sqrt(3)/2}) {};
\node[name=5] at (-1/2,{sqrt(3)/2}) {};
\node[name=6] at ({sqrt(3)/2 - 1/2}, {1/2 + sqrt(3)/2}) {};
\node[name=7] at (-1/2,{1 + sqrt(3)/2}) {};
\node[name=8] at ({-1/2 + sqrt(3)/2},{3/2 + sqrt(3)/2}) {};
\draw (1) -- (2) -- (4) -- (3) -- (5) -- (1) -- (3) -- (2);
\draw (5) -- (6) -- (8) -- (7) -- (5);
\draw (6) -- (7) (8) -- (4);
\end{tikzpicture}
&
\begin{tikzpicture}
\node[name=1] at (0,0) {};
\node[name=2] at (1,0) {};
\node[name=3] at (1/2,{sqrt(3)/2}) {};
\node[name=4] at (0,1) {};
\node[name=5] at (1,1) {};
\node[name=6] at (1/2,{1 + sqrt(3)/2}) {};
\node[name=7] at (2,0) {};
\node[name=8] at (3/2,{sqrt(3)/2}) {};
\draw (6) -- (4) -- (1) -- (2) -- (3) -- (6) -- (5) (1) -- (3);
\draw (4) -- (5);
\draw (8) -- (2) -- (7) -- (8) -- (3);
\draw (5) -- (7);
\end{tikzpicture}
&
\begin{tikzpicture}
\node[name=1] at (0,0) {};
\node[name=2] at (1,0) {};
\node[name=3] at (1/2,{sqrt(3)/2}) {};
\node[name=4] at (3/2,{sqrt(3)/2}) {};
\node[name=5] at (1/2,{sqrt(3)/2 + 1}) {};
\node[name=6] at (3/2,{sqrt(3)/2 + 1}) {};
\node[name=7] at (1,1) {};
\node[name=8] at (2,1) {};
\draw (1) -- (2) -- (4) -- (3) -- (2);
\draw (1) -- (3) -- (5) -- (6) -- (4);
\draw (5) -- (7) -- (8) -- (6) -- (7);
\draw (1) -- (8);
\end{tikzpicture}
&
\begin{tikzpicture}
\node[name=1] at (0,0) {};
\node[name=2] at (1,0) {};
\node[name=3] at (1/2,{sqrt(3)/2}) {};
\node[name=4] at (0,1) {};
\node[name=5] at (1,1) {};
\node[name=6] at (1/2,{1 + sqrt(3)/2}) {};
\draw (6) -- (4) -- (1) -- (2) -- (3) -- (6) -- (5) -- (2) (1) -- (3);
\node[name=7] at (0,2.2) {};
\node[name=8] at (1,2.2) {};
\node[name=9] at (3/2,1.75) {};
\draw (4) -- (7) -- (8) -- (5);
\draw (5) -- (9) -- (7);
\end{tikzpicture} 
&
\begin{tikzpicture}
\node[name=1] at (0,0) {};
\node[name=2] at (1,0) {};
\node[name=3] at (1,1) {};
\node[name=4] at (0,1) {};
\node[name=5] at (1/2,{1 + sqrt(3)/2}) {};
\node[name=6] at (3/2,{1 + sqrt(3)/2}) {};
\node[name=7] at (3/2,{sqrt(3)/2}) {};
\node[name=8] at (-1,1/2) {};
\node[name=9] at (0,-1/2) {};
\draw (1) -- (2) -- (3) -- (4) -- (1) -- (8) -- (5) -- (6) -- (7) -- (2);
\draw (3) -- (6) (4) -- (5) (2) -- (9) -- (8);
\end{tikzpicture}\\ 

$F(8,13,8)$ & $F(8,13,9)$ & $F(8,13,10)$ & $F(9,13,1)$ & $F(9,13,2)$\\
\\

\begin{tikzpicture}
\node[name=1] at (0,0) {};
\node[name=2] at (1/2,{sqrt(3)/2}) {};
\node[name=3] at (1,0) {};
\node[name=4] at (-1/2,{sqrt(3)/2}) {};
\node[name=5] at (3/2,{sqrt(3)/2}) {};
\node[name=6] at (-1/2,{1 + sqrt(3)/2}) {};
\node[name=7] at (1/2,{1 + sqrt(3)/2}) {};
\node[name=8] at (3/2,{1 + sqrt(3)/2}) {};
\node[name=9] at (1/2,{3/2 + sqrt(3)/2}) {};
\draw (1) -- (2) -- (3) -- (1);
\draw (2) -- (5) -- (3);
\draw (1) -- (4) -- (2);
\draw (4) -- (6) -- (7) -- (2);
\draw (5) -- (8) -- (7);
\draw (6) -- (9) -- (8);
\end{tikzpicture}
&
\begin{tikzpicture}
\node[name=1] at (0,0) {};
\node[name=2] at (1/2,{sqrt(3)/2}) {};
\node[name=3] at (1,0) {};
\node[name=4] at (2,0) {};
\node[name=5] at (3/2,{sqrt(3)/2}) {};
\node[name=6] at (3,0) {};
\node[name=7] at (5/2,{sqrt(3)/2}) {};
\draw (1) -- (2) -- (3) -- (1);
\draw (2) -- (5) -- (3);
\draw (1) -- (4);
\draw (4) -- (5) -- (7) -- (4) -- (6) -- (7);
\node[name=8] at (1,-1/2) {};
\node[name=9] at (2,-1/2) {};
\draw (1) -- (8) -- (9) -- (6);
\end{tikzpicture}
&
\begin{tikzpicture}
\node[name=1] at (0,0) {};
\node[name=2] at (1,0) {};
\node[name=3] at (1/2,{sqrt(3)/2}) {};
\node[name=4] at (0,1) {};
\node[name=5] at (1,1) {};
\node[name=6] at (1/2,{1 + sqrt(3)/2}) {};
\draw (6) -- (4) -- (1) -- (2) -- (3) -- (6) -- (5) -- (2) (1) -- (3);
\node[name=7] at (3/2,{sqrt(3)/2}) {};
\node[name=8] at (3/2,{1 + sqrt(3)/2}) {};
\draw (3) -- (7) -- (2);
\draw (6) -- (8) -- (5);
\node[name=9] at (2,3/2) {};
\draw (6) -- (9) -- (7);
\end{tikzpicture}
&
\begin{tikzpicture}
\node[name=1] at (0,0) {};
\node[name=2] at (1,0) {};
\node[name=3] at (1/2,{sqrt(3)/2}) {};
\node[name=4] at (0,1) {};
\node[name=5] at (1,1) {};
\node[name=6] at (1/2,{1 + sqrt(3)/2}) {};
\draw (6) -- (4) -- (1) -- (2) -- (3) -- (6) -- (5) -- (2) (1) -- (3);
\node[name=7] at (1/2,1/4) {};
\node[name=8] at (3/2,{1 + sqrt(3)/2}) {};
\node[name=9] at (2,1) {};
\draw (6) -- (8) -- (9) -- (5);
\draw (4) -- (7) -- (5);
\draw (7) -- (9);
\end{tikzpicture}
&
\begin{tikzpicture}
\node[name=1] at (0,0) {};
\node[name=2] at (1,0) {};
\node[name=3] at (1/2,{sqrt(3)/2}) {};
\node[name=4] at (0,1) {};
\node[name=5] at (1,1) {};
\node[name=6] at (1/2,{1 + sqrt(3)/2}) {};
\draw (6) -- (4) -- (1) -- (2) -- (3) -- (6) -- (5) -- (2) (1) -- (3);
\node[name=7] at ({1 + sqrt(3)/2},1/2) {};
\node[name=8] at (1/2,1/4) {};
\node[name=9] at (1/2,-1/2) {};
\draw (2) -- (7) -- (5);
\draw (7) -- (8) -- (4);
\draw (2) -- (9) -- (8);
\end{tikzpicture}\\
$F(9,14,1)$ & $F(9,14,2)$ & $F(9,14,3)$ & $F(9,14,4)$ & $F(9,14,5)$
\end{tabular}
\end{center}

\newpage

\begin{center}
\begin{tabular}{ccccc}
\begin{tikzpicture}
\node[name=1] at (0,0) {};
\node[name=2] at (1,0) {};
\node[name=3] at (1/2,{sqrt(3)/2}) {};
\node[name=4] at (0,1) {};
\node[name=5] at (1,1) {};
\node[name=6] at (1/2,{1 + sqrt(3)/2}) {};
\draw (6) -- (4) -- (1) -- (2) -- (3) -- (6) -- (5) -- (2) (1) -- (3);
\node[name=7] at (3/2,{sqrt(3)/2}) {};
\node[name=8] at (3/2,{1 + sqrt(3)/2}) {};
\draw (2) -- (7) -- (3);
\draw (7) -- (8) -- (5);
\node[name=9] at (1/2,9/4) {};
\draw (4) -- (9) -- (8);
\end{tikzpicture}
&
\begin{tikzpicture}
\node[name=1] at (0,0) {};
\node[name=2] at (1,0) {};
\node[name=3] at (1/2,{sqrt(3)/2}) {};
\node[name=4] at (0,1) {};
\node[name=5] at (1,1) {};
\node[name=6] at (1/2,{1 + sqrt(3)/2}) {};
\draw (6) -- (4) -- (1) -- (2) -- (3) -- (6) -- (5) -- (2) (1) -- (3);
\node[name=7] at (3/2,{sqrt(3)/2}) {};
\node[name=8] at (3/2,{1 + sqrt(3)/2}) {};
\draw (2) -- (7) -- (3);
\draw (7) -- (8) -- (6);
\node[name=9] at (1/2,9/4) {};
\draw (4) -- (9) -- (8);
\end{tikzpicture}
&
\begin{tikzpicture}
\node[name=1] at (0,0) {};
\node[name=2] at (1,0) {};
\node[name=3] at (1/2,{sqrt(3)/2}) {};
\node[name=4] at (0,1) {};
\node[name=5] at (1,1) {};
\node[name=6] at (1/2,{1 + sqrt(3)/2}) {};
\draw (6) -- (4) -- (1) -- (2) -- (3) -- (6) -- (5) -- (2) (1) -- (3);
\node[name=7] at ({1 + sqrt(3)/2},1/2) {};
\node[name=8] at (5/4,-3/8) {};
\node[name=9] at (1/2,1/4) {};
\draw (5) -- (7) -- (2);
\draw (7) -- (8) -- (1);
\draw (7) -- (9) -- (4);
\end{tikzpicture}
&
\begin{tikzpicture}
\node[name=1] at (0,0) {};
\node[name=2] at (1,0) {};
\node[name=3] at (1/2,{sqrt(3)/2}) {};
\node[name=4] at (0,1) {};
\node[name=5] at (1,1) {};
\node[name=6] at (1/2,{1 + sqrt(3)/2}) {};
\draw (6) -- (4) -- (1) -- (2) -- (3) -- (6) -- (5) -- (2) (1) -- (3);
\node[name=7] at (3/2,{1 + sqrt(3)/2}) {};
\node[name=8] at (1,{1 + sqrt(3)}) {};
\node[name=9] at (1/8,9/4) {};
\draw (5) -- (7) -- (6) -- (8) -- (7);
\draw (8) -- (9) -- (4);
\end{tikzpicture}
&
\begin{tikzpicture}
\node[name=1] at (0,0) {};
\node[name=2] at (1,0) {};
\node[name=3] at (1/2,{sqrt(3)/2}) {};
\node[name=4] at (0,1) {};
\node[name=5] at (1,1) {};
\node[name=6] at (1/2,{1 + sqrt(3)/2}) {};
\draw (6) -- (4) -- (1) -- (2) -- (3) -- (6) -- (5) -- (2) (1) -- (3);
\node[name=7] at (3/2,{1 + sqrt(3)/2}) {};
\draw (5) -- (7) -- (6);
\node[name=8] at (-1/2,{1 + sqrt(3)/2}) {};
\node[name=9] at (1/2,5/2) {};
\draw (4) -- (8) -- (6);
\draw (8) -- (9) -- (7);
\end{tikzpicture}\\
$F(9,14,6)$ & $F(9,14,7)$ & $F(9,14,8)$ & $F(9,14,9)$ & $F(9,14,10)$\\
\\
\begin{tikzpicture}
\node[name=1] at (0,0) {};
\node[name=2] at (1,0) {};
\node[name=3] at (1/2,{sqrt(3)/2}) {};
\node[name=4] at (0,1) {};
\node[name=5] at (1,1) {};
\node[name=6] at (1/2,{1 + sqrt(3)/2}) {};
\draw (6) -- (4) -- (1) -- (2) -- (3) -- (6) -- (5) -- (2) (1) -- (3);
\node[name=7] at (3/2,{1 + sqrt(3)/2}) {};
\draw (5) -- (7) -- (6);
\node[name=8] at (3/2,{sqrt(3)/2}) {};
\draw (2) -- (8) -- (3);
\node[name=9] at (2,3/2) {};
\draw (5) -- (9) -- (8);
\end{tikzpicture}
&
\begin{tikzpicture}
\node[name=1] at (0,0) {};
\node[name=2] at (1,0) {};
\node[name=3] at (1/2,{sqrt(3)/2}) {};
\node[name=4] at (0,1) {};
\node[name=5] at (1,1) {};
\node[name=6] at (1/2,{1 + sqrt(3)/2}) {};
\draw (6) -- (4) -- (1) -- (2) -- (3) -- (6) -- (5) -- (2) (1) -- (3);
\node[name=7] at (3/2,{1 + sqrt(3)/2}) {};
\node[name=8] at (3/2,3/4) {};
\node[name=9] at (1/2,1/4) {};
\draw (5) -- (7) -- (6);
\draw (4) -- (9) -- (5);
\draw (9) -- (8) -- (7);
\end{tikzpicture}
&
\begin{tikzpicture}
\node[name=1] at (0,0) {};
\node[name=2] at (1,0) {};
\node[name=3] at (1/2,{sqrt(3)/2}) {};
\node[name=4] at (0,1) {};
\node[name=5] at (1,1) {};
\node[name=6] at (1/2,{1 + sqrt(3)/2}) {};
\draw (6) -- (4) -- (1) -- (2) -- (3) -- (6) -- (5) -- (2) (1) -- (3);
\node[name=7] at (2,-1/4) {};
\node[name=8] at (2,3/4) {};
\node[name=9] at ({5/4},1/4) {};
\draw (2) -- (7) -- (8) -- (5);
\draw (7) -- (9) -- (8) (9) -- (4);
\end{tikzpicture}
&
\begin{tikzpicture}
\node[name=1] at (0,0) {};
\node[name=2] at (1,0) {};
\node[name=3] at (1/2,{sqrt(3)/2}) {};
\node[name=4] at (0,1) {};
\node[name=5] at (1,1) {};
\node[name=6] at (1/2,{1 + sqrt(3)/2}) {};
\draw (6) -- (4) -- (1) -- (2) -- (3) -- (6) -- (5) -- (2) (1) -- (3);
\node[name=7] at ({1 + sqrt(3)/2},1/2) {};
\node[name=8] at (-1/2,1/2) {};
\node[name=9] at (0,3/2) {};
\draw (2) -- (7) -- (5);
\draw (7) -- (8) -- (1);
\draw (5) -- (9) -- (8);
\end{tikzpicture}
&
\begin{tikzpicture}
\node[name=1] at (0,0) {};
\node[name=2] at (1,0) {};
\node[name=3] at (1/2,{sqrt(3)/2}) {};
\node[name=4] at (0,1) {};
\node[name=5] at (1,1) {};
\node[name=6] at (1/2,{1 + sqrt(3)/2}) {};
\draw (6) -- (4) -- (1) -- (2) -- (3) -- (6) -- (5) -- (2) (1) -- (3);
\node[name=7] at (3/2,{1 + sqrt(3)/2}) {};
\node[name=8] at (2,1) {};
\node[name=9] at (1/2,1/4) {};
\draw (6) -- (7) -- (8) -- (5) -- (7);
\draw (4) -- (9) -- (8);
\end{tikzpicture}\\
$F(9,14,11)$ & $F(9,14,12)$ & $F(9,14,13)$ & $F(9,14,14)$ & $F(9,14,15)$\\
\\
\begin{tikzpicture}
\node[name=1] at (0,0) {};
\node[name=2] at (1,0) {};
\node[name=3] at (1/2,{sqrt(3)/2}) {};
\node[name=4] at (3/2,{sqrt(3)/2}) {};
\node[name=5] at (0,1) {};
\node[name=6] at (1,1) {};
\node[name=7] at (1/2,{sqrt(3)/2 + 1}) {};
\node[name=8] at (3/2,{sqrt(3)/2 + 1}) {};
\draw (3) -- (1) -- (2) -- (4) -- (3) -- (2);
\draw (7) -- (5) -- (6) -- (8) -- (7);
\draw (1) -- (5);
\draw (3) -- (7);
\draw (4) -- (8);
\node[name=9] at (3/2,0) {};
\draw (3) -- (9) -- (6);
\end{tikzpicture} 
&
\begin{tikzpicture}
\node[name=1] at (0,0) {};
\node[name=2] at (1,0) {};
\node[name=3] at (1/2,{sqrt(3)/2}) {};
\node[name=4] at (3/2,{sqrt(3)/2}) {};
\node[name=5] at (1,{sqrt(3)}) {};
\node[name=6] at (0,{sqrt(3)}) {};
\node[name=7] at (-1/2,{sqrt(3)/2}) {};
\node[name=8] at (1,{3*sqrt(3)/2}) {};
\node[name=9] at (0,{3*sqrt(3)/2}) {};
\draw (1) -- (2) -- (3) -- (1) (2) -- (4) -- (5) -- (3) (1) -- (7) -- (6) -- (3) (4) -- (8) -- (9) -- (5) (7) -- (9) (6) -- (8);
\end{tikzpicture} 
&
\begin{tikzpicture}
\node[name=1] at (0,0) {};
\node[name=2] at (1,0) {};
\node[name=3] at (1/2,{sqrt(3)/2}) {};
\node[name=4] at (3/2,{sqrt(3)/2}) {};
\draw (1) -- (2) -- (4) -- (3) -- (1) (2) -- (3);
\node[name=5] at (0,5/4) {};
\node[name=6] at (1,5/4) {};
\node[name=7] at (3/2,{1 + sqrt(3)/2}) {};
\node[name=8] at (1/2,{sqrt(3)/4}) {};
\node[name=9] at (1/2,{1 + sqrt(3)/2}) {};
\draw (1) -- (5) -- (6) -- (2) (4) -- (7) -- (6);
\draw (7) -- (9) -- (1);
\draw (5) -- (8) -- (4);
\end{tikzpicture} 
&
\begin{tikzpicture}
\node[name=1] at (0,0) {};
\node[name=2] at (1,0) {};
\node[name=3] at (2,0) {};
\node[name=4] at (1/2, {sqrt(3)/2}) {};
\node[name=5] at (1/2, {-sqrt(3)/2}) {};
\node[name=6] at (3/2, {sqrt(3)/2}) {};
\node[name=7] at (3/2, {-sqrt(3)/2}) {};
\draw (1) -- (4) -- (2) -- (5) -- (1);
\draw (3) -- (6) -- (2) -- (7) -- (3);
\draw (4) -- (6) (5) -- (7);
\node[name=8] at (1,{sqrt(3)/4}) {};
\node[name=9] at (2,{sqrt(3)/4}) {};
\draw (1) -- (8) -- (3);
\draw (8) -- (9) -- (2);
\end{tikzpicture} 
&
\begin{tikzpicture}
\node[name=1] at (0,0) {};
\node[name=2] at (1/2,{sqrt(3)/2}) {};
\node[name=3] at (1,0) {};
\node[name=4] at (2,{sqrt(3)}) {};
\node[name=5] at (3/2,{sqrt(3)/2}) {};
\node[name=6] at (-1/2,{sqrt(3)/2}) {};
\node[name=7] at (5/2,{sqrt(3)/2}) {};
\node[name=8] at (1,{sqrt(3)}) {};
\node[name=9] at (2,0) {};
\draw (3) -- (1);
\draw (5) -- (3);
\draw (4) -- (5) -- (7) -- (4);
\draw (1) -- (6) -- (2) -- (1);
\draw (2) -- (8) -- (5) -- (2) -- (3);
\draw (8) -- (4);
\draw (7) -- (9) -- (6);
\end{tikzpicture}\\ 
$F(9,14,16)$ & $F(9,14,17)$ & $F(9,14,18)$ & $F(9,14,19)$ & $F(9,15,1)$\\
\\

\begin{tikzpicture}
\node[name=1] at (0,0) {};
\node[name=2] at (1/2,{sqrt(3)/2}) {};
\node[name=3] at (1,0) {};
\node[name=4] at (2,0) {};
\node[name=5] at (3/2,{sqrt(3)/2}) {};
\node[name=6] at (2,{sqrt(3)}) {};
\node[name=7] at (5/2,{sqrt(3)/2}) {};
\draw (1) -- (2) -- (3) -- (1);
\draw (2) -- (5) -- (3);
\draw (1) -- (4);
\draw (4) -- (5) -- (7) -- (4);
\node[name=8] at (3,{sqrt(3)}) {};
\node[name=9] at (3/2,1/4) {};
\draw (1) -- (9) -- (8);
\draw (6) -- (8) -- (7);
\draw (5) -- (6) -- (7);
\end{tikzpicture} 
&
\begin{tikzpicture}
\node[name=1] at (0,0) {};
\node[name=2] at (1,0) {};
\node[name=3] at (1/2,{sqrt(3)/2}) {};
\node[name=4] at (0,1) {};
\node[name=5] at (1,1) {};
\node[name=6] at (1/2,{1 + sqrt(3)/2}) {};
\draw (6) -- (4) -- (1) -- (2) -- (3) -- (6) -- (5) -- (2) (1) -- (3);
\node[name=7] at (3/2,{sqrt(3)/2}) {};
\draw (2) -- (7) -- (3);
\node[name=8] at (2,{3/2}) {};
\node[name=9] at (1,2) {};
\draw (7) -- (8) -- (4);
\draw (8) -- (9) -- (4);
\draw (9) -- (5);
\end{tikzpicture}
&
\begin{tikzpicture}
\node[name=1] at (0,0) {};
\node[name=2] at (1,0) {};
\node[name=3] at (1/2,{sqrt(3)/2}) {};
\node[name=4] at (0,1) {};
\node[name=5] at (1,1) {};
\node[name=6] at (1/2,{1 + sqrt(3)/2}) {};
\draw (6) -- (4) -- (1) -- (2) -- (3) -- (6) -- (5) -- (2) (1) -- (3);
\node[name=7] at ({1 + sqrt(3)/2},1/2) {};
\node[name=8] at ({-sqrt(3)/2},1/2) {};
\node[name=9] at ({1 + sqrt(3)/2},-1/2) {};
\draw (9) -- (8);
\draw (1) -- (8) -- (4);
\draw (2) -- (7) -- (5);
\draw (2) -- (9) -- (7);
\end{tikzpicture}
&
\begin{tikzpicture}
\node[name=1] at (0,0) {};
\node[name=2] at (1,0) {};
\node[name=3] at (1/2,{sqrt(3)/2}) {};
\node[name=4] at (0,1) {};
\node[name=5] at (1,1) {};
\node[name=6] at (1/2,{1 + sqrt(3)/2}) {};
\draw (6) -- (4) -- (1) -- (2) -- (3) -- (6) -- (5) -- (2) (1) -- (3);
\node[name=7] at ({1 + sqrt(3)/2},1/2) {};
\node[name=8] at (3/2,{1 + sqrt(3)/2}) {};
\node[name=9] at (9/4,5/4) {};
\draw (5) -- (7) -- (2);
\draw (6) -- (8) -- (5) -- (9) -- (7);
\draw (9) -- (8);
\end{tikzpicture}
&
\begin{tikzpicture}
\node[name=1] at (0,0) {};
\node[name=2] at (1,0) {};
\node[name=3] at (1/2,{sqrt(3)/2}) {};
\node[name=4] at (0,1) {};
\node[name=5] at (1,1) {};
\node[name=6] at (1/2,{1 + sqrt(3)/2}) {};
\draw (6) -- (4) -- (1) -- (2) -- (3) -- (6) -- (5) -- (2) (1) -- (3);
\node[name=7] at (3/2,{sqrt(3)/2}) {};
\node[name=8] at (1,{sqrt(3)}) {};
\draw (2) -- (7) -- (3) -- (8) -- (7);
\node[name=9] at (2,11/8) {};
\draw (5) -- (9) -- (8);
\draw (9) -- (4);
\end{tikzpicture}\\
$F(9,15,2)$ & $F(9,15,3)$ & $F(9,15,4)$ & $F(9,15,5)$ & $F(9,15,6)$\\
\\

\begin{tikzpicture}
\node[name=1] at (0,0) {};
\node[name=2] at (1,0) {};
\node[name=3] at (1/2,{sqrt(3)/2}) {};
\node[name=4] at (0,1) {};
\node[name=5] at (1,1) {};
\node[name=6] at (1/2,{1 + sqrt(3)/2}) {};
\draw (6) -- (4) -- (1) -- (2) -- (3) -- (6) -- (5) -- (2) (1) -- (3);
\node[name=7] at (1/2,1/4) {};
\node[name=8] at (3/2,1/4) {};
\node[name=9] at (1,{1/4 - sqrt(3)/2}) {};
\draw (5) -- (8) -- (7) -- (5);
\draw (4) -- (7) -- (9) -- (8);
\draw (9) -- (3);
\end{tikzpicture}
&
\begin{tikzpicture}
\node[name=1] at (0,0) {};
\node[name=2] at (1,0) {};
\node[name=3] at (1/2,{sqrt(3)/2}) {};
\node[name=4] at (0,1) {};
\node[name=5] at (1,1) {};
\node[name=6] at (1/2,{1 + sqrt(3)/2}) {};
\draw (6) -- (4) -- (1) -- (2) -- (3) -- (6) -- (5) -- (2) (1) -- (3);
\node[name=7] at ({1 + sqrt(3)/2},1/2) {};
\node[name=8] at ({-sqrt(3)/2},1/2) {};
\node[name=9] at ({1 + sqrt(3)/2},3/2) {};
\draw (9) -- (8);
\draw (1) -- (8) -- (4);
\draw (2) -- (7) -- (5);
\draw (5) -- (9) -- (7);
\end{tikzpicture}
&
\begin{tikzpicture}
\node[name=1] at (0,0) {};
\node[name=2] at (1,0) {};
\node[name=3] at (1/2,{sqrt(3)/2}) {};
\node[name=4] at (3/2,{sqrt(3)/2}) {};
\node[name=5] at (-1/2,{sqrt(3)/2}) {};
\node[name=6] at (0,{sqrt(3)}) {};
\node[name=7] at (1,{sqrt(3)}) {};
\draw (5) -- (1) (3) -- (1) -- (2) -- (3) -- (5) -- (6) -- (7) -- (3) -- (6) (2) -- (4) -- (3);
\node[name=8] at (2,{sqrt(3)}) {};
\node[name=9] at (1,9/8) {};
\draw (6) -- (8) -- (4);
\draw (8) -- (9) -- (1);
\end{tikzpicture}
&
\begin{tikzpicture}
\node[name=1] at (0,0) {};
\node[name=2] at (1,0) {};
\node[name=3] at (1/2,{sqrt(3)/2}) {};
\node[name=4] at (3/2,{sqrt(3)/2}) {};
\node[name=5] at (-1/2,{sqrt(3)/2}) {};
\node[name=6] at (0,{sqrt(3)}) {};
\node[name=7] at (1,{sqrt(3)}) {};
\draw (5) -- (1) (3) -- (1) -- (2) -- (3) -- (5) -- (6) -- (7) -- (3) -- (6) (2) -- (4) -- (3);
\node[name=8] at (2,0) {};
\node[name=9] at (2,1) {};
\draw (2) -- (8) -- (4);
\draw (8) -- (9) -- (7);
\end{tikzpicture}
&
\begin{tikzpicture}
\node[name=1] at (0,0) {};
\node[name=2] at (1,0) {};
\node[name=3] at (1/2,{sqrt(3)/2}) {};
\node[name=4] at (3/2,{sqrt(3)/2}) {};
\node[name=5] at (-1/2,{sqrt(3)/2}) {};
\node[name=6] at (0,{sqrt(3)}) {};
\node[name=7] at (1,{sqrt(3)}) {};
\draw (5) -- (1) (3) -- (1) -- (2) -- (3) -- (5) -- (6) -- (7) -- (3) -- (6) (7) -- (4) -- (2);
\node[name=8] at (2,0) {};
\node[name=9] at (1,-1/2) {};
\draw (1) -- (9) -- (8);
\draw (2) -- (8) -- (4);
\end{tikzpicture}\\
$F(9,15,7)$ & $F(9,15,8)$ & $F(9,15,9)$ & $F(9,15,10)$ & $F(9,15,11)$\\
\\

\begin{tikzpicture}
\node[name=1] at (0,0) {};
\node[name=2] at (1,0) {};
\node[name=3] at (1/2,{sqrt(3)/2}) {};
\node[name=4] at (3/2,{sqrt(3)/2}) {};
\node[name=5] at (-1/2,{sqrt(3)/2}) {};
\node[name=6] at (0,{sqrt(3)}) {};
\node[name=7] at (1,{sqrt(3)}) {};
\draw (5) -- (1) (3) -- (1) -- (2) -- (3) -- (5) -- (6) -- (7) -- (3) -- (6) (7) -- (4) -- (2);
\node[name=8] at (2,{sqrt(3)}) {};
\node[name=9] at (1,1/2) {};
\draw (7) -- (8) -- (4) -- (9) -- (7);
\end{tikzpicture}
&
\begin{tikzpicture}
\node[name=1] at (0,0) {};
\node[name=2] at (1,0) {};
\node[name=3] at (1/2,{sqrt(3)/2}) {};
\node[name=4] at (3/2,{sqrt(3)/2}) {};
\node[name=5] at (-1/2,{sqrt(3)/2}) {};
\node[name=6] at (0,{sqrt(3)}) {};
\node[name=7] at (1,{sqrt(3)}) {};
\draw (5) -- (1) (3) -- (1) -- (2) -- (3) -- (5) -- (6) -- (7) -- (3) -- (6) (7) -- (4) -- (2);
\node[name=8] at (2,{sqrt(3)}) {};
\node[name=9] at (1,1/2) {};
\draw (7) -- (8) -- (4);
\draw (5) -- (9) -- (8);
\end{tikzpicture}
&
\begin{tikzpicture}
\node[name=1] at (0,0) {};
\node[name=2] at (1,0) {};
\node[name=3] at (1/2,{sqrt(3)/2}) {};
\node[name=4] at (3/2,{sqrt(3)/2}) {};
\node[name=5] at (-1/2,{sqrt(3)/2}) {};
\node[name=6] at (0,{sqrt(3)}) {};
\node[name=7] at (1,{sqrt(3)}) {};
\draw (5) -- (1) (3) -- (1) -- (2) -- (3) -- (5) -- (6) -- (7) -- (3) -- (6) (7) -- (4) -- (2);
\node[name=8] at (2,{sqrt(3)}) {};
\node[name=9] at (1,1/2) {};
\draw (7) -- (8) -- (4);
\draw (1) -- (9) -- (8);
\end{tikzpicture}
&
\begin{tikzpicture}
\node[name=1] at (0,0) {};
\node[name=2] at (1,0) {};
\node[name=3] at (1/2,{sqrt(3)/2}) {};
\node[name=4] at (3/2,{sqrt(3)/2}) {};
\node[name=5] at (-1/2,{sqrt(3)/2}) {};
\node[name=6] at (0,{sqrt(3)}) {};
\node[name=7] at (1,{sqrt(3)}) {};
\draw (5) -- (1) (3) -- (1) -- (2) -- (3) -- (5) -- (6) -- (7) -- (3) -- (6) (7) -- (4) -- (2);
\node[name=8] at (-1,0) {};
\draw (1) -- (8) -- (5);
\node[name=9] at (1/2,1/4) {};
\draw (8) -- (9) -- (4);
\end{tikzpicture}
&
\begin{tikzpicture}
\node[name=1] at (0,0) {};
\node[name=2] at (1,0) {};
\node[name=3] at (1/2,{sqrt(3)/2}) {};
\node[name=4] at (3/2,{sqrt(3)/2}) {};
\node[name=5] at (-1/2,{sqrt(3)/2}) {};
\node[name=6] at (0,{sqrt(3)}) {};
\node[name=7] at (1,{sqrt(3)}) {};
\draw (5) -- (1) (3) -- (1) -- (2) -- (3) -- (5) -- (6) -- (7) -- (3) -- (6) (7) -- (4) -- (2);
\node[name=8] at (2,{sqrt(3)}) {};
\node[name=9] at (2,1/2) {};
\draw (7) -- (8) -- (4);
\draw (3) -- (9) -- (8);
\end{tikzpicture}\\
$F(9,15,12)$ & $F(9,15,13)$ & $F(9,15,14)$ & $F(9,15,15)$ & $F(9,15,16)$
\end{tabular}
\end{center}

\newpage

\begin{center}
\begin{tabular}{ccccc}
\begin{tikzpicture}
\node[name=1] at (0,0) {};
\node[name=2] at (1,0) {};
\node[name=3] at (1/2,{sqrt(3)/2}) {};
\node[name=4] at (0,1) {};
\node[name=5] at (-1/2,{1 + sqrt(3)/2}) {};
\node[name=6] at (1/2,{1 + sqrt(3)/2}) {};
\node[name=7] at (2,0) {};
\node[name=8] at (3/2,{sqrt(3)/2}) {};
\node[name=9] at (5/2,{sqrt(3)/2}) {};
\draw (6) -- (4) -- (1) -- (2) -- (3) -- (6) -- (5) (1) -- (3);
\draw (4) -- (5);
\draw (8) -- (2) -- (7) -- (8) -- (3);
\draw (7) -- (9) -- (8);
\draw (5) -- (9);
\end{tikzpicture}
&
\begin{tikzpicture}
\node[name=1] at (0,0) {};
\node[name=2] at (1,0) {};
\node[name=3] at (1/2,{sqrt(3)/2}) {};
\node[name=4] at (2,1/2) {};
\node[name=5] at (3/2,{sqrt(3)}) {};
\node[name=6] at (5/2,{sqrt(3)}) {};
\node[name=7] at (2,0) {};
\node[name=8] at (3/2,{sqrt(3)/2}) {};
\node[name=9] at (5/2,{sqrt(3)/2}) {};
\draw (6) -- (4) (1) -- (2) -- (3)  (6) -- (5) (1) -- (3);
\draw (4) -- (5);
\draw (8) -- (2) -- (7) -- (8) -- (3);
\draw (7) -- (9) -- (8);
\draw (5) -- (8);
\draw (6) -- (9);
\draw (1) -- (4);
\end{tikzpicture}
&
\begin{tikzpicture}
\node[name=1] at (0,0) {};
\node[name=2] at (1/2,{sqrt(3)/2}) {};
\node[name=3] at (1,0) {};
\node[name=4] at (2,{sqrt(3)}) {};
\node[name=5] at (3/2,{sqrt(3)/2}) {};
\node[name=6] at (-1/2,{sqrt(3)/2}) {};
\node[name=7] at (5/2,{sqrt(3)/2}) {};
\node[name=8] at (1,{sqrt(3)}) {};
\node[name=9] at (-1,0) {};
\draw (3) -- (1);
\draw (5) -- (3);
\draw (4) -- (5) -- (7) -- (4);
\draw (1) -- (6) -- (2);
\draw (2) -- (8) -- (5) -- (2) -- (3);
\draw (8) -- (4);
\draw (1) -- (9) -- (6);
\draw (9) -- (7);
\end{tikzpicture}
&
\begin{tikzpicture}
\node[name=1] at (0,0) {};
\node[name=2] at (1,0) {};
\node[name=3] at (1/2,{sqrt(3)/2}) {};
\node[name=5] at (-1/2,{sqrt(3)/2}) {};
\node[name=6] at ({sqrt(3)/2 - 1/2}, {1/2 + sqrt(3)/2}) {};
\node[name=7] at (-1/2,{1 + sqrt(3)/2}) {};
\node[name=8] at ({-1/2 + sqrt(3)/2},{3/2 + sqrt(3)/2}) {};
\node[name=4] at (1,1) {};
\node[name=9] at (3/2,3/2) {};
\draw (1) -- (2) -- (3) -- (5) -- (1) -- (3);
\draw (5) -- (6) -- (8) -- (7) -- (5);
\draw (6) -- (7);
\draw (2) -- (4) -- (9) -- (8) -- (4) (2) -- (9);
\end{tikzpicture}
&
\begin{tikzpicture}
\node[name=1] at (0,0) {};
\node[name=2] at (1,0) {};
\node[name=3] at (1/2,{sqrt(3)/2}) {};
\node[name=5] at (-1/2,{sqrt(3)/2}) {};
\node[name=6] at ({sqrt(3)/2 - 1/2}, {1/2 + sqrt(3)/2}) {};
\node[name=7] at (-1/2,{1 + sqrt(3)/2}) {};
\node[name=8] at ({sqrt(3)/2 - 1/2},{3/2 + sqrt(3)/2}) {};
\node[name=4] at (1,2) {};
\node[name=9] at (3/2,1) {};
\draw (1) -- (2) -- (3) -- (5) -- (1) -- (3);
\draw (5) -- (6) -- (8) -- (7) -- (5);
\draw (6) -- (7);
\draw (3) -- (4) (9) -- (6);
\draw (4) -- (9);
\draw (8) -- (4) (9) -- (2);
\end{tikzpicture}\\
$F(9,15,17)$ & $F(9,15,18)$ & $F(9,15,19)$ & $F(9,15,20)$ & $F(9,15,21)$\\
\\
\begin{tikzpicture}
\node[name=1] at (0,0) {};
\node[name=2] at (1,0) {};
\node[name=3] at (1/2,{-sqrt(3)/2}) {};
\node[name=4] at (0,1) {};
\node[name=5] at (1,1) {};
\node[name=6] at (1/2,1/2) {};
\node[name=7] at (2,0) {};
\node[name=8] at (3/2,{-sqrt(3)/2}) {};
\node[name=9] at (-1/2, 1/3) {};
\draw (6) -- (4) -- (1) -- (2) -- (3) (6) -- (5) (1) -- (3);
\draw (4) -- (5) -- (2);
\draw (8) -- (2) -- (7) -- (8) -- (3);
\draw (4) -- (9) -- (6);
\draw (9) -- (7);
\end{tikzpicture}
&
\begin{tikzpicture}
\node[name=1] at (0,0) {};
\node[name=2] at (1,0) {};
\node[name=3] at (0,1) {};
\node[name=4] at (1,1) {};
\node[name=5] at (1/2,{1 + sqrt(3)/2}) {};
\node[name=6] at (1/2,{-sqrt(3)/2}) {};
\node[name=7] at (2,1/2) {};
\node[name=8] at (5/2,1) {};
\node[name=9] at (5/2,0) {};
\draw (1) -- (2) -- (4) -- (3) -- (1) (8) -- (9);
\draw (1) -- (6) -- (2) (7) -- (6) -- (9) -- (7);
\draw (3) -- (5) -- (4) (7) -- (5) -- (8) -- (7);
\end{tikzpicture}
&
\begin{tikzpicture}
\node[name=1] at (0,0) {};
\node[name=2] at (1,0) {};
\node[name=3] at (1/2,{sqrt(3)/2}) {};
\node[name=4] at (0,-3/2) {};
\node[name=6] at (1/2,-1/2) {};
\node[name=7] at (2,0) {};
\node[name=8] at (3/2,{sqrt(3)/2}) {};
\node[name=5] at (5/4,{-1/2}) {};
\node[name=9] at (1.75,-3/2) {};
\draw (6) -- (4) -- (1) -- (2) -- (3) -- (6) (1) -- (3);
\draw (8) -- (2) -- (7) -- (8) -- (3);
\draw (6) -- (5) -- (9) -- (4);
\draw (5) -- (7) -- (9);
\end{tikzpicture}
&
\begin{tikzpicture}
\node[name=1] at (0,0) {};
\node[name=2] at (1,0) {};
\node[name=3] at (1/2,{sqrt(3)/2}) {};
\node[name=4] at (0,9/8) {};
\node[name=5] at (1,9/8) {};
\node[name=6] at (1/2,{1 + sqrt(3)/2}) {};
\node[name=7] at (2,0) {};
\node[name=8] at (3/2,{sqrt(3)/2}) {};
\node[name=9] at (3/2,{1 + sqrt(3)/2}) {};
\draw (6) -- (4) -- (1) -- (2) -- (3) -- (6) -- (5) (1) -- (3);
\draw (4) -- (5);
\draw (8) -- (2) -- (7) -- (8) -- (3);
\draw (6) -- (9) -- (5);
\draw (9) -- (7);
\end{tikzpicture}
&
\begin{tikzpicture}
\node[name=1] at (0,0) {};
\node[name=2] at (1,0) {};
\node[name=3] at (1/2,{sqrt(3)/2}) {};
\node[name=4] at (0,1) {};
\node[name=5] at (-1/2,{1 + sqrt(3)/2}) {};
\node[name=6] at (1/2,{1 + sqrt(3)/2}) {};
\node[name=7] at (2,0) {};
\node[name=8] at (3/2,{sqrt(3)/2}) {};
\node[name=9] at (-1,1) {};
\draw (6) -- (4) -- (1) -- (2) -- (3) -- (6) -- (5) (1) -- (3);
\draw (4) -- (5);
\draw (8) -- (2) -- (7) -- (8) -- (3);
\draw (4) -- (9) -- (5);
\draw (9) -- (7);
\end{tikzpicture}\\
$F(9,15,22)$ & $F(9,15,23)$ & $F(9,15,24)$ & $F(9,15,25)$ & $F(9,15,26)$
\end{tabular}
\begin{tabular}{cccc}
\begin{tikzpicture}
\node[name=1] at (0,0) {};
\node[name=2] at (1,0) {};
\node[name=3] at (2,0) {};
\node[name=4] at (1/2,{sqrt(3)/2}) {};
\node[name=5] at (3/2,{sqrt(3)/2}) {};
\node[name=6] at (1,{sqrt(3)}) {};
\draw (1) -- (2) -- (3) -- (5) -- (2) -- (4) -- (1);
\draw (4) -- (5) -- (6) -- (4);
\node[name=7] at (1/2,1/3) {};
\node[name=8] at (3/2,1/3) {};
\node[name=9] at (1,9/8) {};
\draw (1) -- (7) -- (8) -- (9) -- (7);
\draw (8) -- (3);
\draw (9) -- (6);
\end{tikzpicture}
&
\begin{tikzpicture}
\node[name=1] at (0,0) {};
\node[name=2] at (1,0) {};
\node[name=3] at (1/2,{sqrt(3)/2}) {};
\node[name=4] at (3/2,{sqrt(3)/2}) {};
\node[name=5] at (1/2,{1 + sqrt(3)/2}) {};
\node[name=6] at (3/2,{1 + sqrt(3)/2}) {};
\draw (2) -- (3) -- (1) -- (2) -- (4) -- (3);
\draw (3) -- (5) -- (6) -- (4);
\node[name=7] at (1, 5/4) {};
\node[name=8] at (5/2,{sqrt(3)/2}) {};
\draw (5) -- (7) -- (6) (8) -- (7);
\node[name=9] at (5/2, {1 + sqrt(3)/2}) {};
\draw (6) -- (9) -- (8);
\draw (9) -- (4);
\draw (8) -- (1);
\end{tikzpicture}
&
\begin{tikzpicture}
\node[name=2] at (1,0) {};
\node[name=4] at (2,{1 + sqrt(3)}) {};
\node[name=5] at (3/2,{1 + sqrt(3)/2}) {};
\node[name=6] at (5/2,{1 + sqrt(3)/2}) {};
\node[name=7] at (2,0) {};
\node[name=8] at (3/2,{sqrt(3)/2}) {};
\node[name=9] at (5/2,{sqrt(3)/2}) {};
\draw (6) -- (4) (6) -- (5);
\draw (4) -- (5);
\draw (8) -- (2) -- (7) -- (8);
\draw (7) -- (9) -- (8);
\draw (5) -- (8);
\draw (6) -- (9);
\node[name=1] at (3/4,7/4) {};
\node[name=3] at (9/4,5/4) {};
\draw (4) -- (1) -- (2) -- (3) -- (4) (1) -- (3);
\end{tikzpicture}
&
\begin{tikzpicture}
\node[name=1] at (-1/2,0) {};
\node[name=2] at (1/2,0) {};
\node[name=3] at (5/2,0) {};
\node[name=4] at (1/2, {sqrt(3)/2}) {};
\node[name=5] at (1/2, {-sqrt(3)/2}) {};
\node[name=6] at (3/2, {sqrt(3)/2}) {};
\node[name=7] at (3/2, {-sqrt(3)/2}) {};
\draw (1) -- (4) -- (2) -- (5) -- (1);
\draw (3) -- (6) -- (2) -- (7) -- (3);
\draw (4) -- (6) (5) -- (7);
\node[name=8] at (7/8,5/8) {};
\node[name=9] at (7/8,-5/8) {};
\draw (8) -- (9) -- (3) -- (8) -- (1) -- (9);
\end{tikzpicture}\\
$F(9,15,27)$ & $F(9,15,28)$ & $F(9,15,29)$ & $F(9,15,30)$\\
\\

\begin{tikzpicture}
\node[name=1] at (0,0) {};
\node[name=2] at (1,0) {};
\node[name=3] at (2,0) {};
\node[name=4] at (3,0) {};
\node[name=5] at (1/2,-{sqrt(3)/2}) {};
\node[name=6] at (5/2,-{sqrt(3)/2}) {};
\node[name=7] at (3/2,{sqrt(3)/2}) {};
\node[name=8] at (1,{sqrt(3)}) {};
\node[name=9] at (2,{sqrt(3)}) {};
\draw (1) -- (2) -- (3) -- (4) -- (6) -- (3) -- (7) -- (2) -- (5) -- (1);
\draw (7) -- (8) -- (9) -- (7);
\draw (1) -- (8) (9) -- (4) (5) -- (6);
\end{tikzpicture}
&
\begin{tikzpicture}
\node[name=1] at (0,0) {};
\node[name=2] at (1,0) {};
\node[name=3] at (1/2,{sqrt(3)/2}) {};
\node[name=4] at (3/2,{sqrt(3)/2}) {};
\node[name=5] at (0,-1) {};
\node[name=6] at (1,-1) {};
\node[name=7] at (1/2,{1 + sqrt(3)/2}) {};
\node[name=8] at (3/2,{1 + sqrt(3)/2}) {};
\draw (1) -- (2) -- (3) -- (4) -- (2) (3) -- (1);
\draw (1) -- (5) -- (6) -- (2);
\draw (3) -- (7) -- (8) -- (4);
\node[name=9] at (-3/8,{sqrt(3)/2}) {};
\draw (7) -- (9) -- (8);
\draw (5) -- (9) --  (6);
\end{tikzpicture}
&
\begin{tikzpicture}
\node[name=1] at (0,0) {};
\node[name=2] at (1,0) {};
\node[name=3] at (1/2,{sqrt(3)/2}) {};
\node[name=4] at (3/2,{sqrt(3)/2}) {};
\node[name=5] at (2,0) {};
\node[name=6] at (1/2,{1 + sqrt(3)/2}) {};
\node[name=7] at (3/2,{1 + sqrt(3)/2}) {};
\node[name=8] at (2,5/4) {};
\node[name=9] at (0,5/4) {};
\draw (1) -- (2) -- (3) -- (4) -- (5) -- (2) -- (4);
\draw (1) -- (3) -- (6) -- (7) -- (4);
\draw (6) -- (8) -- (9) -- (7);
\draw (8) -- (5);
\draw (9) -- (1);
\end{tikzpicture}
&
\begin{tikzpicture}
\node[name=1] at (2,0) {};
\node[name=2] at (1,0) {};
\node[name=3] at (3/2,{sqrt(3)/2}) {};
\node[name=4] at (1/2,{sqrt(3)/2}) {};
\node[name=5] at (1/2,{1 + sqrt(3)/2}) {};
\node[name=6] at (3/2,{1 + sqrt(3)/2}) {};
\draw (2) -- (3) -- (1) -- (2) -- (4) -- (3);
\draw (3) -- (5) -- (6) -- (4);
\node[name=7] at (1, {1 + sqrt(3)}) {};
\node[name=8] at (2, {1 + sqrt(3)}) {};
\draw (5) -- (7) -- (6) (8) -- (7);
\node[name=9] at (5/2, {1 + sqrt(3)/2}) {};
\draw (6) -- (9) -- (8);
\draw (8) -- (1) -- (9);
\end{tikzpicture}\\
$F(9,15,31)$ & $F(9,15,32)$ & $F(9,15,33)$ & $F(9,15,34)$
\end{tabular}
\end{center}

\newpage

\section*{Appendix B:  Exact coordinates for unit-distance graphs}

We collect here coordinates for vertices $(x,y) \in \mathbf{R}^2$ of the unit-distance graphs used for \cref{n8theorem,n9theorem}.  In Table 1, we report exact coordinates for the vertices of $G_{27}$.  In \cref{g118table,h1table,h2table}, we report rounded coordinates for $(x,y) \in \mathbf{R}^2$ along with the minimal polynomial for $z = x + iy$.  We have scaled these minimal polynomials to ensure integer coefficients, and enough precision has been reported in each rounded coordinate to distinguish distinct roots of each polynomial.  The exact coordinates can also be found in the \texttt{SageMath} worksheet (for $G_{27}$ and $G_{118}$) and the \texttt{Mathematica} notebook (for $H_1$ and $H_2$) that we provide~\cite{smallUDcode}.

\begin{table}[b!]
\caption{Coordinates for the vertices of the embedded unit-distance graph $G_{27}$.}\label{g27table}
\begin{center}
{\def\arraystretch{1.10}
\begin{tabular}{rr}
$x$ & $y$\\
\hline
$ 0 $ & $ 0 $\\
$ 1 $ & $ 0 $\\
$ -1 $ & $ 0 $\\
$ 2 $ & $ 0 $\\
$ -1/2 $ & $ \sqrt{3}/2 $\\
$ 1/2 $ & $ \sqrt{3}/2 $\\
$ 1/2 $ & $ -\sqrt{3}/2 $\\
$ -1/2 $ & $ -\sqrt{3}/2 $\\
$ 3/2 $ & $ \sqrt{3}/2 $\\
$ 5/6 $ & $ \sqrt{11}/6 $\\
$ 11/6 $ & $ \sqrt{11}/6 $\\
$ 5/2 $ & $ \sqrt{3}/2 $\\
$ 3/2 $ & $ -\sqrt{3}/2 $\\
$ 12/7 $ & $ \sqrt{3}/7 $\\
$ 11/14 $ & $ 5\sqrt{3}/14 $\\
$ 4/3 $ & $ \sqrt{11}/6 + \sqrt{3}/2 $\\
$ \sqrt{33}/6 $ & $ \sqrt{3}/6 $\\
$ \sqrt{33}/6 + 1 $ & $ \sqrt{3}/6 $\\
$ \sqrt{33}/6 + 1/2 $ & $ -\sqrt{3}/3 $\\
$ \sqrt{33}/12 + 1/12 $ & $ -\sqrt{11}/12 + \sqrt{3}/12$\\
$ \sqrt{33}/12 + 11/12 $ & $ \sqrt{11}/12 + \sqrt{3}/12 $\\
$ \sqrt{33}/12 + 1/4 $ & $ \sqrt{11}/4 - 5\sqrt{3}/12 $\\
$ \sqrt{33}/12 - 5/12 $ & $ -\sqrt{11}/12 - 5\sqrt{3}/12 $\\
$ \sqrt{33}/12 + 13/12 $ & $ -\sqrt{11}/12 + \sqrt{3}/12 $\\
$ \sqrt{33}/12 + 5/12 $ & $ \sqrt{11}/12 - 5\sqrt{3}/12 $\\
$ -\sqrt{385}/60 + 17/12 $ & $ -7\sqrt{35}/60 + \sqrt{11}/12 $\\
$ \sqrt{385}/60 + 17/12 $ & $ 7\sqrt{35}/60 + \sqrt{11}/12 $\\
\end{tabular}}
\end{center}
\end{table}

\clearpage


{\def\arraystretch{1.2}
\captionsetup{width=\textwidth}
\begin{longtable}[c]{rrl}
\caption{Coordinates for the vertices of $G_{118}$ that do not already appear in  $G_{27}$.}\label{g118table}
\\
$x$ & $y$ & Minimal Polynomial of $z = x + iy$\\
\hline
\endhead
$ 0.00000 $ & $ 1.00000 $ & $ z^{2} + 1 $\\
$ 0.00000 $ & $ 1.73205 $ & $ z^2 + 3 $\\
$ 0.00000 $ & $ -1.73205 $ & $ z^2 + 3 $\\
$ 0.70711 $ & $ 0.00000 $ & $ 2z^2 - 1 $\\
$ 0.00000 $ & $ 0.70711 $ & $ 2z^2 + 1 $\\
$ 1.00000 $ & $ 1.00000 $ & $ z^{2} - 2z + 2 $\\
$ 1.00000 $ & $ 1.73205 $ & $ z^2 - 2z + 4 $\\
$ 1.00000 $ & $ -1.73205 $ & $ z^2 - 2z + 4 $\\
$ -1.50000 $ & $ 0.86603 $ & $ z^2 + 3z + 3 $\\
$ 0.16667 $ & $ -0.55277 $ & $ 3z^2 - z + 1 $\\
$ -1.00000 $ & $ -1.73205 $ & $ z^2 + 2z + 4 $\\
$ 2.00000 $ & $ -1.73205 $ & $ z^2 - 4z + 7 $\\
$ 0.87500 $ & $ -0.48412 $ & $ 4z^2 - 7z + 4 $\\
$ 1.57143 $ & $ 1.23718 $ & $ 7z^2 - 22z + 28 $\\
$ 0.86603 $ & $ 0.50000 $ & $ z^4 - z^2 + 1 $\\
$ 1.13397 $ & $ 0.50000 $ & $ z^4 - 8z^3 + 23z^2 - 28z + 13 $\\
$ 0.70711 $ & $ 0.70711 $ & $ z^4 + 1 $\\
$ 0.95743 $ & $ -0.28868 $ & $ 3z^4 - 5z^2 + 3 $\\
$ 0.35355 $ & $ 0.93541 $ & $ 2z^4 + 3z^2 + 2 $\\
$ 0.93541 $ & $ 0.35355 $ & $ 2z^4 - 3z^2 + 2 $\\
 $ 0.45743 $ & $ -0.57735 $ & $ 3z^4 + 6z^3 + z^2 - 2z + 4 $\\
$ 0.95743 $ & $ -1.44338 $ & $ 3z^4 + 7z^2 + 27 $\\
$ 0.51824 $ & $ -0.13381 $ & $ 16z^4 - 60z^3 + 81z^2 - 45z + 9 $\\
$ -1.21353 $ & $ -0.16540 $ & $ 4z^4 + 6z^3 + 3z^2 + 9z + 9 $\\
$ -0.35676 $ & $ 0.35031 $ & $ 16z^4 - 4z^3 - 3z^2 - z + 1 $\\
$ -0.33853 $ & $ -0.64952 $ & $ 16z^4 - 32z^3 + 15z^2 + z + 19 $\\
$ 0.22871 $ & $ -1.55084 $ & $ 3z^4 + 3z^3 + 7z^2 + 10z + 4 $\\
$ -0.27129 $ & $ 0.10747 $ & $ 3z^4 + 9z^3 + 16z^2 + 7z + 1 $\\
$ 0.72871 $ & $ -0.68482 $ & $ 3z^4 - 3z^3 + 4z^2 - 3z + 3 $\\
$ -0.39538 $ & $ -0.42072 $ & $ 9z^4 - 3z^3 - 2z^2 - z + 1 $\\
$ 1.22871 $ & $ -1.55084 $ & $ 3z^4 - 9z^3 + 16z^2 - 7z + 1 $\\
$ 0.22871 $ & $ 0.97349 $ & $ 3z^4 + 3z^3 + 4z^2 + 3z + 3 $\\
$ 1.06205 $ & $ 0.73398 $ & $ 9z^4 - 21z^3 + 34z^2 - 35z + 25 $\\
$ 1.18614 $ & $ -1.26217 $ & $ z^4 + z^3 - 2z^2 + 3z + 9 $\\
$ -0.33333 $ & $ 0.31325 $ & $ 9z^4 + 12z^3 + 25z^2 + 14z + 4 $\\
$ -0.43795 $ & $ -0.13205 $ & $ 9z^4 + 33z^3 + 43z^2 + 22z + 4 $\\
$ -0.72871 $ & $ -1.83952 $ & $ 3z^4 + 3z^3 + 10z^2 - 5z + 1 $\\
$ 0.39538 $ & $ 0.42072 $ & $ 9z^4 + 3z^3 - 2z^2 + z + 1 $\\
$ 1.39538 $ & $ -1.31133 $ & $ 9z^4 - 33z^3 + 88z^2 - 121z + 121 $\\
$ 0.35676 $ & $ -0.35031 $ & $ 16z^4 + 4z^3 - 3z^2 + z + 1 $\\
$ 1.06205 $ & $ -0.99807 $ & $ 9z^4 - 21z^3 + 25z^2 - 8z + 4 $\\
$ 0.22871 $ & $ -0.75856 $ & $ 3z^4 + 3z^3 + 19z^2 - 6z + 12 $\\
$ -1.45743 $ & $ -1.15470 $ & $ 3z^4 + 6z^3 + 7z^2 + 4z + 16 $\\
$ 0.37500 $ & $ -1.35015 $ & $ 16z^4 - 24z^3 + 45z^2 - 27z + 9 $\\
$ -0.04257 $ & $ 0.28868 $ & $ 3z^4 + 12z^3 + 13z^2 + 2z + 1 $\\
$ -0.71353 $ & $ 0.70063 $ & $ 4z^4 - 2z^3 - 3z^2 - 2z + 4 $\\
$ -0.21353 $ & $ 1.56665 $ & $ 4z^4 - 10z^3 + 15z^2 - 25z + 25 $\\
$ -1.71353 $ & $ 0.70063 $ & $ 4z^4 + 14z^3 + 15z^2 + 2z + 1 $\\
$ -0.60462 $ & $ 0.42072 $ & $ 9z^4 + 39z^3 + 61z^2 + 42z + 12 $\\
$ 0.53647 $ & $ -1.13364 $ & $ 4z^4 - 22z^3 + 45z^2 - 49z + 31 $\\
$ 1.89538 $ & $ -0.44530 $ & $ 9z^4 - 51z^3 + 115z^2 - 128z + 64 $\\
$ 0.56205 $ & $ -1.86410 $ & $ 9z^4 - 3z^3 + 43z^2 + 8z + 64 $\\
$ -0.22871 $ & $ -0.97349 $ & $ 3z^4 - 3z^3 + 4z^2 - 3z + 3 $\\
$ 0.39538 $ & $ -1.31133 $ & $ 9z^4 + 3z^3 + 43z^2 - 8z + 64 $\\
$ -0.85676 $ & $ -0.51571 $ & $ 16z^4 + 28z^3 + 33z^2 + 28z + 16 $\\
$ 0.06205 $ & $ 0.73398 $ & $ 9z^4 + 15z^3 + 25z^2 + 6z + 12 $\\
$ -1.39538 $ & $ -0.42072 $ & $ 9z^4 + 33z^3 + 43z^2 + 22z + 4 $\\
$ 1.79076 $ & $ -0.89060 $ & $ 9z^4 - 30z^3 + 64z^2 - 120z + 144 $\\
$ 0.85676 $ & $ 0.51571 $ & $ 16z^4 - 28z^3 + 33z^2 - 28z + 16 $\\
$ -0.27129 $ & $ -0.68482 $ & $ 3z^4 + 9z^3 + 13z^2 + 8z + 4 $\\
$ -0.56205 $ & $ 0.13205 $ & $ 9z^4 + 3z^3 - 2z^2 + z + 1 $\\
$ 0.01824 $ & $ -0.99983 $ & $ 16z^4 - 28z^3 + 33z^2 - 28z + 16 $\\
$ -0.93795 $ & $ -0.99807 $ & $ 9z^4 + 51z^3 + 115z^2 + 128z + 64 $\\
$ 2.18614 $ & $ -1.26217 $ & $ z^4 - 3z^3 + z^2 + 6z + 4 $\\
$ -0.48176 $ & $ -0.13381 $ & $ 16z^4 + 4z^3 - 3z^2 + z + 1 $\\
$ -0.51824 $ & $ 0.13381 $ & $ 16z^4 + 60z^3 + 81z^2 + 45z + 9 $\\
$ 0.48176 $ & $ 0.13381 $ &  $16z^{4} - 4z^{3} - 3z^{2} - z + 1$\\
$ 1.13397 $ & $ -0.50000 $ & $ z^4 - 8z^3 + 23z^2 - 28z + 13 $\\
$ 1.28897 $ & $ 1.28897 $ & $ 16z^8 + 184z^4 + 81 $\\
$ 1.64252 $ & $ 0.35355 $ & $ 4z^8 - 20z^6 + 29z^4 + 4z^2 + 1 $\\
$ 0.35355 $ & $ 1.64252 $ & $ 4z^8 + 20z^6 + 29z^4 - 4z^2 + 1 $\\
$ 1.39781 $ & $ -0.91747 $ & $ 9z^{8} - 18z^{7} + 30z^{5} + z^{4} - 22z^{3} + 12z^{2} - 4z + 1 $\\
$ 0.74967 $ & $ 0.96816 $ & $ 13z^{8} - 23z^{6} + 34z^{5} + 17z^{4} - 52z^{3} + 36z^{2} - 16z + 4 $\\
$ 1.14690 $ & $ -1.15470 $ & $ 9z^8 + 18z^7 + 51z^6 + 36z^5 + 79z^4$\\
& & $+\;154z^3 + 305z^2 + 296z + 124 $\\
$ 0.30574 $ & $ 0.83455 $ & $ 2025z^8 + 8100z^7 + 15480z^6 + 18090z^5 + 17734z^4$\\
& & $+\;14768z^3 + 10117z^2 + 4398z + 1164 $\\
$ -0.40320 $ & $ 0.12928 $ & $ 225z^8 + 375z^7 + 190z^6 + 485z^5 + 581z^4$\\
& & $+\;32z^3 - 52z^2 + 32z + 16 $\\
$ 0.78118 $ & $ -1.10774 $ & $ 279z^8 - 93z^7 + 655z^6 + 82z^5 + 828z^4$\\
& & $+\;484z^3 + 508z^2 + 216z + 48 $\\
$ 1.71329 $ & $ 1.03118 $ & $ 13z^8 - 52z^7 + 113z^6 - 128z^5 + 88z^4$\\
& & $+\;48z^3 - 27z^2 + 108z + 81 $\\
$ 1.74967 $ & $ 0.96816 $ & $ 13z^8 - 104z^7 + 341z^6 - 556z^5 + 412z^4$\\
& & $-\;48z^3 - 27z^2 - 108z + 81 $\\
$ 0.25033 $ & $ -0.96816 $ & $ 13z^8 - 104z^7 + 341z^6 - 624z^5 + 752z^4$\\
& & $-\;624z^3 + 341z^2 - 104z + 13 $\\
$ 0.25033 $ & $ 0.96816 $ & $ 13z^8 - 104z^7 + 341z^6 - 624z^5 + 752z^4$\\
& & $-\;624z^3 + 341z^2 - 104z + 13 $\\
$ -0.48172 $ & $ -0.67568 $ & $ 999z^8 - 666z^7 + 912z^6 - 1230z^5 + 718z^4$\\
& & $-\;548z^3 + 483z^2 - 248z + 64 $\\
$ 0.30608 $ & $ 0.11499 $ & $ 16z^8 - 64z^7 + 183z^6 - 325z^5 + 484z^4$\\
& & $-\;501z^3 + 288z^2 - 81z + 9 $\\
$ 0.59680 $ & $ 0.12928 $ & $225z^{8} - 1425z^{7} + 3865z^{6} - 5380z^{5} + 3631z^{4}$\\
& & $-\;717z^{3} - 237z^{2} + 18z + 36$\\
$ 0.09680 $ & $ 0.99530 $ & $225z^{8} - 525z^{7} + 565z^{6} - 490z^{5} + 486z^{4}$\\
& & $-\;490z^{3} + 565z^{2} - 525z + 225$\\
$ 0.47571 $ & $ -0.38700 $ & $999z^{8} - 666z^{7} - 2418z^{6} - 2634z^{5} + 7384z^{4}$\\
& & $+\;3202z^{3} - 3633z^{2} + 238z + 1156$\\
$ -1.23654 $ & $ -0.42349 $ & $675z^{8} + 5625z^{7} + 21495z^{6} + 51510z^{5} + 85858z^{4}$\\
& & $+\;98070z^{3} + 70987z^{2} + 28665z + 5047$\\
$ 0.61798 $ & $ 0.12699 $ & $ 304z^8 - 684z^7 + 195z^6 + 135z^5 + 576z^4$\\
& & $-\;297z^3 - 108z^2 - 81z + 81 $\\
\\
$ 0.36350 $ & $ 0.77128 $ & $ 105975z^{16} - 671175z^{15} + 2077590z^{14} - 3905505z^{13} + 4876441z^{12}$\\
& &  $-\;3943474z^{11} + 1710135z^{10} + 382907z^9- 1122935z^8$\\
& &  $+\;835038z^7-235773z^6 - 71793z^5 +111474z^4$\\
& & $-\;54999z^3 + 14337z^2 - 2592z + 324$\\
$ 1.41417 $ & $ 1.91020 $ & $637z^{16} - 10192z^{15} + 75677z^{14} - 343602z^{13} + 1055953z^{12} $\\
& & $-\;2302592z^{11} + 3687774z^{10} - 4669072z^{9} + 5664150z^{8}$\\
& & $-\;8051452z^{7} + 11919114z^{6} - 14613824z^{5} + 13282489z^{4}$ \\
& & $-\;8473080z^{3} + 3559697z^{2} - 880066z + 98713$\\
$ 0.41417 $ & $ 1.91020 $ & $637z^{16} - 763z^{14} + 2436z^{13} - 2286z^{12} - 15348z^{11}$\\
& & $+\;29185z^{10} - 17988z^{9} + 16314z^{8} - 28356z^{7}$\\
& & $+\;15845z^{6} + 11436z^{5} - 14891z^{4} + 924z^{3}$\\
& & $+\;5076z^{2} - 2376z + 324$\\

\end{longtable}
}

\begin{table}[b!]
\caption{Coordinates for the unit-distance embedding of $H_1$ depicted in \cref{uglyFigure}.}\label{h1table}
\begin{center}
{\def\arraystretch{1.15}
\begin{tabular}{rrl}
$x$ & $y$ & Minimal Polynomial for $z = x + iy$\\
\hline
$0.00000$ & $0.00000$ & $z$\\
$1.00000$ & $0.00000$ & $z - 1$\\
$2.00000$ & $0.00000$ & $z - 2$\\
$0.50000$ & $0.86603$ & $z^2 - z + 1$\\
$1.50000$ & $0.86603$ & $z^2 - 3z + 3$\\
$0.62836$ & $0.37588$ & $21z^{10} - 231z^9 + 1030z^8 - 2379z^7 + 3189z^6 - 3638z^5$\\
& & $+\;6660z^4 - 11904z^3 + 12832z^2 - 7200z + 1728$\\
$1.72803$ & $-0.96230$ & $7z^{10} - 105z^9 + 726z^8 - 3068z^7 + 8928z^6 - 19212z^5$\\
& & $+\;31888z^4 - 40908z^3 + 38376z^2 - 23085z + 6561$\\
$0.79189$ & $-0.61066$ & $63z^{10} - 672z^9 + 3181z^8 - 8793z^7 + 15717z^6 - 18991z^5$\\
& & $+\;15717z^4 - 8793z^3 + 3181z^2 - 672z + 63$\\
$1.56449$ & $0.02423$ & $63z^{10} - 966z^9 + 7024z^8 - 32067z^7 + 102246z^6 - 238590z^5$ \\
& & $+\;412569z^4 - 520074z^3 + 454356z^2 - 246240z + 62208$
\end{tabular}
}
\end{center}
\end{table}

\begin{table}
\caption{Coordinates for the unit-distance embedding of $H_2$ depicted in \cref{uglyFigure}.}\label{h2table}
\begin{center}
{\def\arraystretch{1.15}
\begin{tabular}{rrl}
$x$ & $y$ & Minimal Polynomial for $z = x + iy$\\
\hline
$0.00000$ & $0.00000$ & $z$\\
$1.73205$ & $0.00000$ & $z^2 - 3$\\
$0.86603$ & $0.50000$ & $z^4 - z^2 + 1$\\
$0.86603$ & $-0.50000$ & $z^4 - z^2 + 1$\\
$0.73836$ & $-0.11215$ & $27z^{12} - 108z^{10} + 522z^8 - 1471z^6 + 3357z^4 - 2619z^2 + 675$\\
$0.33267$ & $-0.94304$ & $6561z^{24} - 39366z^{22} + 129033z^{20} - 438696z^{18} + 1030320z^{16}$\\
& & $-\;421713 z^{14} - 383282z^{12} - 421713 z^{10} + 1030320 z^8 $\\
& & $-\;438696 z^6 + 129033 z^4 - 39366 z^2 + 6561$\\
$0.98304$ & $-0.18342$ & $6561z^{24} - 39366z^{22} + 129033z^{20} - 438696z^{18} + 1030320z^{16}$\\
& & $-\;421713 z^{14} - 383282z^{12} - 421713 z^{10} + 1030320 z^8 $\\
& & $-\;438696 z^6 + 129033 z^4 - 39366 z^2 + 6561$\\
$-0.39973$ & $-0.91664$ & $729z^{24} - 1458z^{22} + 2673 z^{20} - 8964z^{18} + 8316z^{16}$\\
& & $+\;927z^{14} + 23110z^{12} + 927z^{10} + 8316z^8$\\
& & $-\;8964z^6 + 2673z^4 - 1458z^2 + 729$\\
$-0.59397$ & $0.80449$ & $729z^{24} - 1458z^{22} + 2673 z^{20} - 8964z^{18} + 8316z^{16}$\\
& & $+\;927z^{14} + 23110z^{12} + 927z^{10} + 8316z^8$\\
& & $-\;8964z^6 + 2673z^4 - 1458z^2 + 729$\\
\end{tabular}
}
\end{center}
\end{table}
\end{document}